\documentclass[10pt,reqno]{amsart}

\usepackage[a4paper,left=33mm,right=33mm,top=30mm,bottom=30mm,marginpar=25mm]{geometry}

\usepackage{amsmath}
\usepackage{amssymb}
\usepackage{amsthm}
\usepackage[latin1]{inputenc}
\usepackage{eurosym}
\usepackage{graphicx}
\usepackage{epsfig}
\usepackage{hyperref}
\usepackage{dsfont}
\usepackage[nocompress, space]{cite}
\usepackage{color}
\usepackage[displaymath,mathlines]{lineno}
\usepackage{float}
\usepackage{subcaption}
\usepackage{mathtools}

\allowdisplaybreaks

\usepackage{comment}

\usepackage{ifthen}

\makeindex

\renewcommand{\div}{\operatorname{div}}

\newcommand{\Rr}{{\mathbb{R}}}

\newcommand{\Nn}{{\mathbb{N}}}
\newcommand{\Zz}{{\mathbb{Z}}}

\newcommand{\Tt}{{\mathbb{T}}}

\newcommand{\Yd}{{\mathcal{Y}^d}}

\newcommand{\epsi}{\varepsilon}

\def\leq{\leqslant}
\def\geq{\geqslant}

\numberwithin{equation}{section}

\newtheoremstyle{thmlemcorr}{10pt}{10pt}{\itshape}{}{\bfseries}{.}{10pt}{{\thmname{#1}\thmnumber{
#2}\thmnote{ (#3)}}}
\newtheoremstyle{thmlemcorr*}{10pt}{10pt}{\itshape}{}{\bfseries}{.}\newline{{\thmname{#1}\thmnumber{
#2}\thmnote{ (#3)}}}
\newtheoremstyle{defi}{10pt}{10pt}{\itshape}{}{\bfseries}{.}{10pt}{{\thmname{#1}\thmnumber{
#2}\thmnote{ (#3)}}}
\newtheoremstyle{remexample}{10pt}{10pt}{}{}{\bfseries}{.}{10pt}{{\thmname{#1}\thmnumber{
#2}\thmnote{ (#3)}}}
\newtheoremstyle{ass}{10pt}{10pt}{}{}{\bfseries}{.}{10pt}{{\thmname{#1}\thmnumber{
A#2}\thmnote{ (#3)}}}

\theoremstyle{thmlemcorr}
\newtheorem{theorem}{Theorem}
\numberwithin{theorem}{section}
\newtheorem{lemma}[theorem]{Lemma}
\newtheorem{corollary}[theorem]{Corollary}
\newtheorem{proposition}[theorem]{Proposition}

\theoremstyle{thmlemcorr*}
\newtheorem{theorem*}{Theorem}
\newtheorem{lemma*}[theorem]{Lemma}
\newtheorem{corollary*}[theorem]{Corollary}
\newtheorem{proposition*}[theorem]{Proposition}
\newtheorem{problem*}[theorem]{Problem}
\newtheorem{conjecture*}[theorem]{Conjecture}

\theoremstyle{defi}
\newtheorem{definition}[theorem]{Definition}

\newtheorem{problem}{Problem}

\theoremstyle{remexample}
\newtheorem{remark}[theorem]{Remark}

\newtheorem{pro}[theorem]{Proposition}

\theoremstyle{ass}

\newtheorem*{notations*}{Notation}

\begin{document}

\title[Two-scale Homogenization of a stationary mean-field game]{Two-scale Homogenization of a stationary\\ mean-field game}

\author{Rita  Ferreira}
\address[R.  Ferreira]{
        King Abdullah University of Science and Technology (KAUST), CEMSE Division, Thuwal 23955-6900, Saudi Arabia.}
\email{rita.ferreira@kaust.edu.sa}
\author{Diogo  Gomes}
\address[D.  Gomes]{
        King Abdullah University of Science and Technology (KAUST), CEMSE Division, Thuwal 23955-6900, Saudi Arabia.}
\email{diogo.gomes@kaust.edu.sa}
\author{Xianjin Yang}
\address[X. Yang]{
        King Abdullah University of Science and Technology (KAUST), CEMSE Division, Thuwal 23955-6900, Saudi Arabia.}
\email{xianjin.yang@kaust.edu.sa}


\thanks{The authors were supported by King Abdullah University of Science and Technology (KAUST) baseline funds and KAUST OSR-CRG2017-3452.
}
\date{\today}

\begin{abstract}
In this paper, we characterize the asymptotic behavior of a first-order stationary mean-field game (MFG) with a logarithm coupling, a quadratic Hamiltonian, and a periodically oscillating potential. This study falls into the realm of the homogenization theory,
and our main tool is the two-scale convergence. Using this convergence,
we rigourously derive the two-scale homogenized and the homogenized
MFG problems, which encode the so-called macroscopic or effective
behavior of the original oscillating MFG. Moreover, we   prove existence and uniqueness of the solution to these limit  problems.  
\end{abstract}

\maketitle

\section{Introduction}

Mean-field games (MFGs), introduced by Lasry and Lions \cite{ll1, ll2} and by  Huang, Caines, and Malham\'{e} \cite{Caines1, Caines2}, model the behavior of rational and indistinguishable agents in a large population. When the number of elements in the population goes to infinity, the Nash equilibrium is characterized by a system of two partial differential equations (PDEs), a Hamilton--Jacobi (HJ) equation and a Fokker--Plank (FP) equation. The HJ equation determines the cost of a typical agent  and the FP  equation gives the evolution of the agents' distribution. 

Here, we characterize the asymptotic behavior of a stationary first-order MFG that has rapidly periodically oscillating terms. More precisely,
the MFG whose (homogenization) limit we
study is stated in the following problem.

\begin{problem}
\label{POCMFG}
 Let $\mathbb{T}^d$ be the  $d$-dimensional torus and let   $\mathcal{Y}^d$
be the set $ [0,1 )^d$. Let $P\in\mathbb{R}^d$, let $V: \mathbb{T}^d \times \Rr^d \rightarrow \mathbb{R}$ be a smooth function,  \(\mathcal{Y}^d\)-periodic in the second variable, and let $\epsilon>0$. 
 Find $ (u_\epsilon, m_\epsilon, \overline{H}_\epsilon )\in C^\infty (\mathbb{T}^d )\times C^\infty (\mathbb{T}^d )\times\mathbb{R}$,  
 with $m_\epsilon>0$,  solving
\begin{align}
\label{POCMFGEq}
\begin{cases} \frac{ |P+\nabla u_\epsilon (x ) |^2}{2}+V (x,\frac{x}{\epsilon} )=\ln (m_\epsilon (x ) )+\overline{H}_\epsilon (P ), \ &\text{in} \ \mathbb{T}^d, \\
-\div (m_\epsilon (x ) (P+\nabla u_\epsilon (x ) ) )=0,  \ &\text{in} \ \mathbb{T}^d,\\
\int_{\mathbb{T}^d} u_\epsilon (x )dx=0,\,  \int_{\mathbb{T}^d}m_\epsilon (x )dx=1.
\end{cases}
\end{align}
\end{problem}

Stationary MFGs as in Problem~\ref{POCMFG} describe a MFG  where agents seek to minimize an infinite horizon cost \cite{BFl, ll1, ll4} and  appear in the large-time behavior of time-dependent MFG  \cite{CLLP, cllp13, cardaliaguet2017long}. Here, with a rapidly periodically oscillating potential, $V$, Problem~\ref{POCMFG} models the behaviour of agents in the world with $\epsilon$-periodic heterogeneities. This potential determines the spatial preferences of each agent, whose parameters, $(x,\frac{x}{\epsilon})$, provide
the spacial information  of an agent at the macroscopic and the microscopic scales separately. For instance, in the traffic-flow problem with periodically, fast changing road conditions, $V$ may give the cost of a typical car according to the location, given by $x$, of the car on the road and to current road conditions, indexed by $\frac{x}{\epsilon}$. Another example would be agents moving on a grid of regularly spaced obstacles such as a forest or a minefield. 

We further mention that, in Problem~\ref{POCMFG},  \(P\) determines the preferred direction of motion,   $u_\epsilon: \mathbb{T}^d\rightarrow \mathbb{R}$ gives the cost of a typical agent at the state  $x\in\mathbb{T}^d$, $m_\epsilon: \mathbb{T}^d\rightarrow \mathbb{R}$ is a probability measure determining the agents' distribution, and $\overline{H}_\epsilon$ is a real number depending on $P$ that is required to ensure
that the normalization condition   $  \int_{\mathbb{T}^d}m_\epsilon (x )dx=1$ is satisfied.

Our goal is to study the asymptotic behavior of the solutions to Problem~\ref{POCMFG} as $\epsilon\rightarrow 0$, where, we
recall, \(\epsilon\) represents the length-scale of the heterogeneities
characterizing the MFG under study.
  This analysis falls into the realm of homogenization theory, aimed at describing  the macroscopic or effective behavior of a microscopically heterogeneous system.
A typical homogenization problem involves two scales; a microscale
associated  with  the size of the heterogeneities of the system, and a macroscale associated with the size  of the state-space of the system.
The goal is to replace  equations with microscales, which are hard to resolve numerically, by {\em averaged}  macroscopic equations that are easier to solve and whose overall properties approximate
well those of the initial oscillating  equations. The problem comprising the  macroscopic
equations is the {\em homogenized problem} and encodes the
macroscopic or effective behavior of the initial microscopically heterogeneous
problem.  We refer to
\cite{CiDo99} for
 a comprehensive introduction to the theory of homogenization and for an overview of the different homogenization methods to
derive the homogenized problem. 

To the best
of our knowledge, apart from the works \cite{cacace2018ergodic,
MR3537880}, little is known on the characterization of the effective
behavior of MFGs with rapidly periodically oscillating terms.
In \cite{cacace2018ergodic}, the authors consider a second-order,
time-dependent  MFG  with a local coupling and a quadratic Hamiltonian. Using (formal) asymptotic expansion techniques, they derive and
study  the associated 
homogenized  problem. Moreover, for a particular class
of  initial-terminal conditions, they rigourously justify their asymptotic expansion procedure. In \cite{MR3537880}, the authors provide
some qualitative descriptions and some numerical results 
regarding the MFG introduced in \cite{cacace2018ergodic}.
 
Here, we consider a first-order,
stationary  MFG  with a logarithmic coupling and a quadratic Hamiltonian with rapidly periodically oscillating terms. We study the effective behavior of this MFG using the two-scale convergence method. The notion of
two-scale convergence was first introduced by Nguetseng \cite{nguetseng1989general}, and further developed by Allaire \cite{allaire1992homogenization},  to provide a  rigorous justification of formal asymptotic expansions used in periodic homogenization problems. 

Besides providing a rigorous derivation of the effective behavior,   the
two-scale convergence method takes full advantage of the periodic microscopic properties of the system, enabling the explicit characterization of its local behavior encoded in
the so-called two-scale homogenized problem. This problem accounts
for the asymptotic behavior of the original problem at both macroscopic and microscopic levels, through the two space variables $x$ (the macroscopic one) and $y$ (the microscopic one), and through  two unknowns $u$ and $u_1$ (see Problems~\ref{TwoScaleHomogenized}
and \ref{TwoScaleMinimization} below). Then,
 using an average process with respect to the microscopic variable $y$ of the two-scale homogenized problem, one obtains the homogenized problem (see Problems~\ref{Homogenized} and \ref{TheLimitProblem} below). Typically, this problem  only
involves the macroscopic space variable $x$,   has \(u_0=\int_{\Yd}
u(\cdot,y) dy\) as solution, and its coefficients are determined
by an auxiliary problem, called the cell problem (see Problems~\ref{CellProbEulerLagrangeProb}
and \ref{TheCellProblem} below). 

As  pointed out in
\cite[Remark~2.4]{allaire1992homogenization}, in spite of doubling the variables ($x$ and $y$)
and unknowns ($u$ and $u_1$), in most cases, the two-scale
homogenized problem is of the same form of the original one, sharing the same existence and uniqueness properties. In contrast,
 in several cases, the homogenized problem  has
complicated forms by involving, for instance,  integro-differential operators and non-explicit equations.
Consequently, the homogenized problem may belong to a class for which an existence and uniqueness theory is not  available. Hence,
both problems, the two-scale homogenized and the homogenized, have their own advantages and one should not be discarded in favor of the other.

To use the two-scale
convergence method 
to study the asymptotic behavior of Problem~\ref{POCMFG} as \(\epsi\to0\),
we take advantage of its variational structure, revealed in the
next problem.  
\begin{problem}
\label{VariationalProblem}
Under the same assumptions of Problem~\ref{POCMFG}, find $u_\epsilon\in C^\infty (\mathbb{T}^d )$ satisfying \(\int_{\mathbb{T}^d} u_\epsilon (x )dx=0\) and  
\begin{equation}
\label{VariationalFormula}
 I_\epsilon [u_\epsilon ]=\inf_{\overset{u\in {C}^1 (\mathbb{T}^d )}{\int_{\mathbb{T}^d}u (x )dx=0}}I_\epsilon [u ], 
\end{equation} 
where 
\begin{equation}
\label{DefIepsilon}
I_\epsilon [u ]=\int_{\mathbb{T}^d}e^{ \frac{ |P+\nabla u (x ) |^2}{2}+V (x,\frac{x}{\epsilon} )}dx.
\end{equation}
\end{problem}

We stress that Problems~\ref{POCMFG} and \ref{VariationalProblem} are equivalent and admit a unique smooth solution. In fact, 
by  \cite[Theorem 5.2]{evans2003some}, \eqref{POCMFGEq} is the Euler-Lagrange equation of \eqref{DefIepsilon} and there exists a unique smooth solution, $ (u_\epsilon, m_\epsilon, \overline{H}_\epsilon )$, to Problem~\ref{POCMFG}. Thus, $u_\epsilon$ is the smooth minimizer for Problem~\ref{VariationalProblem}. Conversely, if $u_\epsilon$ is the smooth solution to Problem~\ref{VariationalProblem}, defining
\begin{equation}
\label{DefHEpsilon}
\overline{H}_\epsilon (P )=\ln I_\epsilon [u_\epsilon ]
\end{equation}
and 
\begin{equation}
\label{eq:defmepsi}
\begin{aligned}
m_\epsilon=e^{\frac{ |P+\nabla u_\epsilon (x ) |^2}{2}+V (x,\frac{x}{\epsilon} )-\overline{H}_\epsilon (P )},
\end{aligned}
\end{equation}
we conclude that $ (u_\epsilon, m_\epsilon, \overline{H}_\epsilon )$ solves Problem~\ref{POCMFG}.

We observe that in contrast with the majority of the variational
homogenization problems in the literature, the integrand in \eqref{DefIepsilon} does not admit a polynomial upper bound with respect to the gradient
variable. This fact prevents us from mimic the arguments in,
for instance, \cite{allaire1992homogenization} to derive the two-scale homogenized
functional. Problems with non-standard growth
conditions were studied in
                \cite{CaDe02} using a \(\Gamma\)-convergence
                approach. However, here we use a distinct
approach based on two-scale convergence that, as we mentioned
before, is known to take
full advantage of the periodic structure of the problem, enabling
us to provide an explicit characterization of both the 
two-scale homogenized and the \textit{usual} homogenized
problems.

In what follows, we use the subscript \(\#\) to refer to functions
defined on \(\Rr^d\) that are \(\Yd\)-periodic; for instance,
 \(C^\infty_\#(\Yd) = \{u\in C^\infty(\Rr^d):\, u \text{ is \(\Yd\)-periodic}\)\}. Moreover, given \(p\in(1,+\infty)\), \(W^{1,p}_\#(\Yd)\) denotes
the closure of \(C^\infty_\#(\Yd)\) with respect to the \(W^{1,p}(\Yd)\)-norm.

Next, we introduce  the two-scale homogenized problem that,
as stated in Theorem~\ref{MainTheorem},   provides a characterization
of the
effective behavior, at both macroscopic
and microscopic levels, of Problem~\ref{POCMFG} 
as $\epsilon\rightarrow 0$.

\begin{problem}[Two-scale homogenized problem]
        \label{TwoScaleHomogenized}
        Under the same assumptions of Problem~\ref{POCMFG} and
for some  $\alpha\in(0,1)$, find
${u}_0\in  C^\infty (\mathbb{T}^d )$ with \(\int_{\Tt^d}
 u_0dx=0\),  ${u}_1 \in C^\infty (\mathbb{T}^d;C^{2,\alpha}_\#
(\mathcal{Y}^d )/\mathbb{R} )$,    $m\in  C^\infty (\mathbb{T}^d;{{C^{1,\alpha}_\#
(
\mathcal{Y}^d )}} )$ with \(\int_{\mathbb{T}^d}
\int_\Yd m(x,y )dydx=1\), and \(\overline{H}\in\Rr\) satisfying
\begin{align}
        \label{TheoremTwoScaleSystemEq}
        \begin{cases}
        \frac{ |P+\nabla u_0 (x )+\nabla_y u_1 (x,y ) |^2}{2}+V
(x,y )=\ln (m (x,y ) )+\overline{H}(P),\ &\text{in} \ \mathbb{T}^d\times\mathcal{Y}^d,\\
        -\div_x (\int_{\mathcal{Y}^d}m (x,y ) (P+\nabla u_0 (x
)+\nabla_yu_1 (x,y ) )dy )=0,\ &\text{in} \ \mathbb{T}^d\times\mathcal{Y}^d,\\
        -\div_y (m (x,y ) (P+\nabla u_0 (x )+\nabla_yu_1 (x,y
) ) )=0, \ &\text{in} \ \mathbb{T}^d\times\mathcal{Y}^d.
        \end{cases}
        \end{align}
        \end{problem}

Next, we introduce  
the \textit{usual}
homogenized problem, together with the associated cell problem
that, as stated in Theorem~\ref{MainTheorem},  characterize
the
effective behavior of Problem~\ref{POCMFG} 
as $\epsilon\rightarrow 0$. 

\begin{problem}[Cell problem]
\label{CellProbEulerLagrangeProb}
Suppose that the assumptions in Problem~\ref{POCMFG} hold. For
some  $\alpha\in(0,1)$ and for each $x\in\mathbb{T}^d$ and $\Lambda\in\mathbb{R}^d$, find  $\widetilde{w}
\in C^{2,\alpha} _\#(\mathcal{Y}^d )/\mathbb{R}$, $\widetilde{m}
\in C^{1,\alpha} _\#(\mathcal{Y}^d )$, 
and $\widetilde{H}\in\mathbb{R}$, depending on $x$ and $\Lambda$,
such that  $ (\widetilde{w},\widetilde{m},\widetilde{H} )$ solves
\begin{align}
\label{CellProbEulerLagrangeEq}
\begin{split}
\begin{cases}
\frac{ |\Lambda+\nabla_y \widetilde{w} (x,\Lambda,y ) |^2}{2}+V (x,y
)=\ln \widetilde{m} (x,\Lambda,y )+\widetilde{H} (x,\Lambda ), \ &\text{in}\
\mathcal{Y}^d, \\
-\div_y \big(\widetilde{m} (x,\Lambda,y ) (\Lambda+\nabla_y \widetilde{w}
(x,\Lambda,y ) )\big )=0, \ &\text{in}\ \mathcal{Y}^d, \\
\int_{\mathcal{Y}^d} \widetilde{m} (x,\Lambda,y )dy=1.
\end{cases}
\end{split}
\end{align} 
\end{problem}

\begin{problem}[Homogenized problem]
\label{Homogenized}
Suppose that the assumptions in Problem~\ref{POCMFG} hold and
that   $ (\widetilde{w},\widetilde{m},\widetilde{H} )$ solves Problem~\ref{CellProbEulerLagrangeProb}. Find  
${u}_0\in  C^\infty (\mathbb{T}^d )$ with \(\int_{\Tt^d}
 u_0dx=0\),      $m_0\in  C^\infty (\mathbb{T}^d )$ with \(m_0>0\), and \(\overline{H}\in\Rr\) satisfying
\begin{align}
        \label{HomogenizedProblemEq}
        \begin{cases}
        \widetilde{H} (x,P+\nabla u_0 (x ) )=\ln (m_0 (x) )+\overline{H}(P),\ &\text{in} \ \mathbb{T}^d,\\
        -\div\big(m_0 (x )D_\Lambda\widetilde{H} (x, P+\nabla u_0 (x
)\big)=0,\ &\text{in} \ \mathbb{T}^d,\\
         \int_{\mathbb{T}^d}
m_0 dx=1. 
        \end{cases}
        \end{align} 
\end{problem}

The proof of Theorem~\ref{MainTheorem} is strongly hinged on
a variational analysis based on the equivalence between Problems~\ref{POCMFG} and \ref{VariationalProblem}. For this reason, we  now introduce the variational
formulation of Problems~\ref{TwoScaleHomogenized}, \ref{CellProbEulerLagrangeProb},
and
\ref{Homogenized}.
We start with the variational formulation of the two-scale homogenized  problem, Problem~\ref{TwoScaleHomogenized}.

\begin{problem}[Variational two-scale homogenized  problem]
        \label{TwoScaleMinimization}
        Fix  \(p\in(1,+\infty)\). Under the same assumptions of Problem~\ref{POCMFG} and
for some  $\alpha\in(0,1)$, find
${u}_0\in  C^\infty (\mathbb{T}^d )$ with \(\int_{\Tt^d}
 u_0dx=0\) and ${u}_1 \in C^\infty (\mathbb{T}^d;C^{2,\alpha}_\#
(\mathcal{Y}^d )/\mathbb{R} )$ satisfying
        \begin{align*}
        \overline{I} [{u}_0,{u}_1 ]=\inf_{\overset{u\in W^{1,p}
(\mathbb{T}^d ), \int_{\Tt^d}
 udx=0}{w\in L^{p} (\mathbb{T}^d;W^{1,p}_\# (\mathcal{Y}^d )/\mathbb{R}
)}}\overline{I} [u,w ],
        \end{align*}
        where $\overline{I}: {W}^{1,p}  (\mathbb{T}^d )\times{L}^p
(\mathbb{T}^d;{W}^{1,p}_\#(\mathcal{Y}^d )/\mathbb{R} )\rightarrow
\mathbb{R}$ is defined, for all $ (u,w )\in {W}^{1,p}  (\mathbb{T}^d
)\times{L}^p (\mathbb{T}^d;{W}^{1,p}_\# (\mathcal{Y}^d )/\mathbb{R}
)$, by
        \begin{align*}
        \overline{I} [u,w ]=\int_{\mathbb{T}^d}\int_{\mathcal{Y}^d}e^{
\frac{ |P+\nabla u (x )+\nabla_y w (x,y ) |^2}{2}+V (x,y )}dydx.
        \end{align*}
\end{problem} 

Finally, we introduce the variational
formulation of the homogenized  and its    associated cell problems, Problems~\ref{CellProbEulerLagrangeProb}
and
\ref{Homogenized}.

\begin{problem}[Variational cell problem]
        \label{TheCellProblem}
        Fix  \(p\in(1,+\infty)\) and suppose that the assumptions in Problem~\ref{POCMFG}
hold. For
some  $\alpha\in(0,1)$ and for each $x\in\mathbb{T}^d$
and $\Lambda\in\mathbb{R}^d$,
find  $\widetilde{w}
\in C^{2,\alpha} _\#(\mathcal{Y}^d )/\mathbb{R}$, depending on
\(x\) and \(\Lambda\), 
satisfying
        \begin{equation}
        \label{MinCellProblem}
        \widetilde{I} [x,\Lambda;\widetilde{w} ]=\inf_{{w}\in{W}^{1,p}_\#
(\mathcal{Y}^d )/\mathbb{R}  } \widetilde{I} [x,\Lambda;{w} ],
        \end{equation} 
        where $\widetilde{I}[x,\Lambda;\cdot]: {W}^{1,p}_\# (\mathcal{Y}^d
)/\mathbb{R} \rightarrow \mathbb{R}$ is defined, for all ${w}\in
{W}^{1,p}_\# (\mathcal{Y}^d
)/\mathbb{R}$,
by
        \begin{equation}
        \label{DefwidetildeI}
        \widetilde{I} [x,\Lambda;{w} ]=\int_{\mathcal{Y}^d}e^{\frac{
|\Lambda+\nabla{w}
(y ) |^2}{2}+V (x,y )}dy.
        \end{equation}        
\end{problem}

\begin{problem}[Variational homogenized problem]
        \label{TheLimitProblem}
        Fix  \(p\in(1,+\infty)\) and assume that the assumptions in Problem~\ref{POCMFG} hold.
Let $\widetilde{H}:\Tt^d\times \Rr^d\to\Rr$ be  defined,
for each $x\in\mathbb{T}^d$
and $\Lambda\in\mathbb{R}^d$,  by
\begin{equation}
\label{WidetildeH}
\widetilde{H} (x,\Lambda )=\ln \widetilde{I} [x,\Lambda;\widetilde{w}
], 
\end{equation} 
where  $\widetilde{w}$ solves Problem~\ref{TheCellProblem}. 
 Find $u_0\in
C^\infty (\mathbb{T}^d )$ satisfying \(\int_{\Tt^d}
 u_0dx=0\) and 
        \begin{equation}
        \label{MinLimitProblem}
        \widehat{I} [u_0 ]=\inf_{{u}\in {W}^{1,p}  (\mathbb{T}^d
), \int_{\Tt^d}
 udx=0}\widehat{I} [{u} ],
        \end{equation}
        where $\widehat{I}: {W}^{1,p} (\mathbb{T}^d )\rightarrow
\mathbb{R}$ is defined, for all ${u}\in {W}^{1,p} (\mathbb{T}^d
)$, by
        \begin{equation*}
        \widehat{I} [{u} ]=\int_{\mathbb{T}^d}e^{\widetilde{H}
(x,P+\nabla {u} (x ) )}dx.
        \end{equation*}
\end{problem}

\begin{remark}
\label{rmk:p7p4qui}
As what we stated for the relation between Problem~\ref{POCMFG} and Problem~\ref{VariationalProblem}, Problem~\ref{CellProbEulerLagrangeProb} is equivalent to Problem~\ref{TheCellProblem}. More precisely, \eqref{TheoremTwoScaleSystemEq} is the Euler-Lagrange equation to \eqref{DefwidetildeI}. Moreover, if \((\widetilde{w}, \widetilde{m}, \widetilde{H})\) solves Problem~\ref{CellProbEulerLagrangeProb}, then \(\widetilde{w}\) is the minimizer for Problem~\ref{TheCellProblem}. Conversely, if $\widetilde{w}$ is a solution to Problem~\ref{TheCellProblem}, defining $\widetilde{H}$ as in \eqref{WidetildeH} and $\widetilde{m}=e^{\frac{\left|\Lambda+\nabla_y\widetilde{w}\right|^2}{2}+V-\widetilde{H}}$, we see that \((\widetilde{w},\widetilde{m},\widetilde{H})\) solve Problem~\ref{CellProbEulerLagrangeProb}. Similar arguments hold for Problem~\ref{Homogenized} and Problem~\ref{TheLimitProblem} and for Problem~\ref{TwoScaleHomogenized} and Problem~\ref{TwoScaleMinimization}. 
\end{remark}

Our main result is stated in the following theorem. We refer to Section \ref{TwoScaleConvergence} for the definition and some properties of the notion of two-scale convergence. 
\begin{theorem}
        \label{MainTheorem}
        Let $ (u_\epsilon, m_\epsilon, \overline{H}_\epsilon )\in C^\infty (\mathbb{T}^d )\times C^\infty (\mathbb{T}^d )\times\mathbb{R}$ solve Problem~\ref{POCMFG}. If $d>1$, we assume further that $V$ is separable in $y$; that is, there exist smooth functions,  $V_i: \mathbb{T}^d\times\Rr\rightarrow\mathbb{R}$, where $1\leq i\leq d$, such that for all $x\in\mathbb{T}^d$ and $y\in\Rr^d$,  $y= (y_1,\dots,y_i,\dots,y_d )$, we have 
        \begin{equation}
        \label{SeparableV}
        V (x,y )=\sum_{i=1}^{d}V_i (x, y_i ).
        \end{equation}
        Then, there exists   $\alpha\in(0,1)$  and    there exist  $u_0\in C^\infty (\mathbb{T}^d )$ with \(\int_{\mathbb{T}^d} u_0 (x )dx=0\), $u_1\in C^\infty (\mathbb{T}^d;{{C^{2,\alpha}_\# (\mathcal{Y}^d )}}/\mathbb{R} )$,  $m\in  C^\infty (\mathbb{T}^d;{{C^{1,\alpha}_\# (\mathcal{Y}^d )}} )$ with \(\int_{\mathbb{T}^d}
\int_\Yd m(x,y )dxdy=1\),  and $\overline{H}(P)\in\Rr$  such that
$(\nabla u_\epsilon)_\epsilon$ weakly two-scale converges to $\nabla u_0+\nabla_y u_1$ in \( {L}^p (\mathbb{T}^d\times\mathcal{Y}^d )\) for all
\(p\in[1,\infty)\), $(m_\epsilon)_\epsilon$ weakly two-scale converges to
$m$
in \( {L}^1 (\mathbb{T}^d\times\mathcal{Y}^d )\), and   $(\overline{H}_\epsilon(P))_\epsilon$ converges to $\overline{H}$(P)
in \(\Rr\). Moreover,   $ (u_0, u_1, m, \overline{H} )$ is the {(unique)} solution to Problem~\ref{TwoScaleHomogenized}, and   $
(u_0, u_1 )$ is the {(unique)}
solution to Problem~\ref{TwoScaleMinimization}. 

Furthermore, setting \(m_0=\int_\Yd m(\cdot,y )dy\), we have that
\((u_\epsilon)_\epsilon$ weakly converges to \(u_0 \) in \(W^{1,p}(\Tt^d)\) for all
\(p\in[1,\infty)\), \((m_\epsilon
 )_\epsilon$ weakly converges to \( m_0 \) in \(L^1(\Tt^d)\),  $(u_0,  m_0,
\overline{H} )$ is the {(unique)}
solution to Problem~\ref{Homogenized}, and  $ u_0$ is the
{(unique)}
solution to Problem~\ref{TheLimitProblem}.
In addition, %
        \begin{equation}
        \label{OneHomoToTwoHomo}
        \widehat{I} [u_0 ]=\overline{I} [u_0,u_1 ]= \lim\limits_{\epsilon\rightarrow
0}I_\epsilon [u_\epsilon ]
        \end{equation}
and \(\overline{H}(P)=\ln \overline{I} [u_0,u_1 ]=\ln\widehat{I} [u_0 ]\).          
\end{theorem}

\begin{remark}\label{rmk:ongrady}
The term \(\nabla_y u_1\) in the two-scale limit of \((\nabla u_\epsilon)_\epsilon\) in the previous theorem may be regarded
as the gradient limit at the microscale \(y\). This \textit{extra information} on the oscillatory behavior of a bounded sequence
in \(W^{1,p}\) is one of the key features of the two-scale convergence
(see Proposition~\ref{GradientConvergence}).
\end{remark}

Theorem~\ref{MainTheorem} shows that 
Problems~\ref{TwoScaleHomogenized}--\ref{TheLimitProblem} provide the effective behavior of Problems~\ref{POCMFG} and \ref{VariationalProblem}.
Before proving Theorem~\ref{MainTheorem}, we illustrate
how the asymptotic expansion method heuristically leads to
the two-scale homogenized and the homogenized problems. Then,
in Section \ref{TwoScaleConvergence}, we recall the definition and some
properties of the notion of two-scale convergence, which is our main tool to prove Theorem~\ref{MainTheorem}. In 
Section~\ref{SectionOriginalMFGAnalysis}, we establish uniform
bounds in \(\epsilon\) of the solutions to Problems~\ref{POCMFG} and \ref{VariationalProblem} that yield the compactness properties
stated in Theorem~\ref{MainTheorem}. Next, in  Section~\ref{TwoScaleHomogeInOneDim},
we derive and explicitly solve the two-scale homogenized problem
as stated in Theorem~\ref{MainTheorem}  by  explicitly solving  Problem~\ref{POCMFG}
using the \textit{current method} introduced in \cite{Gomes2016b}.
In the one-dimensional case, the arguments in 
Section~\ref{TwoScaleHomogeInOneDim} constitute an alternative
to those in Section~\ref{TwoScaleHomogeInHigherDim}, where, using
the lower semi-continuity of convex functionals with respect to
two-scale convergence and the regularity of the minimizer to
Problem~\ref{TwoScaleMinimization}, we prove Theorem~\ref{MainTheorem} in any  dimension. We establish the existence, uniqueness, and regularity
of solutions to   
Problem~\ref{TwoScaleMinimization}  in 
Section~\ref{TheHomogenizedProb}. To this end, we first use the continuation method to prove
the existence, uniqueness, and regularity for the solution to
Problem~\ref{CellProbEulerLagrangeProb}. Thus, equivalently, Problem~\ref{TheCellProblem} admits a unique  solution. For Problems~\ref{Homogenized}\  and \ref{TheLimitProblem}, the well-posedness follows directly by Evans'  work \cite{evans2003some} after checking that $\widetilde{H}$ satisfies the assumptions in \cite{evans2003some}. We then conclude that Problems~\ref{TwoScaleHomogenized} and \ref{TwoScaleMinimization} admit a unique solution.

\begin{notations*}
        Throughout this manuscript,  $\epsilon$ stands for a small parameter taking values on a sequence of positive numbers converging to zero. Besides, given \(p\in(1,+\infty)\),  $p'$  represents the real number satisfying $1/p+1/{p'}=1$. We denote the transpose of a vector $v$ by $v^T$. For simplicity, we use the Einstein notation; that is, when an index  appears twice in a single term, it means that we sum that term over all the values of the index. For example, we write $\sum_{i=1}^d a_{ij}v_i$ as $a_{ij}v_i$ for short, where $a_{ij}$ and $v_i$ are real values indexed by $1\leq i,j\leq d$. For two Banach spaces, $X$ and $Y$, the set $L^p (X;Y )$ is the $L^p$-space on $X$ with values in $Y$. Similarly, $C^\infty (X;Y )$ denotes the space of smooth functions on $X$ with values in $Y$. We denote $X/\mathbb{R}$ the quotient space consisting of equivalent classes and each class contains elements in $X$, which only differ from each other by a real number. We denote by
 \(C^\infty_\#(\Yd)\) the space of \(C^\infty\) functions on \( \Rr^d\) that are \(\Yd\)-periodic. The spaces \(C^{k,\alpha}_\#(\Yd)\),
with \(k\in\Nn\) and \(\alpha\in(0,1)\),
are defined analogously.
Moreover, given \(p\in(1,+\infty)\), \(W^{1,p}_\#(\Yd)\) denotes
the closure of \(C^\infty_\#(\Yd)\) with respect to the \(W^{1,p}(\Yd)\)-norm.
 Finally,  $ |A |$ stands for the Lebesgue measure of the set $A$. 
\end{notations*}

\section{Asymptotic Expansions}
\label{HeuristicMethods}
In this section, we review the asymptotic expansion method and find its relation with two-scale convergence. 
The key step of the asymptotic expansion method is to introduce an ansatz for the solution of \eqref{POCMFGEq} and expand \eqref{POCMFGEq} in Taylor series. Then, 
by matching asymptotic terms in the resulting equations, we find the homogenized system. 
 
Here, we postulate the following forms for $u_\epsilon$ and $m_\epsilon$:
\begin{equation}
\label{Heuristicumep}
\begin{cases}
u_\epsilon (x )=\widetilde{u}_0 (x )+\epsilon \widetilde{u}_1 (x, \frac{x}{\epsilon} ),\\
m_\epsilon (x )=\widetilde{m}_0 (x ) (\widetilde{m}_1 (x,\frac{x}{\epsilon} )+\epsilon m_2 (x,\frac{x}{\epsilon} ) )
\end{cases}
\end{equation}
and use \eqref{Heuristicumep} in \eqref{POCMFGEq}.

At order $\epsilon^0$ in the first equation, we get
\begin{equation}
\label{uorderep0}
\frac{ |P+\nabla \widetilde{u}_0 (x )+\nabla_y\widetilde{u}_1 (x,\frac{x}{\epsilon} ) |^2}{2}+V\Big (x,\frac{x}{\epsilon} \Big)=\ln \widetilde{m}_0 (x )+\ln \widetilde{m}_1 \Big(x,\frac{x}{\epsilon} \Big)+\overline{H}.
\end{equation}
Denote $\widetilde{\Lambda}=P+\nabla \widetilde{u}_0 (x )$, $y=\frac{x}{\epsilon}$, and 
\begin{equation}
\label{Htideeq}
\widehat{H} (x,\widetilde{\Lambda} )=\ln \widetilde{m}_0 (x )+\overline{H}. 
\end{equation}
Then, \eqref{uorderep0} becomes
\begin{equation}
\label{CellProbFirsteq}
\frac{ |\widetilde{\Lambda}+\nabla_y \widetilde{u}_1 (x,y ) |^2}{2}+V (x,y )=\ln \widetilde{m}_1 (x,y )+\widehat{H} (x,\widetilde{\Lambda} ).
\end{equation}
The terms of order $\epsilon^0$ in the expansion of the second equation of \eqref{POCMFGEq} give
\begin{equation}
\label{morderep0}
-\div_x \big(\widetilde{m}_0 (x )\widetilde{m}_1 (x,y ) (\widetilde{\Lambda}+\nabla_y\widetilde{u}_1 (x,y ) )\big )=0.
\end{equation}
Integrating \eqref{morderep0} over $y$, we obtain 
\begin{equation}
\label{LimitSystemSecondEq}
-\div_x \bigg(\widetilde{m}_0 (x )\int_{\mathcal{Y}^d} (\widetilde{m}_1 (x,y ) (\widetilde{\Lambda}+\nabla_y\widetilde{u}_1 (x,y ) ) )dy \bigg)=0.
\end{equation}
Meanwhile, at order $\epsilon^{-1}$ in the expansion of \eqref{POCMFGEq}, we get
\begin{equation*}
\label{mordereminus1}
-\div_y \big(\widetilde{m}_0 (x )\widetilde{m}_1 (x,y ) (\widetilde{\Lambda}+\nabla_y\widetilde{u}_1 (x,y ) )\big )=0.
\end{equation*}
Since $\widetilde{m}_0>0$, we have 
\begin{equation}
\label{CellProblemeq2}
-\div_y \big(\widetilde{m}_1 (x,y ) (\widetilde{\Lambda}+\nabla_y\widetilde{u}_1 (x,y ) )\big )=0.
\end{equation}
Thus, considering \eqref{Htideeq} and \eqref{LimitSystemSecondEq}, the expected homogenized system of \eqref{POCMFGEq} is 
\begin{align}
\label{HeuristicLimitSystem}
\begin{cases}
\widehat{H} (x,P+\nabla \widetilde{u}_0 (x ) )=\ln \widetilde{m}_0 (x ) + \overline{H},\\
-\div\big (\widetilde{m}_0 (x )\widetilde{b} (x,P+\nabla \widetilde{u}_0 (x ) )\big )=0,
\end{cases}
\end{align}
where 
\begin{equation}
\label{defwidetideb}
\widetilde{b} (x,\widetilde{\Lambda} )=\int_{\mathcal{Y}^d}\widetilde{m}_1 (x,y ) (\widetilde{\Lambda}+\nabla_y\widetilde{u}_1 (x,y ) )dy
\end{equation}
and
$ (\widehat{H}, \widetilde{u}_1, \widetilde{m}_1 )$ solves \eqref{CellProbFirsteq} and \eqref{CellProblemeq2}, called the cell system; that is, for fixed $x\in\mathbb{T}^d$ and $\Lambda\in\mathbb{R}^d$, $ (\widehat{H}, \widetilde{u}_1, \widetilde{m}_1 )$ solves 
\begin{align}
\label{HeuristicCellProblem}
\begin{cases}
\frac{ |\widetilde{\Lambda}+\nabla_y \widetilde{u}_1 (x,y ) |^2}{2}+V (x,y )=\ln \widetilde{m}_1 (x,y )+\widehat{H} (x,\widetilde{\Lambda} ),\\
-\div_y \big(\widetilde{m}_1 (x,y ) (\widetilde{\Lambda}+\nabla_y \widetilde{u}_1 (x,y ) )\big )=0.
\end{cases}
\end{align}
Finally, we differentiate the first equation of \eqref{HeuristicCellProblem} with respect to $\widetilde{\Lambda}$ and get
\begin{equation*}
\left(\widetilde{\Lambda}+\nabla_y\widetilde{u}_1\right)+\left(\widetilde{\Lambda}+\nabla_y\widetilde{u}_1\right)^T\nabla_y\left(\nabla_{\widetilde{\Lambda}}\widetilde{u}_1\right)=\frac{\nabla_{\widetilde{\Lambda}}\widetilde{m}_1}{\widetilde{m}_1}+\nabla_{\widetilde{\Lambda}}\widehat{H}.
\end{equation*}
Multiplying both sides of the prior equation by $\widetilde{m}_1$ and integrating the resulting equation over $\mathcal{Y}^d$, we obtain
\begin{equation}
\label{HeuristicGradHhat}
\begin{aligned}
&\int_{\mathcal{Y}^d}\widetilde{m}_1\left(\widetilde{\Lambda}+\nabla_y\widetilde{u}_1\right) dy+\int_{\mathcal{Y}^d}\widetilde{m}_1\left(\widetilde{\Lambda}+\nabla_y\widetilde{u}_1\right)^T\nabla_y\left(\nabla_{\widetilde{\Lambda}}\widetilde{u}_1\right) dy\\
&=\int_{\mathcal{Y}^d}\nabla_{\widetilde{\Lambda}}\widetilde{m}_1 dy + \int_{\mathcal{Y}^d}\widetilde{m}_1 \nabla_{\widetilde{\Lambda}}\widehat{H}dy. 
\end{aligned}
\end{equation}
Using integration by parts and the second equation of \eqref{HeuristicCellProblem}, we get
\begin{equation}
\label{HeuristicGradHhatItn1}
\int_{\mathcal{Y}^d}\widetilde{m}_1\left(\widetilde{\Lambda}+\nabla_y\widetilde{u}_1\right)^T\nabla_y\left(\nabla_{\widetilde{\Lambda}}\widetilde{u}_1\right) dy=0.
\end{equation}
Besides, assuming that $\int_{\mathcal{Y}^d}\widetilde{m}_1 dy=1$, we have
\begin{equation}
\label{HeuristicGradHhatItn2}
\int_{\mathcal{Y}^d}\nabla_{\widetilde{\Lambda}}\widetilde{m}_1 dy = 0 \ \text{and} \ \int_{\mathcal{Y}^d}\widetilde{m}_1\nabla_{\widetilde{\Lambda}}\widehat{H}dy=\nabla_{\widetilde{\Lambda}}\widehat{H}. 
\end{equation}
Combining \eqref{HeuristicGradHhat}, \eqref{HeuristicGradHhatItn1}, and \eqref{HeuristicGradHhatItn2}, we conclude that
\begin{equation*}
\nabla_{\widetilde{\Lambda}}\widehat{H}=\int_{\mathcal{Y}^d}\widetilde{m}_1\left(\widetilde{\Lambda}+\nabla_y\widetilde{u}_1\right)dy,
\end{equation*}
which implies that $\nabla_{\widetilde{\Lambda}}\widehat{H}=\widetilde{b}$ according to the definition of $\widetilde{b}$ in \eqref{defwidetideb}. 

Therefore, the homogenized system in \eqref{HeuristicLimitSystem} found by the asymptotic method is consistent with  \eqref{HomogenizedProblemEq}  in  Problem~\ref{Homogenized} and the cell system \eqref{HeuristicCellProblem} corresponds to \eqref{CellProbEulerLagrangeEq} in Problem~\ref{CellProbEulerLagrangeProb}.

\section{Two-scale convergence}
\label{TwoScaleConvergence}
Because functions on $\mathbb{T}^d$ can be viewed as  $\Zz^d$-periodic function on $\mathbb{R}^d$, 
 the results on two-scale convergence for functions on bounded domains can be easily adapted to those on $\mathbb{T}^d$. Here, we review some essential results from  \cite{
 zhikov2004two, visintin2006towards,
 LuNgWa02}. Throughout this section, \(p\in (1,+\infty)\) and
\(p'=\frac{p}{p-1}\). 
\begin{definition}
        \label{TwoScaleConvDef}
     Let \(q\in[1,+\infty]\).   We say that a sequence, $ (w_\epsilon )_\epsilon$, in ${L}^q (\mathbb{T}^d )$ weakly two-scale converges to a function  $w\in {L}^q (\mathbb{T}^d\times\mathcal{Y}^d )$,  written as $w_\epsilon \overset{2}{\rightharpoonup} w$ in $ {L}^q (\mathbb{T}^d\times\mathcal{Y}^d )$, if for all $\psi\in {C}
        ^\infty (\mathbb{T}^d; C^\infty_\#(\mathcal{Y}^d) )$, we have 
        \begin{equation}
        \label{TwoScaleWeaklyConvDefEq}
        \lim\limits_{\epsilon\rightarrow 0}\int_{\mathbb{T}^d} w_\epsilon (x )\psi \Big(x,\frac{x}{\epsilon} \Big) dx=\int_{\mathbb{T}^d}\int_{\mathcal{Y}^d} w (x,y ) \psi (x,y ) dydx.
        \end{equation}
        Furthermore, we say that $ (w_\epsilon )_\epsilon$ strongly two-scale converges to $w$, denoted by $w_\epsilon \overset{2}{\rightarrow} w$ in $
{L}^q (\mathbb{T}^d\times\mathcal{Y}^d )$, if $w_\epsilon\overset{2}{\rightharpoonup}w$ in $
{L}^q (\mathbb{T}^d\times\mathcal{Y}^d )$ and 
        \begin{equation*}
        \lim\limits_{\epsilon\rightarrow 0} \|w_\epsilon \|_{{L}^q (\mathbb{T}^d )}= \|w \|_{{L}^q (\mathbb{T}^d\times\mathcal{Y}^d )}.
        \end{equation*}
\end{definition}

\begin{remark}\label{rmk:uniq}
If it exists, the  two-scale limit is
unique.\end{remark}

\begin{remark}\label{rmk:testfcts}
Assume that $ (w_\epsilon )_\epsilon $ is a bounded sequence in \(L^p(\mathbb{T}^d )\). Then, a density argument shows that
\eqref{TwoScaleWeaklyConvDefEq} holds for all  $\psi\in
{C}
        ^\infty (\mathbb{T}^d; C^\infty_\#(\mathcal{Y}^d) )$ if and only
        if it holds for all $\psi\in L^p(\mathbb{T}^d; C_\#(\mathcal{Y}^d) )$.
\end{remark}

The next proposition relates the usual strong and weak convergence with the two-scale counterpart. In particular, it shows that the two-scale weak limit contains more information on the periodic
oscillations of a sequence than the usual weak limit in $L^p$.
This is because the  usual weak limit  equals the average over the periodicity cell $\mathcal{Y}^d$ of the two-scale weak limit.  
\begin{pro}[cf. {\cite[Theorem 1.3]{visintin2006towards}}]
        \label{TwoScaleConvProduct}
        Let $ (w_\epsilon )_\epsilon$ be a bounded sequence in ${L}^p (\mathbb{T}^d )$ and  $w\in{L}^p (\mathbb{T}^d\times\mathcal{Y}^d )$. Then, 
        $$w_\epsilon \overset{2}{\rightharpoonup}  w \text{ in
        } {L}^p (\mathbb{T}^d\times\mathcal{Y}^d )  \Rightarrow \ w_\epsilon \rightharpoonup\int_{\mathcal{Y}^d}w (\cdot,y )dy \ \text{in}\ L^p (\mathbb{T}^d ).$$  
        Moreover,  if $w$ does not depend on $y$ or, in other words,          $w\in{L}^p (\mathbb{T}^d )$, then 
        $$\ w_\epsilon\overset{2}{\rightarrow} w \text{ in }
        {L}^p  (\mathbb{T}^d\times\mathcal{Y}^d )\Leftrightarrow w_\epsilon\rightarrow w \ \text{in}\ L^p (\mathbb{T}^d ).$$
\end{pro}
Next, we give a necessary and sufficient condition for two-scale strong convergence. 
\begin{pro}[cf. {\cite[Definition~4.3 and Lemma 4.4]{zhikov2004two}}]
        \label{StrongWeakProduct}
        Let $ (w_\epsilon )_\epsilon$ be a bounded sequence in ${L}^p (\mathbb{T}^d )$. Then, $w_\epsilon \overset{2}{\rightarrow} w$ in $
{L}^p (\mathbb{T}^d\times\mathcal{Y}^d )$ for some  $w\in {L}^p (\mathbb{T}^d\times\mathcal{Y}^d )$ if and only if  
        \begin{equation*}
        \lim_{\epsilon\rightarrow 0}\int_{\mathbb{T}^d}w_\epsilon (x )
        \phi_\epsilon (x )dx=\int_{\mathbb{T}^d}\int_{\mathcal{Y}^d}w
         (x,y )\phi (x,y )dydx
        \end{equation*}
        for any bounded sequence $ (\phi_\epsilon )_\epsilon\subset {L}^{p'} (\mathbb{T}^d )$ and any function $\phi\in L^{p'} (\mathbb{T}^d\times\mathcal{Y}^d )$ such that $\phi_\epsilon \overset{2}{\rightharpoonup} \phi$ in $
{L}^{p'} (\mathbb{T}^d\times\mathcal{Y}^d )$.
\end{pro}
Below, we state a compactness result for the two-scale convergence. This result asserts that bounded sequences in $L^p (\mathbb{T}^d )$ are pre-compact with respect to the weak two-scale convergence in $L^p (\mathbb{T}^d )$.
\begin{pro}[cf. {\cite[Theorem~14]{LuNgWa02}}]
        \label{Compactness}
        Let $ (w_\epsilon )_\epsilon$ be a bounded sequence in ${L}^p (\mathbb{T}^d )$. Then, there exists a function, $w\in{L}^p (\mathbb{T}^d\times\mathcal{Y}^d )$, such that, up to a subsequence,  $w_\epsilon \overset{2}{\rightharpoonup} w$ in $
{L}^p (\mathbb{T}^d\times\mathcal{Y}^d )$.
\end{pro}

The next result asserts that Proposition~\ref{Compactness} holds
for \(p=1\) under an equi-integrability additional assumption.

\begin{pro}[cf. {\cite[Theorem~1.1]{BuFo15}}]
        \label{CompactnessL1}
        Let $ (w_\epsilon )_\epsilon$ be a bounded sequence in
${L}^1 (\mathbb{T}^d )$. Assume further that $ (w_\epsilon )_\epsilon$
is equi-integrable; that is, for all \(\delta>0\), there exists
\(\tau>0\) such that 
\begin{equation*}
\begin{aligned}
\sup_{\epsilon} \int_E |w_\epsilon| dx\leq \delta
\end{aligned}
\end{equation*}
whenever \(E\subset \Tt^d\) is a measurable set with \(|E|\leq
\tau\).    Then, there exists a function, $w\in{L}^1
(\mathbb{T}^d\times\mathcal{Y}^d )$, such that, up to a subsequence,
 $w_\epsilon \overset{2}{\rightharpoonup} w$ in $
{L}^1 (\mathbb{T}^d\times\mathcal{Y}^d )$.
In particular, \(w_\epsilon \rightharpoonup\int_{\mathcal{Y}^d}w (\cdot,y )dy\)
in \( L^1 (\mathbb{T}^d )\).
\end{pro}

The $L^p$-norm is lower semi-continuous with respect to the weak topology in $L^p$. The next proposition shows that a  similar result holds with respect to weak two-scale convergence.
\begin{pro}[cf. {\cite[Theorem~17]{LuNgWa02}}]
        \label{LowerSemicontinuous}
        Let $ (w_\epsilon )_\epsilon$ be a bounded sequence in ${L}^p (\mathbb{T}^d )$ such that  $w_\epsilon\overset{2}{\rightharpoonup}w$ in $
{L}^p (\mathbb{T}^d\times\mathcal{Y}^d )$ for some $w\in {L}^p (\mathbb{T}^d\times\mathcal{Y}^d )$. Then,
        \begin{equation*}
        \liminf_{\epsilon\rightarrow 0} \|w_\epsilon \|_{{L}^p (\mathbb{T}^d )}\geq  \|w \|_{{L}^p (\mathbb{T}^d\times\mathcal{Y}^d )}\geq\bigg  \|\int_{\mathcal{Y}^d}w (\cdot,y )dy\bigg \|_{{L}^p (\mathbb{T}^d )}.
        \end{equation*} 
\end{pro}
Next, we recall the notion of Carath\'{e}odory functions. These functions are used to provide an important example of sequences that two-scale converge, as stated in the subsequent proposition. 
\begin{definition}
        \label{CaratheodoryFunc}
        A function $\Phi: \mathbb{T}^d\times\Rr^d\rightarrow\mathbb{R}$ is  a {\em Carath\'{e}odory function} if $\Phi (\cdot,y )$ is continuous for a.e. $y\in\Rr^d$ and $\Phi (x,\cdot )$ is measurable
and \(\Yd\)-periodic for every $x\in\mathbb{T}^d$.
\end{definition}

\begin{pro}[cf. {\cite[Lemma 4.5]{zhikov2004two}}]
\label{CaratheodoryFuncConv}
        Let $\Phi: \mathbb{T}^d\times\Rr^d\rightarrow\mathbb{R}$  be a Carath\'{e}odory function such that  $ |\Phi (x,y ) |\leq \Phi_0 (y )$ for all $x\in\mathbb{T}^d$, for a.e. $y\in\mathcal{Y}^d$, and for some $\Phi_0\in{L}^{p}_\# (\mathcal{Y}^d )$. Let $\epsilon>0$ and $x\in\mathbb{T}^d$, and set $\Phi_\epsilon (x )=\Phi (x,\frac{x}{\epsilon} )$. Then, 
        $$\Phi_\epsilon\overset{2}{\rightarrow} \Phi \text{ in $
{L}^{p} (\mathbb{T}^d\times\mathcal{Y}^d )$}.$$
\end{pro}

The next proposition allows us to extend the class of test functions in the definition of two-scale convergence. 
\begin{proposition}[cf. {\cite[Lemma 4.6]{zhikov2004two}}]
        \label{CaratheodoryTestFunc}
        Suppose that  $ (w_\epsilon )_\epsilon\subset{L}^p (\mathbb{T}^d )$ is  such that  $w_\epsilon \overset{2}{\rightharpoonup} w$ in $
{L}^p (\mathbb{T}^d\times\mathcal{Y}^d )$ for some   $w\in{L}^p (\mathbb{T}^d\times\mathcal{Y}^d )$. Let $\Phi$ be as in Proposition~\ref{CaratheodoryFuncConv} with \(p'\) in place of \(p\). Then, 
        \begin{equation*}
        \lim_{\epsilon\rightarrow 0}\int_{\mathbb{T}^d}w_\epsilon (x )\Phi\Big (x,\frac{x}{\epsilon} \Big)dx=\int_{\mathbb{T}^d}\int_{\mathcal{Y}^d} w (x,y )\Phi (x,y )dydx.
        \end{equation*}
\end{proposition}
Next, we recall the characterization of the two-scale limit of bounded sequences in $W^{1,p}$.
\begin{pro}[cf. {\cite[Theorem~20]{LuNgWa02}}]
        \label{GradientConvergence}
        Let $ (w_\epsilon )_\epsilon$ be a bounded sequence in ${W}^{1,p} (\mathbb{T}^d )$ such that $w_\epsilon \rightharpoonup w$ for some $w\in {W}^{1,p} (\mathbb{T}^d )$. Then, $w_\epsilon  \overset{2}{\rightharpoonup} w$ in $
{L}^p (\mathbb{T}^d\times\mathcal{Y}^d )$ and there exists a function $w_1\in {L}^p (\mathbb{T}^d;{W}^{1,p}_\# (\mathcal{Y}^d )/\mathbb{R } )$ such that, up to a subsequence, $\nabla w_\epsilon \overset{2}{\rightharpoonup} \nabla w +\nabla_y w_1$ in $
{[L}^p (\mathbb{T}^d\times\mathcal{Y}^d)]^d$.
\end{pro}
The next proposition, which provides a simple generalization
of \cite[ Theorem~7.1]{zhikov2004two},  entails the lower semi-continuity of certain convex functionals with respect to the two-scale convergence
(also see  \cite[Proposition~3.1(iii)]{Vi07}).
\begin{pro}
        \label{TwoScaleConvLipshitz}
        Let $ (w_\epsilon )_\epsilon$ be a bounded sequence in $ [{L}^p (\mathbb{T}^d ) ]^d$ such that $w_\epsilon  \overset{2}{\rightharpoonup} w$  in $
{[L}^p (\mathbb{T}^d\times\mathcal{Y}^d)]^d$ for some $w\in  [{L}^p (\mathbb{T}^d\times\mathcal{Y}^d ) ]^d$. Suppose that $f: \mathbb{R}^d\times\mathbb{R}^d \rightarrow  [0,+\infty ]$ is a Borel function, \(\Yd\)-periodic in the first variable and such that for each fixed $y$, $f (y,\cdot )$ is convex and lower semi-continuous on $\mathbb{R}^d$ and $f (y,0 )=0$. Then, for all $\phi\in {C}^\infty (\mathbb{T}^d;C^\infty_\#(\mathcal{Y}^d) )$ with $\phi\geq 0$, we have
        \begin{equation}
        \label{TwoScaleConvexIneq}
        \liminf_{\epsilon\rightarrow 0}\int_{\mathbb{T}^d} f \Big(\frac{x}{\epsilon}, w_\epsilon (x )\Big )\phi\Big (x,\frac{x}{\epsilon} \Big)dx\geq \int_{\mathbb{T}^d}\int_{\mathcal{Y}^d}f (y, w (x,y ) )\phi (x,y )dydx.
        \end{equation}
\end{pro}

\begin{proof}
The proof is a simple adaptation of the proof of Theorem 7.1 in \cite{zhikov2004two}, which corresponds to \eqref{TwoScaleConvexIneq} with $\phi\equiv 1$. We first recall the main arguments in \cite{zhikov2004two},
after which we describe how to adapt these arguments to the present
setting. 

In what follows,  $f^*$ is the convex conjugate of $f$; that
is,   for all $ (y,\eta )\in \mathbb{R}^d\times\mathbb{R}^d$, %
\begin{equation*}
f^* (y,\eta )=\sup_{\xi\in\mathbb{R}^d} \{\eta^T\xi-f (y,\xi ) \}.
\end{equation*}
The proof in \cite{zhikov2004two} is done in two steps. In the first step, one assumes that for all $ (y, \xi )\in\mathbb{R}^d\times\mathbb{R}^d$ and for some $c_0>0$, $f$ satisfies 
\begin{equation}
\label{flowerbound}
f (y,\xi )\geq c_0 |\xi |^p.
\end{equation}
Due to the preceding condition, $f^*$ is continuous in $\eta$. Let $w_\epsilon$ and $w$ be as in the claim. A key estimate in \cite{zhikov2004two} follows from the definition of $f^*$; more precisely, for any $\psi\in [{C}^\infty (\mathbb{T}^d;C^\infty_\#(\mathcal{Y}^d)
)]^d$, we have 
\begin{equation}
\label{YoungIneq}
f \Big(\frac{x}{\epsilon},w_\epsilon (x )\Big )\geq \bigg\langle w_\epsilon (x ), \psi\Big (x,\frac{x}{\epsilon} \Big)\bigg\rangle-f^* \Big(\frac{x}{\epsilon}, \psi\Big (x,\frac{x}{\epsilon} \Big) \Big).
\end{equation} 
Then, \eqref{TwoScaleConvexIneq} with $\phi\equiv 1$ is obtained by integrating \eqref{YoungIneq} over $\mathbb{T}^d$ and letting $\epsilon\rightarrow 0$ as follows. For the first term on the right-hand
side of \eqref{YoungIneq}, it suffices to use the 
 definition of two-scale convergence. Regarding the second term, we first  set  $\Phi (x,y )=f^* (y,\psi (x,y ) )$ for $ (x,y )\in\mathbb{T}^d\times\Rr^d$. Then, the continuity of $f^*$ in $\eta$, which is implied by \eqref{flowerbound}, gives that $\Phi$ is a Carath\'{e}odory function. Thus,  it suffices  to use Proposition~\ref{CaratheodoryFuncConv} to pass the (integral over $\mathbb{T}^d$ of the)  second term on the right-hand
side of \eqref{YoungIneq} to the limit
as $\epsilon\rightarrow 0$. 
As in \cite{zhikov2004two}, the  conclusion then follows by taking the supremum of the right-hand side of the resulting equation over $\psi\in [L^{p'}(\mathbb{T}^d\times\mathcal{Y}^d)]^d$, the definition of the convex conjugate of $f^*$, denoted by $ (f^* )^*$, and $ (f^* )^*=f$ implied by the lower semi-continuity of $f$ and $f(y,0)=0$. 

The  second step in  \cite{zhikov2004two}
consists in proving \eqref{TwoScaleConvexIneq} with $\phi\equiv 1$ without assuming \eqref{flowerbound} on $f$. To do that, the author uses the function $g$ defined, for all $ (y,\xi )\in\mathbb{R}^d\times\mathbb{R}^d$ and for a positive real number $\delta>0$, by
\begin{equation}
\label{DefAuxg}
g (y,\xi )=f (y,\xi )+\delta |\xi |^p.
\end{equation}
Since $g$ satisfies the conditions in the first step, \eqref{TwoScaleConvexIneq} with $\phi\equiv 1$ holds for $g$. Then, it suffices to let $\delta\rightarrow 0$ to conclude. 

To obtain \eqref{TwoScaleConvexIneq} for all $\phi\in {C}^\infty (\mathbb{T}^d;C^\infty_\#(\mathcal{Y}^d)
)$ with $\phi\geq 0$, we can proceed exactly as in \cite{zhikov2004two}.  More precisely,  we  assume that $f$ satisfies \eqref{flowerbound} first. 
From \eqref{YoungIneq}, we get 
\begin{equation*}
f \Big(\frac{x}{\epsilon},w_\epsilon (x )\Big )\phi\Big (x,\frac{x}{\epsilon} \Big)\geq \bigg\langle w_\epsilon (x ), \psi\Big (x,\frac{x}{\epsilon} \Big)\bigg\rangle\phi\Big (x,\frac{x}{\epsilon} \Big)-f^* \Big(\frac{x}{\epsilon}, \psi \Big(x,\frac{x}{\epsilon} \Big) \Big)\phi\Big (x,\frac{x}{\epsilon} \Big)
\end{equation*}
for $\phi\in {C}^\infty (\mathbb{T}^d\times\mathcal{Y}^d )$ with $\phi\geq 0$. 
Thus,
\begin{align}
\label{YoungInEqMultiPosi}
\begin{split}
&\int_{\mathbb{T}^d}f\Big (\frac{x}{\epsilon},w_\epsilon (x )\Big )\phi\Big (x,\frac{x}{\epsilon} \Big)dx\\
&\quad\geq  \int_{\mathbb{T}^d}\bigg\langle w_\epsilon (x ), \psi \Big(x,\frac{x}{\epsilon} \Big)\bigg\rangle\phi\Big (x,\frac{x}{\epsilon} \Big)dx-\int_{\mathbb{T}^d}f^* \Big(\frac{x}{\epsilon}, \psi\Big (x,\frac{x}{\epsilon} \Big) \Big)\phi\Big (x,\frac{x}{\epsilon} \Big)dx.
\end{split}
\end{align}
Because $\phi$ is smooth and does not depend on $\xi$, setting
  $\Phi (x,y )=f^* (y,\psi (x,y ) )\phi (x,y )$ for $ (x,y
)\in\mathbb{T}^d\times\mathcal{Y}^d$,  we can use the definition of two-scale convergence and properties of $f^*$ to pass \eqref{YoungInEqMultiPosi} to the limit as $\epsilon\rightarrow 0$. Then, as we did when $\phi\equiv 1$, using the definition and properties  of $ (f^* )^*$, we  conclude that \eqref{TwoScaleConvexIneq} holds under the assumption \eqref{flowerbound} on $f$. To remove this assumption, we use the previous case applied to $g$ in \eqref{DefAuxg} as in \cite{zhikov2004two}. Letting $\delta\rightarrow 0$, we obtain  \eqref{TwoScaleConvexIneq}.  \end{proof}

\begin{remark}\label{rmk:onlscwtsc} In  \cite{zhikov2004two},
the condition \(f(y,0)=0\) is used only to guarantee that \((f^* )^*=f\). We observe that this identity holds if $f (y,\cdot
)$ is a real-valued and  convex function,    bounded
from below on $\mathbb{R}^d$. Hence, it can be checked that  Proposition~\ref{TwoScaleConvLipshitz} holds  for Borel functions 
$f: \mathbb{R}^d\times\mathbb{R}^d \rightarrow  \Rr$
bounded from below by some constant, \(\Yd\)-periodic in the
first variable, and  convex in the second
variable. \end{remark}

Finally, we discuss two-scale convergence for the composition of a strongly two-scale convergent function with a Lipschitz function (also see  \cite[Proposition~3.1(ii)]{Vi07}).

\begin{pro}
\label{LipschitzStronglyTwoScaleConv}
        Let $f: \mathbb{R}\rightarrow \mathbb{R}$ be a Lipschitz function and $ (w_\epsilon )_\epsilon$ be a bounded sequence in $L^p (\mathbb{T}^d )$ such that    $w_\epsilon\overset{2}{\rightarrow} w$ in $
{L}^p (\mathbb{T}^d\times\mathcal{Y}^d)$ for some  $w\in L^p (\mathbb{T}^d\times\mathcal{Y}^d )$. Then,  $f (w_\epsilon )\overset{2}{\rightarrow} f (w )$ in $
{L}^p (\mathbb{T}^d\times\mathcal{Y}^d)$.
\end{pro}
\begin{proof}
        Let $ (\phi_\epsilon )_\epsilon$ be a sequence in  $L^{p'} (\mathbb{T}^d )$ such that $\phi_\epsilon\overset{2}{\rightharpoonup} \phi$ in $
{L}^{p'} (\mathbb{T}^d\times\mathcal{Y}^d)$ for some  $\phi \in L^{p'} (\mathbb{T}^d\times\mathcal{Y}^d )$. We prove that  
$\int_{\mathbb{T}^d}f (w_\epsilon(x) )\phi_\epsilon (x )dx \rightarrow\int_{\mathbb{T}^d}\int_{\mathcal{Y}^d}f (w (x,y ) )\phi (x,y )dydx$ and, hence, by Proposition~\ref{StrongWeakProduct}, we establish the claim. 
        
        Let $ (w_n )_n\subset {C}^\infty (\mathbb{T}^d;C^\infty_\#(\mathcal{Y}^d)
)$
be a sequence that strongly converges to $w$
in \(L^p (\mathbb{T}^d\times\mathcal{Y}^d
)\). For \(x\in\Tt^d\), define \(w_{n,\epsilon} (x )=w_n (x,\frac{x}{\epsilon}
)\). We have 
        \begin{align}
        \label{LipschitzStronglyTwoScaleConvIneq1}
        \begin{split}
        & \bigg|\int_{\mathbb{T}^d}f (w_\epsilon (x))\phi_\epsilon (x )dx-\int_{\mathbb{T}^d}\int_{\mathcal{Y}^d}f (w (x,y ) )\phi (x,y )dydx\bigg | \\
        \leq &  \bigg|\int_{\mathbb{T}^d}f (w_\epsilon (x))\phi_\epsilon (x )dx-\int_{\mathbb{T}^d}f(w_{n,\epsilon}
(x ))\phi_\epsilon (x )dx\bigg | \\
        &+\bigg |\int_{\mathbb{T}^d}f(w_{n,\epsilon}
(x ))\phi_\epsilon (x )dx-\int_{\mathbb{T}^d}\int_{\mathcal{Y}^d}f (w_n (x,y ) )\phi (x,y )dydx\bigg |\\
        &+\bigg |\int_{\mathbb{T}^d}\int_{\mathcal{Y}^d}f (w_n (x,y ) )\phi (x,y )dydx-\int_{\mathbb{T}^d}\int_{\mathcal{Y}^d}f (w (x,y ) )\phi (x,y )dydx\bigg |.
        \end{split}
        \end{align}
        For \((x,y)\in\Tt^d\times \Rr^d\),
let $\Phi_n (x,y )=f (w_n (x,y ) )$. From the Lipschitz continuity of  $f$, we have that $\Phi_n$ is continuous
and there exists a constant $C$ such that  $$ |\Phi_n (x,y ) |=  |f (w_n (x,y ) ) |\leq C(1+ |w_n (x,y ) |) \leq C\Big(1+\sup_{x,y} |w_n (x,y ) |\Big)\in\Rr.$$
        Thus, we use Proposition~\ref{CaratheodoryTestFunc}, with \(\Phi_n\)
in place of \(\Phi\), to conclude that
        \begin{equation}
        \label{LipschitzStronglyTwoScaleConvIneq2}
        \lim_{\epsilon\rightarrow 0}\int_{\mathbb{T}^d}f(w_{n,\epsilon}
(x )) \phi_\epsilon (x )dx=\int_{\mathbb{T}^d}\int_{\mathcal{Y}^d}f (w_n (x,y ) )\phi (x,y )dydx.
        \end{equation}
        Because $w_n$ strongly converges to $w$ in ${L}^{p} (\mathbb{T}^d\times\mathcal{Y}^d )$,  the Lipschitz continuity of $f$
and  H\"older's inequality yield 
        \begin{equation}
        \label{LipschitzStronglyTwoScaleConvIneq3}
        \lim_{n\rightarrow \infty}\int_{\mathbb{T}^d}\int_{\mathcal{Y}^d}f (w_n (x,y ) )\phi (x,y )dydx=\int_{\mathbb{T}^d}\int_{\mathcal{Y}^d}f (w (x,y ) )\phi (x,y )dydx.
        \end{equation}
        Similarly, we have
        \begin{align}
        \label{LipschitzStronglyTwoScaleConvIneq4}
        & \bigg|\int_{\mathbb{T}^d}f (w_\epsilon (x))\phi_\epsilon (x )dx-\int_{\mathbb{T}^d}f(w_{n,\epsilon}
(x ))\phi_\epsilon (x )dx\bigg | \leq C \|w_\epsilon -w_{n,\epsilon}
\|_{{L}^p (\mathbb{T}^d )} \|\phi_\epsilon  \|_{{L}^{p'} (\mathbb{T}^d )}. 
        \end{align}
As shown in \cite[(33) in the proof
of Theorem~18]{LuNgWa02}, using  Clarkson's inequalities, we have %
\begin{equation*}
\begin{aligned}
\limsup_{n\to\infty} \limsup_{\epsilon\to0}
 \|w_\epsilon -w_{n,\epsilon}
\|_{{L}^p (\mathbb{T}^d )} =0. 
\end{aligned}
\end{equation*}
Thus,
recalling that \((\phi_\epsilon )_\epsilon\)
is bounded in \(L^{p'}(\Tt^d)\), 
from \eqref{LipschitzStronglyTwoScaleConvIneq4},
we obtain

        \begin{equation}
        \label{LipschitzStronglyTwoScaleConvIneq5}
        \limsup_{n\to\infty}\limsup_{\epsilon\rightarrow 0} \bigg|\int_{\mathbb{T}^d}f (w_\epsilon (x))\phi_\epsilon (x )dx-\int_{\mathbb{T}^d}f (w_{n,\epsilon}
(x ) )\phi_\epsilon (x )dx\bigg |=0.
        \end{equation}
        Therefore, letting $\epsilon
\to 0$ first and then \(n\to\infty\)
in \eqref{LipschitzStronglyTwoScaleConvIneq1},   we conclude that
        \begin{equation*}
         \lim_{\epsilon\rightarrow
0}\bigg|\int_{\mathbb{T}^d}f (w_\epsilon (x))\phi_\epsilon (x )dx-\int_{\mathbb{T}^d}\int_{\mathcal{Y}^d}f (w (x,y ) )\phi (x,y )dydx\bigg |=0
        \end{equation*}
from \eqref{LipschitzStronglyTwoScaleConvIneq2},
\eqref{LipschitzStronglyTwoScaleConvIneq3},
and \eqref{LipschitzStronglyTwoScaleConvIneq5}.
Hence,          $f (w_\epsilon )\overset{2}{\rightarrow}f (w )$ in $
{L}^p (\mathbb{T}^d\times\mathcal{Y}^d)$ by Proposition~\ref{StrongWeakProduct}.
\end{proof}

\begin{remark}\label{rmk:onconttwosc}
A simple modification of the arguments above show  that Proposition~\ref{LipschitzStronglyTwoScaleConv} still holds if \(f\) is locally
Lipschitz and \(w_\epsilon\) and \(w\) take values on a compact
subset of \(\Rr\).
\end{remark}

\section{Bounds for solutions to Problems~\ref{POCMFG} and \ref{VariationalProblem}}
\label{SectionOriginalMFGAnalysis}
In this section, we examine uniform bounds in $\epsilon$ for the solution to \eqref{POCMFGEq}. We recall that, by the results in \cite{evans2003some},  there exists a unique smooth solution to  Problems~\ref{POCMFG} and \ref{VariationalProblem}.

\begin{pro}
        \label{HigherDimBoundednessofHepandUep}  
        Let $u_\epsilon\in C^\infty (\mathbb{T}^d )$ solve Problem~\ref{VariationalProblem} and $\overline{H}_\epsilon$ be as in \eqref{DefHEpsilon}. Then, 
        \begin{equation}
        \label{HigherDimHepBounds}
        \inf\limits_{(x,y)\in\Tt^d \times \Yd}V (x,y )\leq \overline{H}_\epsilon (P )\leq \frac{ |P |^2}{2}+\sup\limits_{(x,y)\in\Tt^d \times \Yd}V (x,y ).
        \end{equation}
        Moreover, 
        \begin{equation}
        \label{UniformBoundGradientuepsilon}
        \int_{\mathbb{T}^d}  |\nabla u_\epsilon (x ) |^2dx \leq 2\bigg(\sup_{(x,y)\in\Tt^d \times \Yd}
V (x,y )-\inf_{(x,y)\in\Tt^d \times \Yd} V (x,y )\bigg).
        \end{equation}
\end{pro}
\begin{proof}
        Choosing $u=0$ in \eqref{VariationalFormula} and using the definition of $\overline{H}_\epsilon$ in \eqref{DefHEpsilon}, we have 
        \begin{equation}
        \label{HigherDimUpperBoundHbarep}
        \overline{H}_\epsilon(P)\leq \ln \int_{\mathbb{T}^d} e^{\frac{ |P |^2}{2}+V (x,\frac{x}{\epsilon} )}dx\leq \frac{ |P |^2}{2}+\sup_{x,y}V (x,y ).
        \end{equation}
        By Jensen's inequality, we get
        \begin{align}
        \label{HigherDimBoundednessofHepandUepIneq1}
        \begin{split}
        \overline{H}_\epsilon (P )= \ln I_\epsilon [u_\epsilon ] &= \ln  \int_{\mathbb{T}^d}e^{ \frac{ |P+ \nabla u_\epsilon (x ) |^2}{2}+V (x,\frac{x}{\epsilon}  )}dx\\
        &\geq \int_{\mathbb{T}^d}  \bigg(\frac{ |P+ \nabla u_\epsilon (x ) |^2}{2}+V \Big(x,\frac{x}{\epsilon} \Big) \bigg) dx \geq \inf_{x,y}V (x,y ).
        \end{split}
        \end{align}
        Thus, \eqref{HigherDimUpperBoundHbarep} and \eqref{HigherDimBoundednessofHepandUepIneq1} yield \eqref{HigherDimHepBounds}. 
        Besides, combining the first inequality in \eqref{HigherDimBoundednessofHepandUepIneq1} with \eqref{HigherDimUpperBoundHbarep}, we obtain 
        \begin{equation*}
        \int_{\mathbb{T}^d} \frac{ |P+ \nabla u_\epsilon (x ) |^2}{2} dx \leq  \overline{H}_\epsilon (P )-\inf_{x,y} V (x,y )\leq \frac{ |P |^2}{2}+\sup_{x,y} V (x,y )-\inf_{x,y} V (x,y ),
        \end{equation*}
        which gives
        \begin{equation*}
        \frac{ |P |^2}{2}+P^T\int_{\mathbb{T}^d}\nabla u_\epsilon (x )dx+
        \int_{\mathbb{T}^d} \frac{ |\nabla u_\epsilon (x ) |^2}{2} dx \leq \frac{ |P |^2}{2}+\sup_{x,y} V (x,y )-\inf_{x,y} V (x,y ).
        \end{equation*}
        Since $\int_{\mathbb{T}^d} \nabla u_\epsilon (x )dx=0$, we have 
        \begin{equation*}
        \int_{\mathbb{T}^d}\frac{ |\nabla u_\epsilon (x ) |^2}{2} dx \leq \sup_{x,y} V (x,y )-\inf_{x,y} V (x,y ).
        \end{equation*}
        Therefore, we conclude that  \eqref{UniformBoundGradientuepsilon} holds.
\end{proof}

\begin{pro}
        \label{HigherDimmbounds}
        Let $ (u_\epsilon, m_\epsilon, \overline{H}_\epsilon )$ solve Problem~\ref{POCMFG} and \(q\in[1,\infty)\). 
        Then, there exist  positive constants, $C=C(P)$, $C_q=C(q,P)$, and \(C_\epsilon=C(\epsilon, P)\), such that          
        \begin{equation}
        \label{UniformBoundSupGradientu}
        \sup_\epsilon \| u_\epsilon  \|_{W^{1,q}(\Tt^d)}\leq C_q, 
        \end{equation}
        \begin{equation}
        \label{LowerUpperBoundmepsilon}
        \frac{1}{C}\leq \inf_{\Tt^d} m_\epsilon \leq \sup_{\Tt^d} m_\epsilon \leq C_\epsi,
        \end{equation}
and
\begin{equation}
\begin{aligned}\label{bddsmlnm}
\sup_\epsilon\int_{\Tt^d} m_\epsilon(x) \ln(m_\epsilon(x)) dx \leq \frac{|P|^2}{2}+ \sup\limits_{(x,y)\in\Tt^d \times \Yd}V (x,y )-\inf\limits_{(x,y)\in\Tt^d
\times \Yd}V (x,y ).
\end{aligned}
\end{equation}
\end{pro}
\begin{proof}
        The estimate in  \eqref{UniformBoundSupGradientu} follows by Lemma 2.1 in \cite{evans2003some}.    Regarding the estimates in
\eqref{LowerUpperBoundmepsilon}, we first observe that from  the first equation in \eqref{POCMFGEq} and \eqref{HigherDimHepBounds}, we get, for all \(x\in\Tt^d\),
        $$\ln m_\epsilon (x )\geq V (x,\frac{x}{\epsilon} )-\overline{H}_\epsilon (P )\geq \inf\limits_{x,y}V(x,y)-\sup\limits_{x,y}V (x,y )-\frac{ |P |^2}{2}.$$
        Thus, for all \(x\in\Tt^d\), $$m_\epsilon (x )\geq e^{\inf\limits_{x,y}V(x,y)-\sup\limits_{x,y}V (x,y )-\frac{ |P |^2}{2}}.$$ 
        Besides, by Theorem 5.1 in \cite{evans2003some},  $\sup_{\Tt^d}m_\epsilon\leq C_\epsilon$ for some positive
constant $C_\epsilon$, depending on \(\epsilon\) and \(P\).  Hence, we conclude that \eqref{LowerUpperBoundmepsilon} holds.
 
Finally, multiplying the first equation
in \eqref{POCMFGEq}  by \(m_\epsilon\) and the second equation by $u_\epsilon$, and then integrating over \(\Tt^d\) the difference between the resulting equations, we obtain
\begin{equation*}
\begin{aligned}
\int_{\Tt^d}m_\epsilon(x)\frac{|\nabla u_\epsilon(x)|^2}{2}dx+ \int_{\Tt^d} m_\epsilon(x) \ln(m_\epsilon(x)) dx =-\overline H_\epsilon(P)+ \frac{\left|P\right|^2}{2} + \int_{\Tt^d} m_\epsilon(x) V\Big(x, \frac x\epsilon \Big)dx,
\end{aligned}
\end{equation*}
where we used the condition \(\int_{\Tt^d} m_\epsilon(x) dx = 1\). Using this last condition once more and Proposition~\ref{HigherDimBoundednessofHepandUep}, we conclude that
\begin{equation}
\label{eq:onbddsmlnm}
\begin{aligned}
&\int_{\Tt^d} m_\epsilon(x)\frac{ | \nabla u_\epsilon (x
) |^2}{2} dx +\int_{\Tt^d} m_\epsilon(x) \ln(m_\epsilon(x)) dx \\
&\quad \leq \frac{|P|^2}{2}+
\sup\limits_{(x,y)\in\Tt^d \times \Yd}V (x,y )-\inf\limits_{(x,y)\in\Tt^d \times \Yd}V (x,y ).
\end{aligned}
\end{equation}
Because the first term on the left-hand side of \eqref{eq:onbddsmlnm} is nonnegative, we obtain \eqref{bddsmlnm}.           
\end{proof}

\begin{pro}
        \label{UniformlyConvergence}
        Let $ (u_\epsilon, m_\epsilon, \overline{H}_\epsilon )$ solve Problem
\ref{POCMFG} and \(q\in[1,\infty)\).
        Then, there exists \(\alpha\in(0,1)\) and there  exist $u_0\in C^{0,\alpha}\cap{W}^{1,q} (\mathbb{T}^d )$   with \(\int_{\Tt^d} u_0dx =0\),  $u_1\in {L}^q (\mathbb{T}^d;{W}^{1,q}_\# (\mathcal{Y}^d
)/\mathbb{R } )$, 
  $m\in {L}^1
(\mathbb{T}^d\times\mathcal{Y}^d )$
with \(\int_{\mathbb{T}^d}\int_{\mathcal{Y}^d}
m(x,y) dydx=1\),
and \(\overline H(P)\in\Rr\) such that, up to a subsequence, 
\begin{align}
&u_\epsilon \rightarrow u_0 \
\text{in } L^\infty(\Tt^d),
        \quad u_\epsilon \rightharpoonup
 u_0 \ \text{in} \ {W}^{1,q}
(\mathbb{T}^d ), \label{eq:convweakuepsi}\\
& \nabla u_\epsilon
\overset{2}{\rightharpoonup} \nabla
u_0 +\nabla_y u_1 \ \text{in} \ {[L}^q
(\mathbb{T}^d\times\mathcal{Y}^d)]^d,\label{eq:conv2uepsi}\\
&m_\epsilon \overset{2}{\rightharpoonup}
 m \text{ in }  {L}^1
(\mathbb{T}^d\times\mathcal{Y}^d ),
\quad m_\epsilon \rightharpoonup  m_0
=\int_{\mathcal{Y}^d}m
(\cdot,y )dy \text{ in } {L}^1
(\mathbb{T}^d ),\label{eq:conv2mepsi}\\
&\overline H_\epsilon(P) \to \overline
H(P) \text{ in
} \Rr. \label{eq:convH}
\end{align}
\end{pro}

\begin{proof}
The existence  of \(u_0\in C^{0,\alpha}\cap{W}^{1,q} (\mathbb{T}^d
)$, for some \(\alpha\in(0,1)\), with \(\int_{\Tt^d} u_0dx =0\) and  satisfying \eqref{eq:convweakuepsi} follows from the uniform estimate in \eqref{UniformBoundSupGradientu} together with
Morrey's embedding theorem and the condition  \(\int_{\Tt^d} u_\epsilon dx =0\). Then, from  Proposition~\ref{GradientConvergence}, we conclude that there
exists   $u_1\in {L}^q (\mathbb{T}^d;{W}^{1,q}
_\#(\mathcal{Y}^d
)/\mathbb{R } )$ for which \eqref{eq:conv2uepsi} holds. 
Next, we observe that the map  \(t\in\Rr^+\mapsto t \ln
t\) is bounded from below and $\lim_{t\to\infty} \frac{t\ln t}{
t} =\infty$.  Hence, the existence  of \(m\in {L}^1
(\mathbb{T}^d\times\mathcal{Y}^d )\) satisfying \eqref{eq:conv2mepsi} follows from Proposition~\ref{CompactnessL1}
and  the uniform estimate in
\eqref{bddsmlnm} together with the de la Vall\'ee Poussin criterion for equi-integrability. We observe further that the condition \(\int_{\Tt^d} m_\epsilon(x) dx
= 1\)
yields \(1=\int_{\Tt^d} m_0(x) dx =\int_{\mathbb{T}^d}\int_{\mathcal{Y}^d}
m(x,y) dydx\). Finally, \eqref{eq:convH} follows from
the uniform estimates in \eqref{HigherDimHepBounds}, which conclude the proof.
\end{proof}

\section{Two-scale homogenization in one dimension}
\label{TwoScaleHomogeInOneDim}
In this section, we consider the two-scale homogenization of \eqref{POCMFGEq} in one dimension, for which we have an explicit solution. Here, we denote $\mathbb{T}^1=\mathbb{T}$ and $\mathcal{Y}^1=\mathcal{Y}$, and we  identify \(\Tt\) with $ [0,1 ]$. If $d=1$, \eqref{POCMFGEq} becomes
\begin{align}
\label{OneDPOCMFG}
\begin{cases}
\frac{ (P+  (u_\epsilon (x ) )_x )^2}{2}+V (x,\frac{x}{\epsilon} )=\ln m_\epsilon (x )+\overline{H}_\epsilon (P ),  \ &\text{in}\ \mathbb{T}, \\
- (m_\epsilon (x ) (P+  (u_\epsilon (x ) )_x ) )_x=0, \ &\text{in}\ \mathbb{T}, \\
\int_0^1 u_\epsilon (x )dx=0, \int_0^1 m_\epsilon (x )dx=1.
\end{cases}
\end{align}

\subsection{The current formulation}
To find the explicit solution of \eqref{OneDPOCMFG}, we use the method introduced in \cite{Gomes2016b}. Let  $ (u_\epsilon, m_\epsilon, \overline{H}_\epsilon )\in C^\infty
(\mathbb{T} )\times C^\infty (\mathbb{T} )\times\mathbb{R}$,
with $m_\epsilon>0$,  solve  \eqref{OneDPOCMFG}, and set
\begin{equation}
\label{DefCurrent}
j_\epsilon
=m_\epsilon (P+ (u_\epsilon )_x ).
\end{equation}
The second equation of \eqref{OneDPOCMFG} gives $ (j_\epsilon )_x=0$. Then, $j_\epsilon$ is a constant independent of $x$. Thus, the first equation in  \eqref{OneDPOCMFG} becomes  
\begin{align}
\label{OneDCurrentFormulation}
\frac{j^2_\epsilon}{2m^2_\epsilon (x )}+V \Big(x,\frac{x}{\epsilon} \Big)=\ln m_\epsilon (x )+\overline{H}_\epsilon (P ), \ \text{in}\ \mathbb{T}.
\end{align}

Next, we find the explicit formulas for $u_\epsilon, m_\epsilon$,  $\overline{H}_\epsilon$,  and their two-scale limits, $u_0$, $m_0$, and $\overline{H}$, whose existence is asserted in Proposition \ref{UniformlyConvergence}. We discuss
the $P=0$ case first, and then the general case. 
\begin{pro}\label{prop:P0}
 Let  $ (u_\epsilon, m_\epsilon,
\overline{H}_\epsilon (0))\in C^\infty
(\mathbb{T} )\times C^\infty (\mathbb{T} )\times\mathbb{R}$,
with $m_\epsilon>0$,  solve  \eqref{OneDPOCMFG}, let    
\(j_\epsilon\) be given by \eqref{DefCurrent}, and
let \(u_0\),  \(m\), and \(\overline
H(0)\) be given by Proposition~\ref{UniformlyConvergence}, with \(P=0\).  Then, for all  $x\in\mathbb{T}$,   $u_\epsilon (x )=0$,  $m_\epsilon
(x )=e^{V (x,\frac{x}{\epsilon} )-\overline{H}_\epsilon(0)}$,   $\overline{H}_\epsilon(0)=\ln\int_0^1e^{V (x,\frac{x}{\epsilon} )}dx$, and $j_\epsilon=0$.  Moreover, for all
\(x\in\Tt\) and \(y\in
\mathcal{Y}\), \(u_0(x) =0\),    \(m(x,y)= e^{V(x,y)-\overline{H}(0)} \), and $\overline{H}(0)=\ln \int_0^1\int_0^1e^{V
(x,y )}dydx$. Furthermore,   $m_\epsilon\overset{2}{\rightarrow}m$ in \(L^p(\Tt\times \mathcal{Y})\)
 for all \(p\in(1,\infty)\). \end{pro}
\begin{proof}
By  \eqref{DefCurrent} with  $P=0$, we have $j_\epsilon = m_\epsilon  (u_\epsilon )_x$. Because \(u_\epsilon\in C^\infty(\Tt)\) and \(\int_0^1 u_\epsilon dx=0\), there exists \(x_\epsi\in\Tt\) such that \((u_\epsilon)_x(x_\epsilon)=0\).
Hence, recalling that \(j_\epsilon\) is independent of \(x\)
and \(m_\epsilon>0\) in \(\Tt\), it follows that \(j_\epsilon
=0\) and \((u_\epsilon)_x=0\) in \(\Tt\).
This last condition together with the fact that \(\int_0^1 u_\epsilon dx =
0\) yields \(u_\epsilon =0\) in \(\Tt\). Consequently, also \(u_0=0\)
in \(\Tt\).
 
 Next, using \eqref{OneDCurrentFormulation}
with  $P=0$ and $\int_0^1 m_\epsilon (x )dx=1$, we get 
\begin{equation*}
\begin{aligned}
m_\epsilon (x )=e^{V (x,\frac{x}{\epsilon} )-\overline{H}_\epsilon(0)} \text{
 for all
\(x\in\Tt\), and }\overline{H}_\epsilon(0)=\ln\int_0^1e^{V (x,\frac{x}{\epsilon} )}dx. 
\end{aligned}
\end{equation*}
Since $V$ is smooth, $\Phi(x,y)=e^{V(x,y)}$ satisfies the conditions of Proposition~\ref{CaratheodoryFuncConv} for any \(p\in(1,\infty)\). Accordingly, setting \(\Phi_\epsilon(x)= \Phi(x,\frac x \epsilon)\), we
have   $\Phi_\epsilon\overset{2}{\rightarrow}\Phi$ in \(L^p(\Tt\times
\mathcal{Y})\)
 for all \(p\in(1,\infty)\); in particular,   
\begin{equation*}
\lim_{\epsilon\rightarrow 0}\int_0^1e^{V (x,\frac{x}{\epsilon} )}dx=\int_0^1\int_0^1e^{V (x,y )}dydx,
\end{equation*}
 which yields $\overline{H}_\epsilon(0)\rightarrow\overline{H}(0)$
 in \(\Rr\).  Moreover, using the uniqueness of two-scale limits,
\(m(x,y)= e^{V(x,y)-\overline{H}(0)}
\) for all \((x,y)\in\Tt\times\mathcal{Y}\) and   $m_\epsilon\overset{2}{\rightarrow}m$ in \(L^p(\Tt\times
\mathcal{Y})\)
 for all \(p\in(1,\infty)\).
\end{proof}

Next, we examine the general case \(P\in\Rr\).

\begin{pro}
\label{umepsilonexplicit}
Let  $ (u_\epsilon, m_\epsilon,
\overline{H}_\epsilon )\in C^\infty
(\mathbb{T} )\times C^\infty (\mathbb{T} )\times\mathbb{R}$,
with $m_\epsilon>0$,  solve  \eqref{OneDPOCMFG}  and  
\(j_\epsilon\) be given by \eqref{DefCurrent}. Let \(F_\epsilon:\Rr^+\to\Rr\)
be the function defined, for  $t>0$,  by 
\begin{equation}
\label{DefFEpsilon}
F_\epsilon (t )=\frac{j_\epsilon^2}{2t^2}-\ln t.
\end{equation}
Then, for all \(x\in\Tt\),
\begin{equation*}
j_\epsilon=\frac{P}{\int_{0}^{1}\frac{1}{m_\epsilon (s )}ds},
\end{equation*}
\begin{align}
m_\epsilon (x )=F^{-1}_\epsilon \Big(\overline{H}_\epsilon (P )-V \Big(x,\frac{x}{\epsilon} \Big) \Big), \label{eq:me=Fe-1}
\end{align}
and 
\begin{align*}
u_\epsilon (x )=\int_0^x\frac{j_\epsilon}{m_\epsilon (s )}ds-P x+P-\int_0^1\int_0^z\frac{j_\epsilon}{m_\epsilon (s )}dsdz.
\end{align*}
Furthermore, there exists $j\in\mathbb{R}$ such that, up to a subsequence, $j_\epsilon\rightarrow j$ as $\epsilon\rightarrow 0$. 
\end{pro}
 \begin{proof}
We start by observing that $F_\epsilon$ belongs to \(C^\infty(\Rr^+)\)  and is a decreasing and convex function in \(\Rr^+\).  Then,  \eqref{OneDCurrentFormulation} yields  
\begin{align*}
m_\epsilon (x )=F^{-1}_\epsilon \Big(\overline{H}_\epsilon (P )-V \Big(x,\frac{x}{\epsilon} \Big) \Big). 
\end{align*}
Moreover, by Jensen's inequality, 

\begin{equation*}
\begin{aligned}
\frac{j_\epsilon^2}{2} = F_\epsilon(1) = F_\epsilon \bigg( \int_0^1 m_\epsilon (x )dx\bigg) \leq \int_0^1 F_\epsilon(m_\epsilon (x ))dx = \int_0^1 \Big(\overline{H}_\epsilon (P
)-V \Big(x,\frac{x}{\epsilon}\Big) \Big)dx .
\end{aligned}
\end{equation*}
This estimate, Proposition~\ref{HigherDimBoundednessofHepandUep},
and the smoothness of \(V\) imply that \(j_\epsilon\) is uniformly
bounded; thus,  up to a subsequence \(j_\epsilon\to j\) in \(\Rr\)
for some \(j\in\Rr\).

On the other hand, from  \eqref{DefCurrent}, recalling that \(m_\epsilon>0\) and 
 $\int_{0}^{1}u_\epsilon (x )dx=0$,
 we obtain  
\begin{align*}
u_\epsilon (x )=\int_0^x\frac{j_\epsilon}{m_\epsilon (s )}ds-P x+ P-\int_0^1\int_0^z\frac{j_\epsilon}{m_\epsilon (s )}dsdz.
\end{align*}
Moreover, by the periodicity of $u_\epsilon$, we have $u_\epsilon (0 )=u_\epsilon (1 )$, which implies that 
\begin{equation*}
\int_{0}^{1}\frac{j_\epsilon}{m_\epsilon (x )}dx-P=0.
\end{equation*}
Therefore, 
\begin{equation*}
j_\epsilon=\frac{P}{\int_{0}^{1}\frac{1}{m_\epsilon (s )}ds}\cdot\qedhere
\end{equation*}
\end{proof}

Let $j$ be the limit of $j_\epsilon$ given in Proposition~\ref{umepsilonexplicit},
and let \(F:\Rr^+\to\Rr\)
be the function defined, for  $t>0$,  by 
\begin{equation}
\label{DefF}
F (t )=\frac{j^2}{2t^2}-\ln t.
\end{equation}
Note  that $F$ belongs to \(C^\infty(\Rr^+)\)
 and is a decreasing and convex function in \(\Rr^+\).

\begin{lemma}
        \label{InvFLipschitz}
        Let  $F_\epsilon^{-1}$ and $F^{-1}$ be the inverse functions of $F_\epsilon$ and $F$  defined in \eqref{DefFEpsilon} and \eqref{DefF},
respectively. 
        Then, $F_\epsilon$, $F_\epsilon^{-1}$, $F$, and $F^{-1}$ are Lipschitz continuous on any closed and bounded interval, $ [a,b ]$, of their domains and the corresponding Lipschitz constants on $ [a,b ]$ are bounded uniformly as $\epsilon\rightarrow 0$.
\end{lemma}
\begin{proof}
        By the definition of \(F_\epsilon\), we have, for all  $0<a\leq
        t \leq b <+\infty$, 
        $$-\frac{j^2_\epsilon}{a^3}-\frac{1}{a}\leq F'_\epsilon (t )=-\frac{j^2_\epsilon}{t^3}-\frac{1}{t}\leq -\frac{j^2_\epsilon}{b^3}-\frac{1}{b}<0.$$
        Since $j_\epsilon$ is convergent by Proposition~\ref{umepsilonexplicit}, we have that $F'_\epsilon$ is  uniformly bounded   on $ [a,b ]$. Hence, $F_\epsilon$ is Lipschitz on \([a,b]\) and the corresponding Lipschitz constant is bounded uniformly as $\epsilon\rightarrow 0$. 
        By the inverse function theorem, a similar statement
holds true $F^{-1}_\epsilon$. Moreover, analogous  arguments hold for $F$ and $F^{-1}$. 
\end{proof}
\begin{lemma}
        \label{FUniformlyConvergence}
        Let  $F_\epsilon^{-1}$ and $F^{-1}$ be the inverse functions
of $F_\epsilon$ and $F$  defined in \eqref{DefFEpsilon} and \eqref{DefF},
respectively, and let $-\infty<c<d<+\infty$. Then, there exists a subsequence of  $ (F^{-1}_\epsilon )_\epsilon$ that converges uniformly  to $F^{-1}$ on $ [c,d ]$ as $\epsilon\rightarrow 0$.
\end{lemma}
 \begin{proof}
        By Proposition~\ref{umepsilonexplicit}, we have, up to a subsequence that
we do not relabel, \(j_\epsilon\to j\) in \(\Rr\). Then, using the definitions of $F$ and $F_\epsilon$, we obtain, for any   $0<a\leq
        t \leq b <+\infty$, 
        \begin{align}\label{eq:unifFepsiF}
        \limsup_{\epsilon\to0}\sup_{t\in [a,b ]} |F_\epsilon (t )-F (t ) |&=\limsup_{\epsilon\to0}\sup_{t\in [a,b ]} \bigg|\frac{j_\epsilon^2}{2t^2}-\frac{j^2}{2t^2} \bigg|
\leq\lim_{\epsilon\to0} \frac{ |j_\epsilon^2-j^2 |}{2a^2}=0 .
        \end{align}
        Thus,  $(F_\epsilon)_\epsilon$ converges uniformly to $F$  on every compact subset of \(\Rr^+\).

Fix $-\infty<c<d<+\infty$, and set \(a=F^{-1}(d)\) and \( b=F^{-1}(c)\).
Then, by the uniform convergence just established, there exist  \(c'\leq c\) and \(d'\geq d\) such that \(F_\epsilon([a,b])\subset
[c',d']\)  for all \(\epsilon>0\) sufficiently small. Moreover,
by  
Lemma~\ref{InvFLipschitz},  there exists a constant
$C_{c',d',\epsilon}>0$ depending on $c'$, $d'$, and $\epsilon$, and
uniformly bounded as $\epsilon\rightarrow 0$,  such that
\begin{equation*}
\begin{aligned}
 |F_\epsilon^{-1}(z) - F_\epsilon^{-1}(z')| \leq C_{c',d',\epsilon}|z-z'| \quad \text{for all } z,\, z'\in[c',d']. 
\end{aligned}
\end{equation*}
Hence, 
\begin{equation*}
\begin{aligned}
\sup_{z\in[c,d]} |F_\epsilon^{-1}(z) - F^{-1}(z)| &= \sup_{z\in[c,d]} |F_\epsilon^{-1}(F(F^{-1}(z))) - F^{-1}(z)| = \sup_{t\in[a,b]} |F_\epsilon^{-1}(F(t)) - t|\\
& = \sup_{t\in[a,b]}
|F_\epsilon^{-1}(F(t)) - F_\epsilon^{-1}(F_\epsilon(t))|\leq C_{c',d',\epsilon}
\sup_{t\in[a,b]}|F(t) -F_\epsilon(t) |
 \end{aligned}
\end{equation*}
for all  \(\epsilon>0\) sufficiently small. Thus, by \eqref{eq:unifFepsiF},
we conclude that   $ (F^{-1}_\epsilon )_\epsilon$ that converges
uniformly  to $F^{-1}$ on $ [c,d ]$ as $\epsilon\rightarrow 0$.
 \end{proof}

\begin{pro}
\label{OneDmConv} 
        Let  $ (u_\epsilon, m_\epsilon,
\overline{H}_\epsilon )\in C^\infty
(\mathbb{T} )\times C^\infty (\mathbb{T}
)\times\mathbb{R}$,
with $m_\epsilon>0$,  solve  \eqref{OneDPOCMFG}
and let \(m\) and $\overline{H}
(P )$ be given by Proposition~\ref{UniformlyConvergence}.
Then,
 for all $ (x,y )\in\mathbb{T}\times\mathcal{Y}$, we have 
\begin{equation}
\label{eq:m1d}
\begin{aligned}
 m (x,y )=F^{-1} (\overline{H} (P )-V (x,y ) ).
\end{aligned}
\end{equation}
 Moreover, \((m_\epsilon)_\epsilon\) is uniformly bounded in \(L^\infty(\Tt)\) and   
        $m_\epsilon\overset{2}{\rightarrow} m$ in \(L^p(\Tt\times
\mathcal{Y})\)
 for all \(p\in(1,\infty)\).  \end{pro}

\begin{proof}
        
For \(x\in\Tt\) and \(y\in\mathcal{Y}\), set  $w_\epsilon (x )=\overline{H}_\epsilon
(P )-V (x,\frac{x}{\epsilon} )$ and
\(w(x,y) =\overline{H}
(P )-V (x,y ) \). By  Proposition
\ref{CaratheodoryFuncConv}, \eqref{eq:convH},
and the smoothness of \(V\),   we
have
\begin{equation}
\label{eq:cons2scwe}
\begin{aligned}
w_\epsilon\overset{2}{\rightarrow}w
\text{ in } L^p(\Tt\times
\mathcal{Y}) \text{ for all } p\in(1,\infty)\end{aligned}
\end{equation}
and there exists \(c\in\Rr\),
independent of \(\epsilon\),
such that \(|w_\epsilon(x)|\leq c\)
for all \(x\in\Tt\).
 On the other hand, by Lemma~\ref{FUniformlyConvergence},
there exists  \(\tilde c\in\Rr\),
independent of \(\epsilon\),
such that \(F_\epsilon^{-1}([-
c,  c])\subset [-\tilde
c, \tilde c]\). Recalling \eqref{eq:me=Fe-1} and the lower bound in
\eqref{LowerUpperBoundmepsilon},
we conclude that \(\frac 1C \leq m_\epsi
(x)\leq \tilde c
\) for all \(x\in\Tt\). Then, by Proposition~\ref{StrongWeakProduct}, to show that
       $m_\epsilon\overset{2}{\rightarrow}
F^{-1}(w)$ in \(L^p(\Tt\times
\mathcal{Y})\)
 for all \(p\in(1,\infty)\), it suffices to show that
\begin{equation}
\label{eq:constwme1d}
\lim_{\epsilon\rightarrow 0}\int_{\mathbb{T}^d}m_\epsilon
(x )
        \phi_\epsilon (x )dx=\lim_{\epsilon\rightarrow 0}\int_{\mathbb{T}^d}F^{-1}_\epsilon
(w_\epsilon (x ) )
        \phi_\epsilon (x )dx=\int_{\mathbb{T}^d}\int_{\mathcal{Y}^d}F^{-1}(w
         (x,y ))\phi (x,y )dydx
        \end{equation}
        for any bounded sequence $ (\phi_\epsilon
)_\epsilon\subset {L}^{p'} (\mathbb{T}
)$ and any function $\phi\in L^{p'}
(\mathbb{T}\times\mathcal{Y} )$
such that $\phi_\epsilon \overset{2}{\rightharpoonup}
\phi$ in $
{L}^{p'} (\mathbb{T}^d\times\mathcal{Y}^d
)$.
Fix any such sequence  $ (\phi_\epsilon
)_\epsilon$, and let $c_\phi=\sup_\epsi
\|\phi_\epsi\|_{L^1(\Tt)}$. Then,  we have  
\begin{equation}\label{almconstwme1d}
\begin{aligned}        
        &\bigg  |\int_0^1F^{-1}_\epsilon (w_\epsilon (x ) )\phi_\epsilon (x )dx-\int_0^1\int_0^1F^{-1} (w (x,y ) )\phi (x,y )dydx\bigg |\\
      &\quad  \leq  \bigg|\int_0^1F^{-1}_\epsilon (w_\epsilon (x ) )\phi_\epsilon (x )dx-\int_0^1F^{-1} (w_\epsilon (x ) )\phi_\epsilon (x )dx\bigg |\\
         &\qquad + \bigg|\int_0^1F^{-1} (w_\epsilon (x ) )\phi_\epsilon (x )dx-\int_0^1\int_0^1F^{-1} (w (x,y ) )\phi (x,y )dydx\bigg |\\
&\quad  \leq c_\phi\sup_{z\in [-c,c]} |F^{-1}_\epsilon (z) -F^{-1}
(z)| \\
         &\qquad+ \bigg|\int_0^1F^{-1} (w_\epsilon
(x ) )\phi_\epsilon (x )dx-\int_0^1\int_0^1F^{-1}
(w (x,y ) )\phi (x,y )dydx\bigg |. 
\end{aligned}
\end{equation}
        By Lemma~\ref{FUniformlyConvergence}, $F^{-1}_\epsilon$ converges to $F^{-1}$ uniformly on \([-c,c]\). Moreover,  $F^{-1}$ is locally Lipschitz
by Lemma~\ref{InvFLipschitz}, which
together with \eqref{eq:cons2scwe} and
the boundedness of      $w_\epsilon$ and  $w$, yields  \(F^{-1} (w_\epsilon )\overset{2}{\rightarrow}F^{-1}
(w )\) by Proposition~\ref{LipschitzStronglyTwoScaleConv} (also
see Remark~\ref{rmk:onconttwosc}).  Consequently, letting \(\epsilon\to0\)
 in \eqref{almconstwme1d}, we obtain \eqref{eq:constwme1d}. Finally, we observe
by the uniqueness of two-scale limits,
we have \(m (x,y )=F^{-1}
(w (x,y ) ) = F^{-1} (\overline{H}
(P )-V (x,y ) ).\)           
\end{proof}

\begin{pro}
\label{OneDConvergence}
        Let  $ (u_\epsilon, m_\epsilon,
\overline{H}_\epsilon )\in C^\infty
(\mathbb{T} )\times C^\infty (\mathbb{T}
)\times\mathbb{R}$,
with $m_\epsilon>0$,  solve  \eqref{OneDPOCMFG}
and \(j_\epsilon\) be given by \eqref{DefCurrent}.
Let \(u_0\), \(u_1\), and \(m\) be given
by Proposition~\ref{UniformlyConvergence} and  $j$ by Proposition~\ref{umepsilonexplicit}. Then, for $ (x,y )\in\mathbb{T}\times\mathcal{Y}$, we have 
        \begin{equation}
        \label{JFormula}
        j=\frac{P}{\int_{0}^{1}\int_0^1\frac{1}{m (x,y )}dydx},
        \end{equation}
\begin{equation}
\begin{aligned}\label{eq:u0formula}
u_0 (x )=\int_{0}^{x}\int_{0}^{1}\frac{j}{m (s,y )}dyds-Px+P-\int_0^1\int_{0}^{z}\int_{0}^{1}\frac{j}{m (s,y )}dydsdz,
\end{aligned}
\end{equation}
and
\begin{equation}
\begin{aligned}\label{eq:u1formula}
u_1 (x,y )=u_1 (x,0 )+\int_{0}^{y}\frac{j}{m (x,s )}ds-y\int_0^1\frac{j}{m (x,z )}dz.
\end{aligned}
\end{equation}
\end{pro}
\begin{proof}
        
We first prove that for all $ (x,y )\in\mathbb{T}\times\mathcal{Y}$, we have
        \begin{equation}
        \label{jtwoscalelimit}
         (u_0  )_x(x )+ (u_1  )_y(x,y )=\frac{j}{m (x,y )}-P.
        \end{equation}
Note that \(m>0\) in \(\Tt\times\mathcal{Y}\) by Proposition~\ref{OneDmConv}.
Let \(\varphi \in C^\infty(\Tt; C^\infty_\#( \mathcal{Y}))\), and set \(\varphi_\epsilon(x) = \varphi(x, \frac x\epsilon) \) for \(x\in\Tt\). By Propositions~\ref{CaratheodoryFuncConv} and ~\ref{umepsilonexplicit} and by  \eqref{DefCurrent}, we have
\begin{equation}
\label{eq:jone}
\begin{aligned}
\int_0^1\int_0^1 j\varphi(x,y) dxdy &=\lim_{\epsilon\to0} \int_0^1\int_0^1 j_\epsi \varphi_\epsilon(x) dx=\lim_{\epsilon\to0} \int_0^1\int_0^1
m_\epsi(x) \big(P+ (u_\epsilon )_x (x)\big)  \varphi_\epsilon(x)dx.
\end{aligned}
\end{equation}
We claim that
\begin{equation}
\label{eq:jtwo}
\begin{aligned}
&\lim_{\epsilon\to0} \int_0^1\int_0^1
m_\epsi(x)\big(P+ (u_\epsilon )_x (x)\big)
 \varphi_\epsilon(x)dx\\
&\quad = \int_0^1\int_0^1 m(x,y)\big(P+(u_0  )_x(x )+ (u_1  )_y(x,y )\big)  \varphi(x,y)dxdy,
\end{aligned}
\end{equation}
which, together with \eqref{eq:jone}, yields \eqref{jtwoscalelimit}. To prove \eqref{eq:jtwo}, we first observe that  \((m_\epsilon\varphi_\epsilon)_\epsilon\) is a bounded  sequence in \(L^\infty(\Tt)\), and thus in \(L^p(\Tt)\), by Proposition~\ref{OneDmConv}; next, we observe that  \((P+ (u_\epsilon )_x)_\epsilon \) is a bounded sequence
in \(L^{p'}(\Tt)\) that weakly two-scale converges to \(P+(u_0  )_x+ (u_1
 )_y\) in \(L^{p'}(\Tt\times \mathcal{Y})\) by Proposition~\ref{UniformlyConvergence}.
  Hence, if we show that  \( m_\epsilon\varphi_\epsilon\overset{2}{\rightarrow}m\varphi\)
in \(L^p(\Tt\times \mathcal{Y})\),  then 
\eqref{eq:jtwo}
follows by 
 Proposition~\ref{StrongWeakProduct}. To prove this last convergence,
 we first note that if  $ (\phi_\epsilon )_\epsilon\subset
{L}^{p'}
(\mathbb{T}^d )$ is a bounded sequence such that  $ \phi_\epsilon
\overset{2}{\rightharpoonup}  \phi$ in $
{L}^{p'} (\mathbb{T}\times\mathcal{Y} )$ for some  $\phi\in
L^{p'} (\mathbb{T}\times\mathcal{Y}
) \), then \(\tilde\phi_\epsilon=\phi_\epsilon \varphi_\epsilon\)
defines a bounded sequence in \( {L}^{p'} (\mathbb{T}^d )\) such
that   $\tilde \phi_\epsilon \overset{2}{\rightharpoonup}
\phi\varphi$ in $
{L}^{p'} (\mathbb{T}\times\mathcal{Y} )$ by Proposition~\ref{CaratheodoryTestFunc}
and Definition~\ref{TwoScaleConvDef}. Then,    Proposition~\ref{StrongWeakProduct}  and the convergence   $m_\epsilon\overset{2}{\rightarrow}m$ in \(L^p(\Tt\times \mathcal{Y})\), established in Proposition~\ref{OneDmConv}, yield \( m_\epsilon\varphi_\epsilon\overset{2}{\rightarrow}m\varphi\)
in \(L^p(\Tt\times \mathcal{Y})\).
Hence, \eqref{eq:jtwo}, and consequently  \eqref{jtwoscalelimit}, holds.

        Integrating \eqref{jtwoscalelimit} over $\mathcal{Y}$ and  using the periodicity of \(u_1(x,\cdot)\),   we get
        \begin{equation}
        \label{du0}
         (u_0  )_x=\int_{0}^{1}\frac{j}{m (\cdot,y )}dy-P.
        \end{equation}
        Integrating \eqref{du0} over $\mathbb{T}$ and using the periodicity of $u_0$, we conclude that  \eqref{JFormula} holds.
On the other hand, integrating \eqref{du0} on $ [0, x ]$ and using the condition  $\int_{0}^{1}u_0
(x )dx=0$, we obtain \eqref{eq:u0formula}. Finally, we observe that  
        from \eqref{jtwoscalelimit} and \eqref{du0}, we get
        \begin{equation*}
         (u_1 )_y (x,y )=\frac{j}{m (x,y )}-\int_0^1\frac{j}{m (x,y )}dy,
        \end{equation*}
        from which we deduce \eqref{eq:u1formula}  by integration over \([0,y]\).
\end{proof}

\begin{remark}
Note that if \(P=0\), then \(j=0\) (see \eqref{JFormula}), and
the formulas for   $u_0$, $u_1$, $\overline{H}$, and $m$  in Propositions~\ref{OneDmConv} and  \ref{OneDConvergence} reduce
to those in Proposition~\ref{prop:P0}. Moreover, the smoothness of \(V\) combined with the smoothness on \(F^{-1}\) on compact
sets yield \(m\) smooth. Hence, so is \(u_0\); also,  choosing an appropriate
representative, we may  assume that \(u_1\) is smooth.
Moreover, one can check that \(P+\nabla u_0(x) + \nabla_y u_1(x,y) =
\frac{j}{m(x,y)} \); from this identity, \eqref{DefF}, and \eqref{eq:m1d},
we conclude that  $ (u_0,u_1,m, \overline{H})$ solves Problem~\ref{TwoScaleHomogenized}.

\end{remark}

\section{The homogenized problem}
\label{TheHomogenizedProb}
To obtain the two-scale homogenization of Problems~\ref{POCMFG}
and \ref{VariationalProblem} in higher dimensions, we need to
examine in detail the existence, uniqueness, and regularity of
the solution to the two-scale homogenized problem, 
Problem~\ref{TwoScaleMinimization}.

To do that, we study two subproblems: the cell problem,  Problem~\ref{TheCellProblem}, and the homogenized problem,  Problem~\ref{TheLimitProblem}. The two preceding
problems are analyzed separately in Sections \ref{CellProblemSection}
and \ref{HomogenizedProblemSection}  below. 
\subsection{The cell problem}
\label{CellProblemSection}
Here, we study Problem~\ref{TheCellProblem}. We stress that  Problems~\ref{TheCellProblem}
and \ref{CellProbEulerLagrangeProb} are equivalent (see Remark \ref{rmk:p7p4qui}). 
Thus, 
if we prove existence and  uniqueness of the solution to Problem~\ref{CellProbEulerLagrangeProb},
 Problem~\ref{TheCellProblem} admits a unique minimizer.
 
\subsubsection{Uniqueness} Here, we prove uniqueness of the solution Problem~\ref{CellProbEulerLagrangeProb}.
\begin{pro}
\label{CellProblemNonUniqueness} 
For each $x\in \mathbb{T}^d$ and $\Lambda\in \mathbb{R}^d$, 
Problem~\ref{CellProbEulerLagrangeProb} admits at most one 
solution. 
\end{pro}
\begin{proof}
Here, we use the Lasry-Lions monotonicity argument. For each $x$ and $\Lambda$, we assume that $ (\widetilde{w}_1,\widetilde{m}_1, \widetilde{H}_1
)$ and $ (\widetilde{w}_2,\widetilde{m}_2, \widetilde{H}_2 )$
are two solutions of Problem \ref{CellProbEulerLagrangeProb} as in the statement.
Then, we have
\begin{equation}
\label{PropCellDiffEq}
        \begin{cases}
        \frac{ |\Lambda+\nabla_y\widetilde{w}_1 |^2}{2}-\frac{
|\Lambda+\nabla_y\widetilde{w}_2 |^2}{2}=\ln\widetilde{m}_1-\ln\widetilde{m}_2+\widetilde{H}_1 (x,\Lambda )-\widetilde{H}_2
(x,\Lambda ),\\
        -\div_y \big(\widetilde{m}_1  (\Lambda+\nabla_y\widetilde{w}_1
 )\big )+\div_y \big(\widetilde{m}_2  (\Lambda+\nabla_y\widetilde{w}_2
 )\big )=0.
        \end{cases}
\end{equation}
        Multiplying the first equation by $ (\widetilde{m}_1-\widetilde{m}_2
)$, subtracting it  from the second equation multiplied by $
(\widetilde{w}_1-\widetilde{w}_2 )$, integrating by parts, and using $\int_{\mathcal{Y}^d}\widetilde{m}_1 dy=\int_{\mathcal{Y}^d}\widetilde{m}_2dy=1$,
we get
        \begin{equation*}
        \frac{1}{2}\int_{\mathcal{Y}^d} (\widetilde{m}_1+\widetilde{m}_2
) |\nabla_y\widetilde{w}_1-\nabla_y\widetilde{w}_2 |^2dy+\int_{\mathcal{Y}^d}
(\ln \widetilde{m}_1-\ln \widetilde{m}_2 ) (\widetilde{m}_1-\widetilde{m}_2
)dy=0,
        \end{equation*}
        which implies $\widetilde{m}_1=\widetilde{m}_2$ and $\nabla_y\widetilde{w}_1=\nabla_y\widetilde{w}_2$.
Thus, using \eqref{PropCellDiffEq}, we see that  $\widetilde{H}_1=\widetilde{H}_2$.
Meanwhile, since  $\widetilde{w}_1, \widetilde{w}_2\in C^{2,\alpha}_{\#}(\mathcal{Y}^d)/\mathbb{R}$, $\nabla_y\widetilde{w}_1=\nabla_y\widetilde{w}_2$ 
implies that $\widetilde{w}_1=\widetilde{w}_2$. Therefore, we conclude that there exists at most one solution to Problem \ref{CellProbEulerLagrangeProb}.
\end{proof}

\subsubsection{A priori estimates} We say that solutions to
a PDE are classical if they have enough smoothness to solve the
PDE. To prove existence of the
solution to Problem~\ref{CellProbEulerLagrangeProb}, we use the continuation
method, which is similar to the argument in \cite{evans2003some}.
For that, we begin by assuming that Problem~\ref{CellProbEulerLagrangeProb}
 admits a classical solution, $ (\widetilde{w}, \widetilde{m},
\widetilde{H} )$. Then, we establish various uniform bounds for
$\widetilde{w}$, $\widetilde{m}$, and $\widetilde{H}$.

\begin{pro}
        \label{CoercivityWidetideH}
        Let $(\widetilde{w}, \widetilde{m}, \widetilde{H})$ solve
Problem~\ref{CellProbEulerLagrangeProb}. Then, 
        for
        any $x\in\mathbb{T}^d$ and  $\Lambda\in\mathbb{R}^d$, $\widetilde{H}$ is coercive in $\Lambda$; that is, 
        \begin{equation}
        \label{Coercivity}
        \frac{ |\Lambda |^2}{2}+\inf_{x,y}V (x,y )\leq \widetilde{H}
(x,\Lambda )\leq \sup_{x,y}V (x,y )+\frac{ |\Lambda |^2}{2},
        \end{equation}
        and
        \begin{equation}
        \label{BoundIntegralGradw}
        \int_{\mathcal{Y}^d} |\nabla_y\widetilde{w}(x,\Lambda,y) |^2 dy\leq
2\bigg(\sup_{x,y}V(x,y)-\inf_{x,y}V(x,y)\bigg). 
        \end{equation}
\end{pro}
\begin{proof}
As stated in Remark \ref{rmk:p7p4qui}, $\widetilde{H}$ in \eqref{WidetildeH} is the same as $\widetilde{H}$ in the solution to \eqref{CellProbEulerLagrangeProb}. Thus,
for each $x\in \mathbb{T}^d$ and $\Lambda\in \mathbb{R}^d$,        choosing ${w}=0$ in \eqref{MinCellProblem} and using
the formulation of $\widetilde{H}$ in \eqref{WidetildeH}, we
get 
        \begin{align}
        \label{HtildeUpperBounds}
        \widetilde{H} (x,\Lambda )\leq \ln\int_{\mathcal{Y}^d}e^{\frac{
|\Lambda |^2}{2}+V (x,y )}dy\leq \frac{ |\Lambda |^2}{2}+ \sup_{x,y}V
(x,y ).
        \end{align}
        Using Jensen's inequality and the periodicity of $\widetilde{w}$,
we obtain 
        \begin{align*}
        \widetilde{H} (x,\Lambda )&=\ln \int_{\mathcal{Y}^d}e^{\frac{
|\Lambda +\nabla_y \widetilde{w} (x,\Lambda, y ) |^2}{2}+V (x,y )}dy
        \geq \int_{\mathcal{Y}^d} \bigg(\frac{ |\Lambda+\nabla_y\widetilde{w}
(x,\Lambda, y ) |^2}{2}+V (x,y )\bigg )dy\\
        &\geq\frac{ |\Lambda |^2}{2}+\int_{\mathcal{Y}^d}\Lambda^T\nabla_y\widetilde{w}
(x,\Lambda, y )dy+\int_{\mathcal{Y}^d}\frac{ |\nabla_y\widetilde{w} (x,\Lambda, y
) |^2}{2}dy+\inf_{x,y}V (x,y )\\
        &=\frac{ |\Lambda |^2}{2}+\int_{\mathcal{Y}^d}\frac{
|\nabla_y\widetilde{w} (x,\Lambda, y ) |^2}{2}dy+\inf_{x,y}V (x,y )
        \geq  \frac{ |\Lambda |^2}{2}+\inf_{x,y} V (x,y ).
        \end{align*}
        Thus, the preceding estimate and \eqref{HtildeUpperBounds}
yield \eqref{Coercivity}. Furthermore, combining the second to
last equality in the preceding estimate with \eqref{HtildeUpperBounds},
we get
        \begin{equation*}
        \int_{\mathcal{Y}^d}\frac{ |\nabla_y\widetilde{w} |^2}{2}dy\leq
\widetilde{H}-\frac{ |\Lambda |^2}{2}-\inf_{x,y}V\leq  \sup_{x,y}V-\inf_{x,y}V.
        \end{equation*}
        Therefore, we conclude \eqref{BoundIntegralGradw}.
\end{proof}

The above estimates combined with the first equation of \eqref{CellProbEulerLagrangeEq}
immediately gives us a lower bound for $\widetilde{m}$. 
\begin{corollary}
        \label{Lowerboundmtilde}
        Let $(\widetilde{w}, \widetilde{m}, \widetilde{H})$ solve
Problem~\ref{CellProbEulerLagrangeProb}. Then, for any $x\in \mathbb{T}^d$ and $\Lambda\in \mathbb{R}^d$ and for any  $y
\in \mathcal{Y}^d$, 
        \begin{equation*}
        \widetilde{m} (x,\Lambda, y )\geq e^{\inf\limits_{x,y}V (x,y )-\sup\limits_{x,y}V
(x,y )-\frac{ |\Lambda |^2}{2}}.
        \end{equation*}
\end{corollary}
\begin{proof}
Using the first equation of \eqref{CellProbEulerLagrangeEq},
we get for any $x\in \mathbb{T}^d$ and  $\Lambda\in \mathbb{R}^d$ and for any  $y
\in \mathcal{Y}^d$,

\begin{equation*}
\widetilde{m} (x,\Lambda, y)=e^{\frac{ |\Lambda+\nabla_y\widetilde{w}
(x,\Lambda, y) |^2}{2}+V (x,y )-\widetilde{H} (x,\Lambda )}.
\end{equation*}
Using \eqref{Coercivity} and the boundedness of $V$, we get
\begin{equation*}
\widetilde{m} (x,\Lambda, y)\geq e^{V (x,y )-\widetilde{H} (x,\Lambda
)}\geq e^{\inf\limits_{x,y}V (x,y )-\sup\limits_{x,y}V (x,y )-\frac{
|\Lambda |^2}{2}}.\qedhere
\end{equation*}
\end{proof}
Next, we obtain an upper bound for $\widetilde{m}$. To do that,
we get an upper bound on the norm of $\widetilde{m}$ in $L^{\theta}$
for some $0<\theta<+\infty$. Then, we use Moser's argument to
bound $\widetilde{m}$ in $L^{p_1}$ by the norm of $\widetilde{m}$
in $L^{\theta}$ for all $p_1$ satisfying $\theta<p_1<+\infty$.
Finally, we consider the limit $p_1\rightarrow+\infty$ and conclude
that $\widetilde{m}$ is bounded. 

\begin{pro}
        \label{mbetabounds}
        Let $(\widetilde{w}, \widetilde{m}, \widetilde{H})$ solve
Problem~\ref{CellProbEulerLagrangeProb}. 
        Define $\theta=\frac{2}{d-2}$ if $d\geq 3$ and any positive
number if $d=2$ or $d=1$. Then, there exists a constant, $C$, independent of $x$ and $\Lambda$ such that 
        \begin{align}
        \label{Probmboundseq}
         \bigg(\int_{\mathcal{Y}^d}\widetilde{m}^{\theta+1}dy\bigg
)^{\frac{1}{\theta+1}}\leq C.
        \end{align}
\end{pro}
\begin{proof}
        We define
        \begin{equation}
        \label{DefK}
        K=\Lambda+\nabla_y\widetilde{w}
        \end{equation}
        and denote each component of $K$ by $K^i$, where $1\leq
i\leq d$. Let 
        \begin{equation*}
        \label{Defh1}
        h=\frac{ |\Lambda+\nabla_y\widetilde{w} |^2}{2}+V. 
        \end{equation*}
        Then, the second equation in \eqref{CellProbEulerLagrangeEq}
becomes $-\div_y (\widetilde{m}K )=0$. Next, we use Einstein's notation.  Multiplying both sides of $-\div_y (\widetilde{m}K )=0$ by $\Delta_y\widetilde{w}$ and integrating, we get
        \begin{align}
        \label{mbetaboundsEq1}
        \begin{split}
        0=&\int_{\mathcal{Y}^d} \div_y (\widetilde{m}K )\Delta_y\widetilde{w}
dy=\int_{\mathcal{Y}^d}  (\widetilde{m}K^i )_{y_i}\widetilde{w}_{y_jy_j}
dy=\int_{\mathcal{Y}^d}  (\widetilde{m}K^i )_{y_j}\widetilde{w}_{y_iy_j}
dy\\
        =&\int_{\mathcal{Y}^d}  (\widetilde{m}_{y_j}K^i+\widetilde{m}K^i_{y_j}
)\widetilde{w}_{y_iy_j} dy.
        \end{split}
        \end{align}
        From the first equation in  \eqref{CellProbEulerLagrangeEq},
we get  $\widetilde{m}_{y_j}=\widetilde{m}h_{y_j}$. Using this
identity in the last equality in \eqref{mbetaboundsEq1}, we obtain
        \begin{align*}
        0=\int_{\mathcal{Y}^d}  (\widetilde{m}h_{y_j}K^i+\widetilde{m}K^i_{y_j}
)\widetilde{w}_{y_iy_j} dy.
        \end{align*}
        Because  $h_{y_j}=K^i\widetilde{w}_{y_iy_j}+V_{y_j}$
and $K^i_{y_j}=\widetilde{w}_{y_iy_j}$, we have
        \begin{align*}
        \int_{\mathcal{Y}^d} \widetilde{m}h_{y_j}h_{y_j} dy+\int_{\mathcal{Y}^d}
\widetilde{m}\widetilde{w}_{y_iy_j}\widetilde{w}_{y_iy_j} dy=\int_{\mathcal{Y}^d}
\widetilde{m}h_{y_j}V_{y_j} dy.
        \end{align*}
        Using a weighted Cauchy's inequality and the smoothness
of $V$, we conclude that there exists a constant, $C$, independent of $x$ and $\Lambda$ such that
 
        \begin{align*}
        \int_{\mathcal{Y}^d}\widetilde{m} |\nabla_y h |^2dy\leq
C.
        \end{align*}
        Since $\nabla_y \widetilde{m}=\widetilde{m}\nabla_yh$
and $\int_{\mathcal{Y}^d}\widetilde{m}dy=1$, 
        \begin{align*}
        \int_{\mathcal{Y}^d}  \Big(\big |\nabla_y \widetilde{m}^{\frac{1}{2}}
\big|^2+ \big(\widetilde{m}^{\frac{1}{2}} \big)^2 \Big)dy=&\int_{\mathcal{Y}^d}
\Big(\frac{1}{4}\widetilde{m} |\nabla_yh |^2+\widetilde{m}\Big
)dy\leq C. 
        \end{align*}
        Using Sobolev's inequality, we obtain
        \begin{align}
        \label{uniformboundsmthetanorm}
         \bigg(\int_{\mathcal{Y}^d}\widetilde{m}^{\theta+1}dy\bigg
)^{\frac{1}{\theta+1}}\leq C\int_{\mathcal{Y}^d} \Big(\big |\nabla_y\widetilde{m}^{\frac{1}{2}}
\big|^2+ \big(\widetilde{m}^{\frac{1}{2}} \big)^2 \Big)dy\leq
C,
        \end{align}
        where $\theta=\frac{2}{d-2}$ for $d\geq 3$ and any positive
real number for $d=2$. If $d=1$, Morrey's inequality gives 
        \begin{equation}
        \label{UpperBoundofmtilde}
        \sup_y\widetilde{m}^\frac{1}{2}\leq C\int_{\mathcal{Y}}
\Big(\big |\nabla_y\widetilde{m}^{\frac{1}{2}}
\big|^2+ \big(\widetilde{m}^{\frac{1}{2}}
\big)^2 \Big)dy\leq C.
        \end{equation} 
        Thus, for $d=1$ and any $\theta>0$, we also have \eqref{uniformboundsmthetanorm}. Therefore, we conclude that \eqref{Probmboundseq} holds. 
\end{proof}
Next, we use Moser's iteration method  to obtain an upper bound of $\widetilde{m}$.

\begin{pro}
        \label{UpperBoundw}
        Let $(\widetilde{w}, \widetilde{m}, \widetilde{H})$ solve
Problem~\ref{CellProbEulerLagrangeProb}. Then, 
        there exists a constant, $C$, such that for any $x$ and
$\Lambda$, 
\begin{equation}
\label{UpperBoundwIneq}
\sup_y\widetilde{m} (x, \Lambda, y) \leq C
\end{equation}
and
\begin{equation}
\label{UpperBoundLambdaPlusGradw}
\sup_{y}{ |\Lambda+\nabla_y\widetilde{w} (x,\Lambda, y)|}\leq C+{ |\Lambda| }. 
\end{equation}
\end{pro}
\begin{proof}
First, we show that \eqref{UpperBoundwIneq} implies \eqref{UpperBoundLambdaPlusGradw}. Assume that there exists a constant, $C$, such that for any $x$ and $\Lambda$, \eqref{UpperBoundwIneq} holds. 
Using the first equation in \eqref{CellProbEulerLagrangeEq}, we have $\widetilde{m}=e^{\frac{ |\Lambda+\nabla_y\widetilde{w}
                |^2}{2}+V-\widetilde{H}}$, we obtain
\begin{equation*}
\frac{ |\Lambda+\nabla_y\widetilde{w} |^2}{2}+V-\widetilde{H}=\ln
\widetilde{m}\leq \ln C. 
\end{equation*}
Then, using Proposition~\ref{CoercivityWidetideH}, we
have 
\begin{equation*}
\sup_{y}{ |\Lambda+\nabla_y\widetilde{w} |}\leq C+{ |\Lambda
        |}.
\end{equation*}
Thus, we conclude that \eqref{UpperBoundLambdaPlusGradw} holds. 
Next, we prove \eqref{UpperBoundwIneq}. If $d=1$, \eqref{UpperBoundwIneq} follows by \eqref{UpperBoundofmtilde}
in the proof of Proposition~\ref{mbetabounds}. Thus, in what follows, we suppose
that $d\geq 2$. 

As before, we define $K$ as in \eqref{DefK} and use Einstein's notation.   Multiplying
the second  equation in \eqref{CellProbEulerLagrangeEq} by $-\div_y
(\widetilde{m}^qK )$ for $q\geq 0$ and integrating, we get 
        \begin{align}
        \label{TwoDivProduct}
        \begin{split}
        0=&\int_{\mathcal{Y}^d} \div_y (\widetilde{m}K )\div_y
(\widetilde{m}^qK )dy
=\int_{\mathcal{Y}^d}  (\widetilde{m}K^i )_{y_i} (\widetilde{m}^qK^j
)_{y_j} dy=\int_{\mathcal{Y}^d}  (\widetilde{m}K^i )_{y_j} (\widetilde{m}^qK^j
)_{y_i} dy\\
        =&\int_{\mathcal{Y}^d}  (\widetilde{m}_{y_j}K^i+\widetilde{m}K^i_{y_j}
) (q\widetilde{m}^{q-1}\widetilde{m}_{y_i}K^j+\widetilde{m}^qK^j_{y_i}
)dy\\
        =&\int_{\mathcal{Y}^d} q\widetilde{m}^{q-1}\widetilde{m}_{y_j}\widetilde{m}_{y_i}K^iK^j
dy+\int_{\mathcal{Y}^d} \widetilde{m}^{q+1}K^i_{y_j}K^{j}_{y_i}
dy\\
        &+\int_{\mathcal{Y}^d} (\widetilde{m}_{y_j}\widetilde{m}^qK^iK^j_{y_i}+q\widetilde{m}^{q}\widetilde{m}_{y_i}K^i_{y_j}K^j
)dy=:\int_{\mathcal{Y}^d} A_1 dy + \int_{\mathcal{Y}^d}
A_2 dy + \int_{\mathcal{Y}^d} A_3 dy, 
        \end{split}
        \end{align}
        where using \eqref{DefK}, 
        \begin{equation}
        \label{EstimateA}
        A_1=q\widetilde{m}^{q-1}\widetilde{m}_{y_j}\widetilde{m}_{y_i}K^iK^j=q\widetilde{m}^{q-1}
\big| (\nabla_y\widetilde{m} )^TK \big|^2,
        \end{equation}
        \begin{equation}
        \label{EstimateB}
        A_2=\widetilde{m}^{q+1}K^i_{y_j}K^j_{y_i}=\widetilde{m}^{q+1}\widetilde{w}_{y_iy_j}\widetilde{w}_{y_jy_i}=\widetilde{m}^{q+1}
|\nabla_y^2\widetilde{w} |^2,
        \end{equation}
        and 
        \begin{align}
        \label{EstimateC}
        \begin{split}
        A_3&=\widetilde{m}_{y_j}\widetilde{m}^qK^iK^j_{y_i}+
        q\widetilde{m}^{q}\widetilde{m}_{y_i}K^i_{y_j}K^j= (q+1 )\widetilde{m}^q\widetilde{m}_{y_j}K^iK^j_{y_i}\\
        &= (q+1 )\widetilde{m}^q (\widetilde{m}_{y_j}K^i\widetilde{w}_{y_jy_i}
)= (q+1 )\widetilde{m}^q  (\nabla_y \widetilde{m} )^T\nabla_y^2\widetilde{w}
(\Lambda + \nabla_y \widetilde{w} ).
        \end{split}
        \end{align}
        Combining \eqref{TwoDivProduct}--\eqref{EstimateC}, we
have  
        \begin{align}
        \label{IneqABC}
        \int_{\mathcal{Y}^d} \big(q\widetilde{m}^{q-1} | (\nabla_y\widetilde{m}
)^TK |^2+\widetilde{m}^{q+1} |\nabla_y^2\widetilde{w} |^2+ (q+1
)\widetilde{m}^q  (\nabla_y \widetilde{m} )^T\nabla_y^2\widetilde{w}
(\Lambda+ \nabla_y \widetilde{w} )\big )dy= 0.
        \end{align}
        Differentiating the first equation in  \eqref{CellProbEulerLagrangeEq}
with respect to $y$, we get
        \begin{equation*}
        \nabla_y^2\widetilde{w} (\Lambda+\nabla_y\widetilde{w}
)+\nabla_y V=\frac{\nabla_y \widetilde{m}}{\widetilde{m}}\cdot
        \end{equation*}
        Then, multiplying both sides in the prior equation by
$ (q+1 )\widetilde{m}^q\nabla_y\widetilde{m}$, we obtain 
        \begin{align*}
        & (q+1 )\widetilde{m}^q  (\nabla_y \widetilde{m} )^T\nabla_y^2\widetilde{w}
(\Lambda + \nabla_y \widetilde{w} )+ (q+1 )\widetilde{m}^q (\nabla_y\widetilde{m}
)^T\nabla_yV= (q+1 )\widetilde{m}^{q-1} |\nabla_y\widetilde{m} |^2.
        \end{align*}
        Combining the previous identity with  \eqref{IneqABC},
we get 
        \begin{align*}
        \begin{split}
        &\int_{\mathcal{Y}^d} \big(q\widetilde{m}^{q-1} | (\nabla_y\widetilde{m}
)^TK |^2+\widetilde{m}^{q+1} |\nabla_y^2\widetilde{w} |^2+ (q+1
)\widetilde{m}^{q-1} |\nabla_y\widetilde{m} |^2 \big)dy\\
       &\quad \leq   \int_{\mathcal{Y}^d} \big| (q+1 )\widetilde{m}^q
(\nabla_y\widetilde{m} )^T\nabla_yV\big |dy\\
        &\quad \leq \frac{q+1}{2}\int_{\mathcal{Y}^d} \widetilde{m}^{q-1}
|\nabla_y\widetilde{m} |^2dy+\frac{q+1}{2}\int_{\mathcal{Y}^d}
\widetilde{m}^{q+1} |\nabla_yV |^2 dy,
        \end{split}
        \end{align*}
        where we use Cauchy's inequality in the last inequality. Since the first two terms on the most left-hand side of the preceding inequalities are positive, we have
        \begin{align}
        \label{ABCIneq2}
        \begin{split}
        \int_{\mathcal{Y}^d} \widetilde{m}^{q-1} |\nabla_y\widetilde{m}
|^2 dy
        \leq \int_{\mathcal{Y}^d} \widetilde{m}^{q+1} |\nabla_yV
|^2dy\leq \sup_{x,y} |\nabla_yV |^2\int_{\mathcal{Y}^d} \widetilde{m}^{q+1}dy.
        \end{split}
        \end{align}
        Let $\theta=\frac{2}{d-2}$ if $d\geq 3$ and $\theta$
any positive number if $d=2$.  By Sobolev's inequality and \eqref{ABCIneq2},
there exists a constant, $C$, independent of $x$ and $\Lambda$ such that 
        \begin{align*}
         \bigg(\int_{\mathcal{Y}^d}\widetilde{m}^{ (q+1 ) (1+\theta
)}dy\bigg )^{\frac{1}{1+\theta}}\leq& C\int_{\mathcal{Y}^d} \Big(
\big|\nabla_y \big(\widetilde{m}^{\frac{q+1}{2}} \big) \big|^2+\widetilde{m}^{q+1}
\Big)dy\\
        \leq& C (q+1 )^2\int_{\mathcal{Y}^d} \widetilde{m}^{q-1}
|\nabla_y\widetilde{m} |^2 dy+C\int_{\mathcal{Y}^d}\widetilde{m}^{q+1}dy\\
        \leq &C (q+1 )^2\int_{\mathcal{Y}^d}\widetilde{m}^{q+1}dy.
        \end{align*}
        Let  $\beta= (1+\theta )^{\frac{1}{2}}>1$. Using H\"{o}lder's
inequality, we obtain
        \begin{equation*}
         \bigg(\int_{\mathcal{Y}^d}\widetilde{m}^{ (q+1 )\beta^2}dy\bigg
)^{\frac{1}{\beta^2}}\leq C (q+1 )^2\int_{\mathcal{Y}^d}\widetilde{m}^{q+1}dy\leq
C (q+1 )^2 \bigg(\int_{\mathcal{Y}^d}\widetilde{m}^{ (q+1 )\beta}dy
\bigg)^{\frac{1}{\beta}}.
        \end{equation*}
        Thus, letting
        \begin{equation}
        \label{DefMoserGamma}
        \gamma= (q+1 )\beta\geq \beta,
        \end{equation} we get
        \begin{equation}
        \label{MoserIterationBase}
         \bigg(\int_{\mathcal{Y}^d}\widetilde{m}^{\gamma\beta}dy\bigg
)^{\frac{1}{\beta}}\leq C \gamma^{2\beta}\int_{\mathcal{Y}^d}\widetilde{m}^\gamma
dy.
        \end{equation}
        Then, we use Moser's method. We choose a sequence $(q_s)_s$
such that $q_1=0$ and $q_{s+1}= (q_s+1 )\beta-1$ for $s\in\mathbb{N}$
and $s\geq 1$ and let $\gamma_s=(q_s+1)\beta$. Then, by \eqref{DefMoserGamma},
 $\gamma_{s+1}=\beta (q_{s+1}+1 )=\beta (q_s+1 )\beta=\gamma_s\beta$.
Hence, $\gamma_s=\beta^s$.  Using \eqref{MoserIterationBase},
we have
        \begin{equation*}
         \bigg(\int_{\mathcal{Y}^d}\widetilde{m}^{\gamma_{s+1}}dy\bigg
)^{\frac{1}{\gamma_{s+1}}}\leq C^{\frac{1}{\gamma_s}}\gamma_s^{\frac{2\beta}{\gamma_s}}
\bigg(\int_{\mathcal{Y}^d}\widetilde{m}^{\gamma_s}dy\bigg )^{\frac{1}{\gamma_s}}.
        \end{equation*} 
        Iterating the preceding inequality, we get
        \begin{equation*}
        \|\widetilde{m}\|_{L^{\gamma_{s+1}}(\mathcal{Y}^d)} \leq C^{\sum_{t=2}^{s}\frac{1}{\beta^t}}\beta^{\sum_{t=2}^{s}\frac{2t}{\beta^{t-1}}}
        \bigg(\int_{\mathcal{Y}^d}\widetilde{m}^{\beta^2} dx\bigg )^{\frac{1}{{\beta^2}}}. 
        \end{equation*}
         Letting $s\rightarrow\infty$ and using \(\beta=(1+\theta)^{\frac{1}{2}}\) and Proposition
\ref{mbetabounds}, we get
        \begin{equation*}
         \|\widetilde{m} \|_{L^\infty (\mathcal{Y}^d )}\leq C^{\sum_2^{\infty}\frac{1}{\beta^s}}\beta^{\sum_2^{\infty}\frac{2s}{\beta^{s-1}}}
\bigg(\int_{\mathcal{Y}^d}\widetilde{m}^{\beta^2} dx\bigg )^{\frac{1}{\beta^2}}\leq
C.
        \end{equation*}
Therefore, we conclude that \eqref{UpperBoundwIneq} holds. 
\end{proof}
Next, we examine the H\"{o}lder continuity of $\widetilde{w}$.
To do that, we consider the regularity of $\nabla_y\widetilde{w}$.
First, we rewrite \eqref{CellProbEulerLagrangeEq} as 
\begin{align}
\label{DivForm}
-\div_y \Big(e^{\frac{ |\Lambda+\nabla_y\widetilde{w} |^2}{2}+V}
(\Lambda+\nabla_y\widetilde{w} )\Big )=0.
\end{align} 
We stress that \eqref{CellProbEulerLagrangeEq} and \eqref{DivForm} are equivalent. More precisely, 
suppose that $\widetilde{w}$ satisfies \eqref{DivForm}, we define for any $x\in \mathbb{T}^d$ and  $\Lambda\in \mathbb{R}^d$ 
\begin{equation}
\label{EllipticWidetideH}
\widetilde{H}(x,\Lambda)=\ln \int_{\mathcal{Y}^d} e^{\frac{\left|\Lambda+\nabla_y(x,y)\widetilde{w}\right|^2}{2}+V(x,y)}dy
\end{equation}
and for any $y\in \mathcal{Y}^d$, 
\begin{equation}
\label{EllipticWidetidem}
\widetilde{m}(x,\Lambda, y)=e^{\frac{\left|\Lambda+\nabla_yw(x,\Lambda, y)\right|^2}{2}+V(x,y)}.
\end{equation}
Then, $(\widetilde{w}, \widetilde{m}, \widetilde{H})$ solves
Problem~\ref{CellProbEulerLagrangeProb}.

Denote $v=\widetilde{w}_{y_l}$, where $l=1,\dots,d$. Differentiating
\eqref{DivForm} with respect to $y_l$, we get
\begin{align}
\label{EllipticPDEofV}
\begin{split}
&-\div_y \Big(e^{\frac{ |\Lambda+\nabla_y\widetilde{w} |^2}{2}+V}
\big( (\Lambda+\nabla_y\widetilde{w} ) (\Lambda+\nabla_y\widetilde{w}
)^T+I \big)\nabla_y v \Big)\\
&\quad=\div_y \Big(e^{\frac{ |\Lambda+\nabla_y\widetilde{w} |^2}{2}+V}V_{y_l}
(\Lambda+\nabla_y\widetilde{w} )\Big ),
\end{split}
\end{align}
where $I$ is the Identity matrix. 
For simplicity, we denote $$A=e^{\frac{ |\Lambda+\nabla_y\widetilde{w}
|^2}{2}+V} \big( (\Lambda+\nabla_y\widetilde{w} ) (\Lambda+\nabla_y\widetilde{w}
)^T+I\big )$$ and  $\phi=e^{\frac{ |\Lambda+\nabla_y\widetilde{w}
|^2}{2}+V}V_{y_l} (\Lambda+\nabla_y\widetilde{w} )$. Then, \eqref{EllipticPDEofV}
becomes
\begin{equation}
\label{SimplifiedEllipticPDEofV}
-\div_y (A\nabla_yv )=\div_y\phi.
\end{equation}
By the definition of $A$, for any $\xi\in\mathbb{R}^d$, we have

\begin{align*}
\xi^TA\xi&=e^{\frac{ |\Lambda+\nabla_y\widetilde{w} |^2}{2}+V}\xi^T\big
( (\Lambda+\nabla_y\widetilde{w} ) (\Lambda+\nabla_y\widetilde{w}
)^T+I\big )\xi\\
&=e^{\frac{ |\Lambda+\nabla_y\widetilde{w} |^2}{2}+V} \big| (\Lambda+\nabla_y\widetilde{w}
)^T\xi\big |^2+e^{\frac{ |\Lambda+\nabla_y\widetilde{w} |^2}{2}+V}
|\xi |^2\geq e^{\inf V} |\xi |^2.
\end{align*}
Let $\eta=e^{\inf V}$. The preceding expressions yield 
\begin{align}
\label{EllipticityofA}
\eta |\xi |^2\leq \xi^TA\xi, \ \forall \xi\in\mathbb{R}^d.
\end{align} 
From Proposition~\ref{UpperBoundw}, we know that there exists
a positive constant $c_\Lambda$ independent on $x$ such that
\begin{equation}
\label{BoundednessPhi}
 \|\phi \|_{L^\infty (\mathcal{Y}^d )}\leq c_{\Lambda}
\end{equation}
and 
\begin{align}
\label{BoundsofA}
 \|A \|_{L^\infty (\mathcal{Y}^d )}\leq c_{\Lambda}.
\end{align} 
In the rest of this section,  $C_{\Lambda}$ denotes any positive
real number depending on $\Lambda$ and $c_\Lambda$ and independent on $x$, whose value
may change from one expression to another and is  uniformly bounded
in $\Lambda$ on compact sets; that is, if $(\Lambda_n)_n$ is
a bounded sequence, so is $ (C_{\Lambda_n} )_n$.

To get the regularity of $v$ in \eqref{SimplifiedEllipticPDEofV}, we consider  \eqref{SimplifiedEllipticPDEofV}
restricted to a small ball, $B (y_0, R )$, with a radius $R>0$
around the point, $y_0\in\mathcal{Y}^d$. After obtaining estimates
on $B (y_0, R )$, the compactness of $\mathcal{Y}^d$ implies bounds on the whole $\mathcal{Y}^d$. First, we split $v$
into two parts.  
Let $v=\psi+z$, where $\psi$ solves
\begin{equation}
\label{EqPsi}
-\div_y (A\nabla_y\psi )=\div_y\phi, \ \forall y\in B (y_0,2R
)
\end{equation}
and $\psi=0$ on the boundary of $B (y_0,2R )$, denoted by $\partial B (y_0,2R )$, and $z$ solves
\begin{equation}
\label{Eqz}
-\div_y (A\nabla_yz )=0, \ \forall y\in B (y_0,2R )
\end{equation}
and $z |_{\partial B (y_0, 2R )}=v  |_{\partial B (y_0, 2R )}$.
To get the boundedness of $\psi$ and the oscillation of $z$ and
$v$,  we proceed as in the proof of Theorem 8.13 in \cite{giaquinta2013introduction}.

\begin{pro}
        \label{BoundPsiLinfty}
        Let $\psi$ be a weak solution to \eqref{EqPsi}. Suppose that $A$ satisfies
\eqref{EllipticityofA} and \eqref{BoundsofA} and that $\phi$
satisfies \eqref{BoundednessPhi}. Then,
        there exist a constant, $C_\Lambda$, and a number, $\delta>0$,
such that
        \begin{equation*}
         \|\psi \|_{L^\infty (B (y_0,2R ) )}\leq C_\Lambda R^\delta,
        \end{equation*}
        where $\delta=1$ if $d>2$ and $0<\delta<1$ if $d=2$ or
$d=1$. 
\end{pro}
\begin{proof}
        For a function $\varphi$, we denote by $\varphi^+$ the
nonnegative part of $\varphi$. 
        Multiplying by  \eqref{EqPsi} $ (\psi-k )^+$, where $k\in
\mathbb{N}$, and integrating the resulting equation over $B (y_0,2R
)$, we get
        \begin{equation*}
        \int_{B (y_0,2R )}  (\nabla_y\psi )^TA^T\nabla_y ( (\psi-k
)^+ )dy=\int_{B (y_0,2R )}-\phi^T\nabla_y ( (\psi-k )^+ )dy,
        \end{equation*}
taking into account that $(\psi-k)^+=0$ on $\partial B(y_0, 2R)$. 
        Thus, 
        \begin{equation}
        \label{BoundPsiLinfityEq1}
        \int_{B (y_0,2R )\cap \{\psi>k \}}  (\nabla_y\psi )^TA^T\nabla_y\psi
dy=\int_{B (y_0,2R )\cap \{\psi>k \}}-\phi^T\nabla_y\psi dy.
        \end{equation}
        For simplicity, we denote $$E (k )=B (y_0,2R )\cap \{\psi>k
\}.$$
        Using \eqref{EllipticityofA}, \eqref{BoundednessPhi},
and \eqref{BoundPsiLinfityEq1}, we get
        \begin{align*}
        \eta\int_{E (k )} |\nabla_y\psi |^2dy\leq& \int_{E (k
)}  (\nabla_y\psi )^TA^T\nabla_y\psi dy=\int_{E (k )}-\phi^T\nabla_y\psi
        \leq c_\Lambda\int_{E (k )} |\nabla_y\psi |dy.
        \end{align*}
        Then, by Cauchy's inequality, we obtain 
        \begin{align}
        \label{BoundDPsiBySmallBallArea}
        \int_{E (k )} |\nabla_y\psi |^2dy \leq  \frac{c_\Lambda^2}{\eta^2}
 |E (k ) |.
        \end{align}
        Let $\theta=\frac{1}{\frac{1}{2}-\frac{1}{d}}$ if $d>2$
and $\theta$ any positive number greater than $2$ if $d=2$ or
$d=1$.  For any $a>k$, there exists a constant $C_1>0$, independent on $R$, $x$, and $\Lambda$, such that

        \begin{align}
        \label{BoundSmallBallAreaByDPsi}
        \begin{split}
         (a-k )^2 |E (a ) |^{2/\theta}=&  \bigg(\int_{E (a
)} (a-k )^{\theta}dy\bigg )^{2/\theta}\leq  \bigg(\int_{E (a
)} ( (\psi-k )^+ )^{\theta}dy\bigg )^{2/\theta}\\
        \leq &  \bigg(\int_{B (y_0,2R )} ( (\psi-k )^+ )^{\theta}dy\bigg
)^{2/\theta} \leq C_1\int_{B (y_0,2R )} \big|\nabla_y (\psi-k
)^+ \big|^2dy\\
        = & C_1\int_{E (k )} |\nabla_y\psi |^2dy.
        \end{split}
        \end{align}
        The second to last inequality in the prior expressions
follows by Sobolev's inequality.
        Combining  \eqref{BoundDPsiBySmallBallArea} and \eqref{BoundSmallBallAreaByDPsi},
we get
        \begin{equation*}
         |E (a ) |\leq \frac{C_1^\frac{\theta}{2} c_\Lambda^\theta
|E (k ) |^{\frac{\theta}{2}}}{\eta^\theta(a-k )^{\theta}}.
        \end{equation*}
        We denote $\widetilde{C}_\Lambda=C_1^\frac{\theta}{2}c_\Lambda^\theta/\eta^\theta$
and define 
        \begin{equation}
        \label{DefM}
        M= \Big(\widetilde{C}_\Lambda |E (0 ) |^{\frac{\theta}{2}-1}2^{\frac{\theta^2}{\theta-2}}
\Big)^{\frac{1}{\theta}}.
        \end{equation}
Furthermore, let $ (k_n )_n$ be a sequence such that
\begin{equation}
\label{Defkn}
k_n=M (1-\frac{1}{2^n} ).
\end{equation}
        Then, we obtain 
        \begin{align}
        \label{BoundEn1byEkn}
         |E (k_{n+1} ) |\leq \frac{\widetilde{C}_\Lambda}{ (k_{n+1}-k_{n}
)^\theta} |E (k_n ) |^{\frac{\theta}{2}}\leq \widetilde{C}_\Lambda\frac{2^{\theta
(n+1 )}}{M^{\theta}} |E (k_n ) |^{\frac{\theta}{2}}.
        \end{align}
        We claim that
        \begin{equation}
        \label{BoundEknbyE0}
         |E (k_n ) |\leq \frac{ |E (0 ) |}{2^{n\nu}},
        \end{equation}
        where $\nu=\frac{\theta}{\frac{\theta}{2}-1}$. We  prove
this claim by induction. For $n=0$, $k_0=0$. Hence, \eqref{BoundEknbyE0}
is trival. Suppose that \eqref{BoundEknbyE0} holds for some $n\in\mathbb{N}$,
$n>0$. Then, using \eqref{DefM}, \eqref{BoundEn1byEkn}, and \eqref{BoundEknbyE0},
we get          
\begin{equation*}
\begin{aligned}
|E (k_{n+1} ) |&\leq  \widetilde{C}_\Lambda\frac{2^{\theta
(n+1 )}}{M^{\theta}} |E (k_n ) |^{\frac{\theta}{2}}
        \leq   \widetilde{C}_\Lambda\frac{2^{\theta (n+1 )}}{M^{\theta}}
\frac{ |E (0 ) |^\frac{\theta}{2}}{2^{n\nu\theta/2}}\\ &=\widetilde{C}_\Lambda2^{\theta
(n+1 )-n\nu\theta/2}  |E (0 ) |^\frac{\theta}{2}\frac{1}{\widetilde{C}_\Lambda
|E (0 ) |^{\frac{\theta}{2}-1}2^{\frac{\theta^2}{\theta-2}}}=    |E (0 ) |2^{- (n+1 )\nu}.
\end{aligned}
\end{equation*}

        Thus, we conclude \eqref{BoundEknbyE0}. As $n\rightarrow
+\infty$, we see from \eqref{Defkn} that  $k_n\rightarrow M$. Hence, letting $n\rightarrow
+\infty$ in \eqref{BoundEknbyE0}, we get    $ |E (M ) |=0$. From
\eqref{DefM}, there exists a constant $C_\Lambda$ such that 
        $$M= {C}_\Lambda |E (0 ) |^{\frac{1}{2}-\frac{1}{\theta}}\leq
{C}_\Lambda |B (y_0,2R ) |^{\frac{1}{2}-\frac{1}{\theta}}\leq
{C}_\Lambda R^{d (\frac{1}{2}-\frac{1}{\theta} )}.$$
        Therefore,      letting $\delta=d (\frac{1}{2}-\frac{1}{\theta}
)$, we have 
\begin{equation*}
\begin{aligned}
 \|\psi \|_{L^\infty (B (y_0,2R ) )}\leq C_\Lambda R^\delta.
\end{aligned}\qedhere
\end{equation*}
\end{proof}
Next, we bound the oscillation of the solution, $z$,  satisfying
\eqref{Eqz}.  
\begin{pro}
\label{Estimatezosc}
        Let $z$ solve \eqref{Eqz}. Then, 
        there exists a constant, $0<\rho<1$, such that $$\operatorname{osc}
  \Big(y_0,\frac{R}{2},z \Big)\leq \rho \operatorname{osc} (y_0,R,z
),$$
        where 
        $$\operatorname{osc} (y_0,R,z )=\sup_{B (y_0,R )}z-\inf_{B
(y_0,R )}z.$$
\end{pro}
\begin{proof}
The claim follows directly from the  DeGiorgi--Nash--Moser estimate
\cite[Theorem~8.22]{GilTru}.
\end{proof}
\begin{corollary}
        \label{Boundoscv}
        Let $v$ solve \eqref{SimplifiedEllipticPDEofV} and $z$
solve \eqref{Eqz}. Then, there exists a constant $C_\Lambda>0$
such that 
        \begin{equation}
        \label{Boundoscveq}
        \operatorname{osc} \Big(y_0,\frac{R}{4},v\Big )\leq C_\Lambda
R^\delta+\operatorname{osc} \Big(y_0,\frac{R}{4},z\Big )\leq
C_\Lambda R^\delta+\rho \operatorname{osc} (y_0,R,v ),
        \end{equation}
        where $\delta$ is given in Proposition~\ref{BoundPsiLinfty}.

\end{corollary}
\begin{proof}
Since $v=\psi+z$, we have
\begin{align*}
osc \Big(y_0,\frac{R}{4}, v\Big )=&\sup_{B (y_0,\frac{R}{4} )}
(\psi+z )-\inf_{B (y_0,\frac{R}{4} )} (\psi+z )\\
\leq& \sup_{B (y_0,\frac{R}{4} )}\psi + \sup_{B (y_0,\frac{R}{4}
)} z - \inf_{B (y_0,\frac{R}{4} )} \psi - \inf_{B (y_0,\frac{R}{4}
)}z.\\
\end{align*}
By Proposition~\ref{BoundPsiLinfty}, there exists a positive
constant $C_\Lambda$ such that 
\begin{align*}
\sup_{B (y_0,\frac{R}{4} )}\psi - \inf_{B (y_0,\frac{R}{4} )}
\psi \leq 2 \|\psi \|_{L^\infty (B (y_0,2R ) )}\leq C_\Lambda
R^\delta.
\end{align*}
Using Proposition~\ref{Estimatezosc}, we obtain 
$$\operatorname{osc} \Big(y_0,\frac{R}{4},z\Big )\leq \rho \operatorname{osc}
(y_0,R,z ).$$
Therefore, the above estimates give  \eqref{Boundoscveq}. 
\end{proof}
\begin{pro}
        \label{BoundOSCv}
        Let $v$ solve \eqref{SimplifiedEllipticPDEofV}. Then,   there exist constants $C_\Lambda>0$ and $0<\beta<1$ such
that for any $0<R<1$,
        \begin{equation*}
        \operatorname{osc} (y_0,R,v )\leq C_\Lambda R^\beta.
        \end{equation*}
\end{pro}

\begin{proof}
        Let $0<\beta<\min \big\{-\frac{\ln\rho}{\ln 4}, \delta
 \big\}$, where $\delta$ is given in Proposition~\ref{BoundPsiLinfty}
and $\rho$ in Proposition~\ref{Estimatezosc}. For $n\in \mathbb{N}$,
we set  
        \begin{align}
        \label{DefTn}
        T_n=\sup_{\frac{1}{4^{n+1}}\leq r\leq \frac{1}{4^n}}\frac{\operatorname{osc}
(y_0,r,v )}{r^\beta}.
        \end{align}
Let
\begin{equation*}
1<\gamma<\min \Big\{4^{\delta-\beta},\frac{1}{4^\beta\rho}
\Big\}
\end{equation*}
and $\widetilde{C}_\Lambda$ be the constant given in Corollary \ref{Boundoscv}; that is,
\begin{equation}
\label{Boundoscveqcp}
\operatorname{osc} \Big(y_0,\frac{R}{4},v\Big )\leq
\widetilde{C}_\Lambda R^\delta+\rho \operatorname{osc} (y_0,R,v ).
\end{equation}
Then, we claim that there exists a large number $T>0$ such that
        \begin{equation}
        \label{BoundTn}
        T_n\leq T\gamma^{-n}.
        \end{equation}
        We prove \eqref{BoundTn} by induction. 

        Because the prior choice of $\gamma$ implies  $4^\beta\rho\gamma<1$,
there exists a real number $T$ satisfying $$\max \big\{4^{\frac{1}{2}+\beta}
\|v \|_{L^\infty (\mathcal{Y}^d )}, \widetilde{C}_\Lambda 4^\delta+\rho 4^\beta
T\gamma \big\}<T.$$
        Thus, for any $n\in \mathbb{N}$, we have  
        $$\widetilde{C}_\Lambda 4^\delta\Big (\frac{\gamma}{4^{\delta-\beta}}
\Big)^{n+1}+\rho 4^\beta T\gamma<T.$$
        When $n=0$, we have
        \begin{equation*}
        T_0=\sup_{\frac{1}{4}\leq r\leq 1}\frac{\operatorname{osc}
(y_0,r,v )}{r^\beta}\leq 4^\beta \operatorname{osc}  (y_0,1,v
)\leq 4^{\frac{1}{2}+\beta} \|v \|_{L^\infty (\mathcal{Y}^d )}<
T.
        \end{equation*}
        Next, we assume that \eqref{BoundTn} holds for some $n>0$.
Then, using \ref{Boundoscveqcp}, we get
        \begin{align*}
        T_{n+1}&=\sup_{\frac{1}{4^{n+2}}\leq r\leq \frac{1}{4^{n+1}}}\frac{\operatorname{osc}
(y_0,r,v )}{r^\beta}
        \leq \sup_{\frac{1}{4^{n+1}}
                \leq r\leq \frac{1}{4^{n}}}4^\beta\frac{\widetilde{C}_\Lambda
r^{\delta}+\rho  \operatorname{osc} (y_0,r,v )}{r^\beta}\\
        &\leq  \widetilde{C}_\Lambda 4^{\beta} \Big(\frac{1}{4^n} \Big)^{\delta-\beta}+4^\beta
T_n\rho \leq \widetilde{C}_\Lambda 4^{\beta} \Big(\frac{1}{4^n} \Big)^{\delta-\beta}+4^\beta
T\gamma^{-n}\rho\\
        &=  \big(\widetilde{C}_\Lambda 4^\delta \Big(\frac{\gamma}{4^{\delta-\beta}}
\Big)^{n+1}+\rho 4^\beta T \gamma \Big)\gamma^{- (n+1 )} \leq
T\gamma^{- (n+1 )}.
        \end{align*}
        Thus, we conclude that \eqref{BoundTn} holds. For any
$0<R<1$, we can find  $n\in\mathbb{N}$ such that $\frac{1}{4^{n+1}}\leq
R\leq \frac{1}{4^n}$. Hence, using the definition of $T_n$ in
\eqref{DefTn}, we obtain
        \begin{equation*}
        \operatorname{osc} (y_0,R,v )\leq R^\beta T_n\leq R^\beta
T \gamma^{-n}\leq TR^{\beta}.\qedhere
        \end{equation*}
\end{proof}
\begin{pro}
        \label{Holderwtilde}
        Let $\widetilde{w}$ solve \eqref{DivForm}. Then, 
        $\widetilde{w}$ is H\"{o}lder continuous. More precisely,
there exists a positive constant, $C_\Lambda$, such that  
        \begin{align*}
         \|\widetilde{w} \|_{C^{2,\beta}_\# (\mathcal{Y}^d )}\leq
C_\Lambda, 
        \end{align*}
        where $\beta$ is given in Proposition~\ref{BoundOSCv}.
\end{pro}
\begin{proof}
        By Proposition~\ref{BoundOSCv}, we conclude that for
any compact subset  $\mathcal{Y}^d_1$ of $\mathcal{Y}^d$, there
exists a constant $C_\Lambda$ such that  $ \|v \|_{C^{0,\beta}_\#
({\mathcal{Y}^d_1} )}\leq C_\Lambda$, where $v=\widetilde{w}_{y_l}$
for $l=1,\dots,d$.  Then, $ \|\nabla_y \widetilde{w} \|_{C^{0,\beta}_\#
(\mathcal{Y}^d )}\leq C_\Lambda$. Besides, since \eqref{SimplifiedEllipticPDEofV}
is  uniformly elliptic, Schauder's estimate gives that $ \|v
\|_{C^{1,\beta}_\# (\mathcal{Y}^d )}\leq C_\Lambda$. Then, $
\|\widetilde{w} \|_{C^{2,\beta}_\# (\mathcal{Y}^d )}\leq C_\Lambda$.
\end{proof}
\subsubsection{The existence of the solution to the cell problem}
Here, we use a continuation argument similar to the one in Chapter 11.3 of
\cite{GPV} to prove existence of the solution  to \eqref{DivForm}.
The key difference is that we work in H\"{o}lder spaces instead
of Sobolev spaces. Let $0<\alpha<\beta<1$, where $\beta$ is as
in Proposition~\ref{Holderwtilde}. Thus, $C^{2,\beta}_\# (\mathcal{Y}^d
)$ is compactly embedded in $C^{2,\alpha}_\# (\mathcal{Y}^d )$,
denoted by $C^{2,\beta}_\# (\mathcal{Y}^d )\subset\subset C^{2,\alpha}_\#
(\mathcal{Y}^d )$. We define $F: C^{2,\alpha}_\# (\mathcal{Y}^d
)\times[0,1]\times\mathbb{T}^d\times\mathbb{R}^d\rightarrow C^{0,\alpha}_\#
(\mathcal{Y}^d )$ by
\begin{equation}
\label{DefFFunc}
F (\widetilde{w}, \lambda, x,  \Lambda )=\div \Big(e^{\frac{
|\Lambda+\nabla_y\widetilde{w} |^2}{2}+\lambda V (x,y )} (\Lambda+\nabla_y\widetilde{w}
)\Big ). 
\end{equation} 
For fixed $x\in\mathbb{T}^d$ and $\Lambda\in\mathbb{R}^d$, we
define 
\begin{equation}
\label{DefS}
\mathcal{S}= \{\lambda\in [0,1 ] :  \exists \widetilde{w}\in
 C^{2,\alpha}_\# (\mathcal{Y}^d ) \ \text{such that}\  F (\widetilde{w},
 \lambda, x, \Lambda )=0  \}.
\end{equation}
If $\mathcal{S}$ is not empty and both open and closed in $ [0,1
]$,
we have solutions for all  $\lambda\in [0,1 ]$. Thus, $F (\widetilde{w},
1, x, \Lambda )=0$ admits a solution in $C^{2,\alpha}_\# (\mathcal{Y}^d
)$, which in return prove the existence of a solution to Problem \ref{CellProbEulerLagrangeProb}.

Clearly, when $\lambda=0$, for any $x\in\mathbb{T}^d$ and $\Lambda\in\mathbb{R}^d$,
$\widetilde{w}\equiv 0$ solves $F=0$. Thus, $\mathcal{S}$ is
not empty. Next, we prove $\mathcal{S}$ is open. 

Let $\mathcal{L}:C^{2,\alpha}_\# (\mathcal{Y}^d )\rightarrow
C^{0,\alpha}_\# (\mathcal{Y}^d )$ be the linearized operator
of $F$ with respect to $\widetilde{w}$. For $v\in C^{2,\alpha}_\#
(\mathcal{Y}^d )$, we have
\begin{equation}
\label{DefL}
\mathcal{L} (v )=\div \Big(e^{\frac{ |\Lambda+\nabla_y\widetilde{w}
|^2}{2}+\lambda V (x,y )} \big( (\Lambda+\nabla_y\widetilde{w}
)^T\nabla_yv (\Lambda+\nabla_y\widetilde{w} )+\nabla_y v\big
) \Big).
\end{equation}
It is sufficient to prove that $\mathcal{L}$ is invertible. We
define 
\begin{equation*}
\mathcal{H}= \bigg\{v\in W^{1,2}_\# (\mathcal{Y}^d ): \int_{\mathcal{Y}^d}
v dy = 0\bigg  \}
\end{equation*}
and endow $\mathcal{H}$ with the norm
\begin{equation*}
 \|v \|_{\mathcal{H}}^2= \|v \|^2_{W^{1,2}_\# (\mathcal{Y}^d
)}.
\end{equation*}
For $v, \varpi\in \mathcal{H}$, and $\widetilde{m}\in\mathcal{S}$,
 we define $B: \mathcal{H}\times\mathcal{H}\rightarrow\mathbb{R}$
as 
\begin{equation}
\label{DefB}
B[v,\varpi]=\int_{\mathcal{Y}^d}  e^{\frac{ |\Lambda+\nabla_y\widetilde{w}
|^2}{2}+\lambda V} \big( (\Lambda+\nabla_y\widetilde{w}
)^T\nabla_yv (\Lambda+\nabla_y\widetilde{w} )+\nabla_y v )\big
)^T\nabla_y \varpi dy.
\end{equation}
\begin{pro}
        Let $B$ be the bilinear form in \eqref{DefB}. Then, 
        $B$ is bounded; that is, for all $v, \varpi\in \mathcal{H}$,
there exists a positive number $C_\Lambda$ such that 
        \begin{equation*}
         |B[v,\varpi] |\leq C_\Lambda \|v \|_{\mathcal{H}} \|\varpi
\|_{\mathcal{H}}. 
        \end{equation*}
\end{pro}
\begin{proof}
        Using Proposition  \ref{UpperBoundw} and  H\"{o}lder's
inequality, there exists a constant $C_\Lambda$ such that  
        \begin{equation*}
         |B [v,\varpi ] |\leq C_\Lambda\int_{\mathcal{Y}^d} 
|\nabla_yv | |\nabla_y\varpi |dy\leq C_\Lambda \|v \|_\mathcal{H}
\|\varpi \|_\mathcal{H}.\qedhere
        \end{equation*} 
\end{proof}
Thus, by the Riesz Representation theorem, there exists a linear
continuous injective mapping, $A: \mathcal{H}\rightarrow\mathcal{H}$,
such that, for all $v,\varpi\in\mathcal{H}$,
\begin{equation}
\label{DefA}
B[v, \varpi]= \langle Av, \varpi \rangle_\mathcal{H}.
\end{equation}
\begin{pro}
\label{LowerBoundA}
        Let $A$ be the operator defined in \eqref{DefA}. Then,
        there exists a constant $C_\Lambda>0$, such that 
        \begin{equation}
        \label{LowerBoundAEq}
        \|A\varpi \|_\mathcal{H}\geq C_\Lambda \|\varpi \|_\mathcal{H}
        \end{equation}
        for all $\varpi\in\mathcal{H}$. 
\end{pro} 
\begin{proof}
        We prove the result by contradiction. Suppose that \eqref{LowerBoundAEq}
does not hold. Then, there exists a sequence, $ (\varpi_n )_n$,
in $\mathcal{H}$ such that $ \|\varpi_n \|_\mathcal{H}=1$ and
$A\varpi_n\rightarrow 0$. Then, 
        \begin{align*}
         \| A\varpi_n \|_\mathcal{H}&=B[\varpi_n,\varpi_n]\\
        &=\int_{\mathcal{Y}^d}  e^{\frac{ |\Lambda+\nabla_y\widetilde{w}
|^2}{2}+\lambda V} \big( (\Lambda+\nabla_y\widetilde{w}
)^T\nabla_y\varpi_n (\Lambda+\nabla_y\widetilde{w} )+\nabla_y
\varpi_n )\big )^T\nabla_y \varpi_n dy\\
        &=\int_{\mathcal{Y}^d}e^{\frac{ |\Lambda+\nabla_y\widetilde{w}
|^2}{2}+\lambda V} \big( \big( (\Lambda+\nabla_y\widetilde{w}
)^T\nabla_y\varpi_n \big)^2+ |\nabla_y\varpi_n |^2 \big)dy\rightarrow
0. 
        \end{align*}
        Since $V$ is smooth, $e^{\frac{ |\Lambda+\nabla_y\widetilde{w}
|^2}{2}+\lambda V (x,y )}$ does not vanish. Thus, $\int_{\mathcal{Y}^d}
|\nabla_y\varpi_n |^2 dy\rightarrow 0$. Since $\int_{\mathcal{Y}^d}{\varpi}_ndy=0$, by Poincar\'{e}'s inequality,
we conclude that  $\varpi_n\rightarrow 0$ in $\mathcal{H}$, which
contradicts with $ \|\varpi_n \|_\mathcal{H}=1.$
\end{proof}
\begin{corollary}
\label{ClosedRA}
        Let $A$ be as in \eqref{DefA}. Then, the range of $A$,
        $R (A )$, is closed in $\mathcal{H}$.
\end{corollary}
\begin{proof}
        Let $ (v_n )_n\subset R (A )$ be a Cauchy sequence and
$ (\varpi_n )_n\subset A$ such that $A\varpi_n=v_n$. By Proposition
\ref{LowerBoundA}, we know that $\varpi_n$ is Cauchy. Thus, there
exists $\varpi$ in $\mathcal{H}$, such that $\varpi_n\rightarrow
\varpi$. By the continuity of $A$, we have $A\varpi_n\rightarrow
A\varpi$. Since $A\varpi\in R (A )$, $v_n$ converges to an element
in $R(A)$. Therefore, $R (A )$ is closed.  
\end{proof}

\begin{corollary}
\label{SurjectiveA}
        Let $A$ be as in \eqref{DefA}. Then, 
        $R (A )=\mathcal{H}$.
\end{corollary}
\begin{proof}
        Suppose that $R (A )\not=\mathcal{H}$. Since $R (A )$
is closed in $\mathcal{H}$, there exists $\varpi\in\mathcal{H}$
such that $\varpi\perp R (A )$. Then, we have
        \begin{align*}
        0&= \langle A\varpi, \varpi \rangle=B [\varpi, \varpi
]=\int_{\mathcal{Y}^d}e^{\frac{ |\Lambda+\nabla_y\widetilde{w}
|^2}{2}+\lambda V} \big( \big( (\Lambda+\nabla_y\widetilde{w}
)^T\nabla_y\varpi\big )^2+ |\nabla_y\varpi |^2 \big)dy.
        \end{align*}
        Due to the smoothness of $V$, $e^{\frac{ |\Lambda+\nabla_y\widetilde{w}
|^2}{2}+\lambda V (x,y )}$ is strictly positive. Thus, $\int_{\mathcal{Y}^d}
|\nabla_y\varpi |^2 dy= 0$. Since  $\int_{\mathcal{Y}^d}{\varpi}_ndy=0$, by Poincar\'{e}'s inequality, we
conclude that  $\varpi= 0$. Hence, $R (A )=\mathcal{H}$.
\end{proof}
\begin{pro}
\label{LBijectiveH2}
        Let $B$ be the bilinear form given in \eqref{DefB} and
$\mathcal{L}$ as in \eqref{DefL}. 
        For any $v_0\in  C^{0,\alpha}_\# (\mathcal{Y}^d )$, there
exists a unique $v\in \mathcal{H}$ such that  $B [v,\varpi ]=
\langle v_0,\varpi \rangle_{\mathcal{H}^0}$ for all $\varpi\in
\mathcal{H}$. Moreover, $v\in C^{2,\alpha}_\# (\mathcal{Y}^d
)$ and solves  $\mathcal{L} (v )=v_0$.
\end{pro}
\begin{proof}
        Let $v_0\in C^{0,\alpha}_\# (\mathcal{Y}^d )$. By the
Riesz representation theorem, there exists a unique $\sigma\in\mathcal{H}$
such that for all $\varpi\in\mathcal{H}$, 
\begin{equation}
\label{RiezRepre}
\langle \sigma, \varpi \rangle_{\mathcal{H}}=  \langle
v_0, \varpi \rangle_{\mathcal{H}^0}. 
\end{equation}
By Corollary \ref{SurjectiveA} and the injectivity of $A$, $A$
is invertible. Thus, defining $v=A^{-1}\sigma$ and using \eqref{DefA} and \eqref{RiezRepre}, we obtain
        \begin{equation*}
        B [v,\varpi ]= \langle Av,\varpi \rangle_{\mathcal{H}}=
\langle \sigma, \varpi \rangle_{\mathcal{H}}= \langle v_0, \varpi
\rangle_{\mathcal{H}^0}.
        \end{equation*}
        Therefore, $v$ is a weak solution of 
        \begin{equation*}
        \div\Big (e^{\frac{ |\Lambda+\nabla_y\widetilde{w} |^2}{2}+\lambda
V (x,y )} \big( (\Lambda+\nabla_y\widetilde{w} )^T\nabla_yv (\Lambda+\nabla_y\widetilde{w}
)+\nabla_y v\big )\Big )=v_0.
        \end{equation*}
        By Schauder's estimate, since all coefficients are in
$C^{0,\alpha}_\# (\mathcal{Y}^d )$, we see that $v\in C^{2,\alpha}_\#
(\mathcal{Y}^d )$.  
\end{proof}

\begin{pro}
\label{SOpen} 
        $\mathcal{S}$ defined in \eqref{DefS} is open. 
\end{pro}
\begin{proof}
         According to Proposition~\ref{LBijectiveH2},  $\mathcal{L}:C^{2,\alpha}_\#
(\mathcal{Y}^d )\rightarrow C^{0,\alpha}_\# (\mathcal{Y}^d )$
is an isomorphism. Let $\lambda\in\mathcal{S}$.  By the implicit
function theorem, given $x\in \mathbb{T}^d$ and $\Lambda\in \mathbb{R}^d$, there exist a neighborhood $U$ of $\lambda$ and
a unique solution $\widetilde{w}\in C^{2,\alpha}_\# (\mathcal{Y}^d
)$ to $F (\widetilde{w}, \widetilde{\lambda},x, \Lambda )=0$
for any $\widetilde{\lambda}\in U$. Thus, $\mathcal{S}$ is open.
\end{proof}
\begin{remark}
\label{Smoothnessofwonxlambda}
In the proof of Proposition~\ref{SOpen}, the implicit function
theorem also gives  that $\widetilde{w}$ is smooth in $\lambda$,
$x$, and $\Lambda$.
\end{remark}
\begin{pro}
\label{SClosed} 
        $\mathcal{S}$ given in \eqref{DefS} is closed. 
\end{pro}
\begin{proof}
        Let $ (\lambda_n )_n$ be a Cauchy sequence in $\mathcal{S}$
converging to $\lambda\in \mathbb{R}$. Moreover, we take   $
(\widetilde{w}_n )_n\subset C^{2,\alpha}_\# (\mathcal{Y}^d )$
such that 
\begin{equation}
\label{Fseq}
F (\widetilde{w}_n, \lambda_n, x, \Lambda )=0.
\end{equation}
According to Proposition~\ref{Holderwtilde}, 
        $\widetilde{w}_n$ is uniformly bounded in  $C^{2,\beta}_\#
(\mathcal{Y}^d )$. Since $C^{2,\beta}_\# (\mathcal{Y}^d )\subset\subset
C^{2,\alpha}_\# (\mathcal{Y}^d )$, there exists a function $\widetilde{w}\in
C^{2,\alpha}_\# (\mathcal{Y}^d )$ such that, up to a subsequence,
$\widetilde{w}_n$ converges to $\widetilde{w}$ in $C^{2,\alpha}_\#
(\mathcal{Y}^d )$. Thus, considering the limit of \eqref{Fseq}, as $n\rightarrow+\infty$, we conclude that
$F (\widetilde{w}, \lambda,x,\lambda )=0$. Therefore, $\mathcal{S}$
is closed.
\end{proof}
Then, we have existence of the solution to \eqref{CellProbEulerLagrangeEq}.
\begin{pro}
\label{ExistenceUniqueness} 
There exists a unique function $\widetilde{w}\in C^\infty (\mathbb{R}^d\times\mathbb{T}^d;C^{2,\alpha}_\#
(\mathcal{Y}^d )/\mathbb{R} )$ such that for $\Lambda\in\mathbb{R}^d$
and $x\in\mathbb{T}^d$, $\widetilde{w} (x, \Lambda,\cdot )$ solves
\eqref{DivForm}.  Moreover, let $\widetilde{m}$  be defined in
\eqref{EllipticWidetidem} and $\widetilde{H}$ be given in \eqref{EllipticWidetideH},
then $(\widetilde{w}, \widetilde{m}, \widetilde{H})$ solves 
Problem~\ref{CellProbEulerLagrangeProb}. Accordingly, $\widetilde{w}
(x, \Lambda,\cdot )$ is the unique minimizer to Problem~\ref{TheCellProblem}.

\end{pro}
\begin{proof}
By      Propositions~\ref{SOpen} and \ref{SClosed},  $\mathcal{S}$
in \eqref{DefS} is open and closed. Thus, for $F$ as in \eqref{DefFFunc},
$F=0$ has a solution in $\widetilde{w}\in C^{2,\alpha}_\# (\mathcal{Y}^d )$ when
$\lambda=1$. Thus, let $\widetilde{m}$  be as in
\eqref{EllipticWidetidem} and $\widetilde{H}$ be as in \eqref{EllipticWidetideH},
then $(\widetilde{w}, \widetilde{m}, \widetilde{H})$ solves 
Problem~\ref{CellProbEulerLagrangeProb}. 
By the uniqueness in Proposition~\ref{CellProblemNonUniqueness},
we  conclude that \eqref{CellProbEulerLagrangeEq} admits a unique
 solution.
 According to Remark~\ref{Smoothnessofwonxlambda}, $\widetilde{w}$
depends smoothly on $\Lambda$ and $x$. Thus, $\widetilde{w}\in
C^\infty (\mathbb{R}^d\times\mathbb{T}^d;C^{2,\alpha}_\# (\mathcal{Y}^d
)/\mathbb{R} )$.
\end{proof}

\subsubsection{Lower bounds for $\widetilde{m}$}
Next, we prove a uniform lower bound for $\widetilde{m}$, which
is used to prove the existence of solutions to Problem~\ref{TheLimitProblem}.
According to Corollary \ref{Lowerboundmtilde}, it is sufficient
to prove a lower bound for $\widetilde{m}$ as $ |\Lambda |\rightarrow
\infty$. 

\begin{pro}
        \label{UniformlyLowerboundnessofwidetildem}
        Let $(\widetilde{w}, \widetilde{m}, \widetilde{H})$ solve
Problem~\ref{CellProbEulerLagrangeProb}. If $d>1$, we assume
further that $V$ satisfies \eqref{SeparableV}. Then, there exists
a constant $C>0$ such that $\widetilde{m}\geq C$ for all $x\in\mathbb{T}^d,
y\in\mathcal{Y}^d$, and $\Lambda\in\mathbb{R}^d$.
\end{pro}
\begin{proof}
For $d=1$, we use the \textit{current method} introduced in \cite{Gomes2016b} to  get an explicit formula for the solution of
 \eqref{CellProbEulerLagrangeEq} as in Section \ref{TwoScaleHomogeInOneDim}.
According to the second equation in \eqref{CellProbEulerLagrangeEq},
 
\begin{equation}
\label{CurrentEq}
j=\widetilde{m} (\Lambda+\nabla_y\widetilde{w} )
\end{equation}
depends only on $x$. 
Integrating the prior equation and 
using $\int_{\mathcal{Y}^d}\widetilde{m}dy=1$, we get
\begin{equation*}
j=\Lambda+\int_{\mathcal{Y}^d}\widetilde{m}\nabla_y\widetilde{w}dy.
\end{equation*}
By Propositions~\ref{CoercivityWidetideH} and  \ref{UpperBoundw},
there exists a constant $C>0$, independent on $x$, such that
\begin{equation*}
\int_{\mathcal{Y}^d}\widetilde{m}\nabla_y\widetilde{w}dy \leq
C \int_{\mathcal{Y}^d} |\nabla_y\widetilde{w}|dy \leq C.
\end{equation*} 
Thus, 
\begin{equation}
\label{BoundjbyLambda}
 |\Lambda |-C\leq  |j |\leq  |\Lambda |+C. 
\end{equation}
Using \eqref{CurrentEq}, the first equation of \eqref{CellProbEulerLagrangeEq}
becomes
\begin{equation*}
\frac{j^2}{2\widetilde{m}^2}+V=\ln\widetilde{m}+\widetilde{H}.
\end{equation*}
According to \eqref{BoundjbyLambda} and Proposition~\ref{CoercivityWidetideH},
we rewrite $\widetilde{H}=\frac{j^2}{2}+\widetilde{H}_1$, where
$\widetilde{H}_1\in\mathbb{R}$  and $\widetilde{H}_1$ is bounded
as $ |j |\rightarrow\infty$. Thus, 
when $j\not=0$, we 
divide both sides of the prior equation by $j^2$ and get 
\begin{equation}
\label{CurrentFormulationofCellProbEulerLagrangeEq}
\frac{1}{2\widetilde{m}^2}+\frac{1}{j^2}V=\frac{1}{j^2}\ln \widetilde{m}+\frac{1}{2}+\frac{1}{j^2}\widetilde{H}_1.
\end{equation}
We consider the Banach space
\begin{equation*}
\mathcal{B}= \bigg\{\overline{m}\in L^2_\# (\mathcal{Y}^d ) :
\int_{\mathcal{Y}^d}\overline{m}dy=0\bigg\}
\end{equation*}
and its subset
\begin{equation*}
\mathcal{B}_1= \bigg\{\overline{m}\in L^2_\# (\mathcal{Y}^d ):\overline{m}>-\frac{1}{2},
 \int_{\mathcal{Y}^d}\overline{m}dy=0\bigg\}.
\end{equation*}
Then,   \eqref{CurrentFormulationofCellProbEulerLagrangeEq} inspires
us to define $G: \mathbb{T}^d\times  [0,+\infty )\times\mathcal{B}_1\rightarrow
L^2(\mathbb{T}^d;L^2_\#(\mathcal{Y}^d))$ as  
\begin{equation*}
G (x, \epsilon, \overline{m} )=\frac{1}{2(\overline{m}+1)^2}+\epsilon
V-\epsilon\ln (\overline{m}+1)-\frac{1}{2}.
\end{equation*}
For given $x\in\mathbb{T}^d$ and $\epsilon=0$, we see that $\overline{m}=0$
solves $G (x, 0, \overline{m} )=\frac{1}{2(\overline{m}+1)^2}-\frac{1}{2}=0.$
Differentiating $G$ with respect to $\overline{m}$, we obtain
\begin{align*}
\frac{\partial G}{\partial \overline{m}}=-\frac{1}{(\overline{m}+1)^3}-\epsilon\frac{1}{\overline{m}+1}.
\end{align*}
Then, $\frac{\partial G}{\partial \overline{m}}\not=0$ when $\epsilon=0$.
By the implicit function theorem, there exists a neighborhood, $U\subset\mathcal{B}$,
of $x$ and a positive number, $\epsilon_0$, such that for any
$\widehat{x}\in U$ and $0<\widehat{\epsilon}<\epsilon_0$, there
exists a unique function $\widehat{m}\in\mathcal{B}$ such that
$G (\widehat{x}, \widehat{\epsilon}, \widehat{m})=0$. Moreover,
for  $\widehat{\epsilon}$ small enough, $\widehat{m}$ is bounded
uniformly below by $-\frac{1}{2}$. Thus, given $x$ and when $j$ in \eqref{CurrentFormulationofCellProbEulerLagrangeEq}
is large enough, by the uniqueness of the solution to \eqref{CurrentFormulationofCellProbEulerLagrangeEq}
given in Proposition~\ref{ExistenceUniqueness}, $\widetilde{H}_1=0$
and $\widetilde{m}$ is uniformly bounded by below in $j$. In
particular, using the compactness of $\mathbb{T}^d$, it is possible
to choose $\epsilon_0$ that is valid for all $x$. Thus, combining
the previous arguments with Corollary \ref{Lowerboundmtilde},
which gives a uniform bound of $\widetilde{m}$ when $j$ is small,
we see that $\widetilde{m}$ is uniformly bounded by below for
all $x\in\mathbb{T}, y\in\mathcal{Y}$, and $\Lambda\in\mathbb{R}$.

For $d>1$. We assume that $V$ satisfies \eqref{SeparableV}. 
In this case, the solution $ (\widetilde{m},\widetilde{w},\widetilde{H}
)$ of \eqref{CellProbEulerLagrangeEq} is separable in $y$ and
can be written as
\begin{equation}
\label{SeparableSols}
\widetilde{m} (x,y )=\prod_{i=1}^{d}\widetilde{m}_i (x,y_i ),
\ \widetilde{w} (x,y )=\sum_{i=1}^{d}\widetilde{w}_i (x,y_i ),
\  \widetilde{H} (x,\Lambda )=\sum_{i=1}^{d}\widetilde{H}_i (x,\Lambda_{i}
),
\end{equation}
where $\widetilde{m}_i: \mathbb{T}^d\times\mathcal{Y}\rightarrow
\mathbb{R}$, $\widetilde{w}_i: \mathbb{T}^d\times\mathcal{Y}\rightarrow
\mathbb{R}$, and $\widetilde{H}_i: \mathbb{T}^d\times\mathbb{R}\rightarrow
\mathbb{R}$ are defined for all $i=1,\dots,d$.
Accordingly, \eqref{CellProbEulerLagrangeEq} can be split into
one-dimensional systems; that is, for each $i=1,\dots,d$, we
have 
\begin{align*}
\begin{split}
\begin{cases}
\frac{ |\Lambda_i+  (\widetilde{w}_i (x,y_i) )_{y_i} |^2}{2}+V_i
(x,y_i )=\ln \widetilde{m}_i (x,y_i )+\widetilde{H}_i (x,\Lambda_i
),\\
 \big(\widetilde{m}_i (x,y_i) (\Lambda_i+  (\widetilde{w}_i (x,y_i)
)_{y_i} )\big )_{y_i}=0, \\
\int_0^1 \widetilde{m}_i (x,y_i)dy=1.
\end{cases}
\end{split}
\end{align*}
By the above arguments, $\widetilde{m}_i$ is bounded below as
$ |\Lambda_{i} |\rightarrow \infty$ for each $i=1,\dots,d$. Therefore,
$\widetilde{m}$ satisfying \eqref{SeparableSols} is uniformly
bounded by below.
\end{proof}

\begin{remark}
        When $d>1$ and $V$ is non-separable, the lower boundedness
of $\widetilde{m}$ is still unknown. Here, we present the difficulty
we faced. Let $\Lambda=\lambda \Gamma$, where $\lambda\in \mathbb{R_+}$,
$\Gamma\in\mathbb{R}^d$,  and $ |\Gamma |=1$. Besides, let $\epsilon=\frac{1}{\lambda}$
and  $\widetilde{H}=\frac{\lambda^2}{2}+\widetilde{H}_1$. Then,
multiplying both sides of the first equation in   \eqref{CellProbEulerLagrangeEq}
by $\epsilon$ and rearranging, we get
        \begin{align*} \Gamma^T\nabla_y\widetilde{w}+\epsilon\bigg
(\frac{ |\nabla_y\widetilde{w} |^2}{2}+ V-\ln \widetilde{m}-\widetilde{H}_1
\bigg)=0.
        \end{align*}
        Similarly, the transport equation of \eqref{CellProbEulerLagrangeEq}
becomes
        \begin{align*}
        \Gamma^T\nabla_y\widetilde{m}+\epsilon\big (\nabla_y\widetilde{m}^T\nabla_y\widetilde{w}+\widetilde{m}\Delta_y\widetilde{w}
\big)=0.
        \end{align*}
        Thus, we can define an operator $F$ such that
        \begin{align*}
        F (\epsilon,\widetilde{w}, \widetilde{m} )=\begin{pmatrix}
        \Gamma^T\nabla_y\widetilde{w}+\epsilon\Big (\frac{ |\nabla_y\widetilde{w}
|^2}{2}+ V-\ln \widetilde{m}-\widetilde{H}_1 \Big)\\
        \Gamma^T\nabla_y\widetilde{m}+\epsilon\big (\nabla_y\widetilde{m}^T\nabla_y\widetilde{w}+\widetilde{m}\Delta_y\widetilde{w}
\big) 
        \end{pmatrix}. 
        \end{align*}
        The linearized operator of $F$ with respect to $ (\widetilde{w},\widetilde{m}
)$ when $\epsilon=0$ is given by
        \begin{equation*}
        \mathcal{L}_{\widetilde{w},\widetilde{m}} (v,\mu )= (\Gamma^T\nabla_yv,
 \Gamma^T\nabla_y\mu ),
        \end{equation*}
        which fails to be an isomorphism. Therefore, in this
case, we cannot use the implicit function theorem as  we did
for the one-dimensional case.  
\end{remark}

\subsection{The homogenized problem} 
\label{HomogenizedProblemSection}
Here, we study existence of minimizers for Problem~\ref{TheLimitProblem}.
 Since \eqref{MinLimitProblem} is considered in \cite{evans2003some},
we only need to check that the Hamiltonian $\widetilde{H}$ satisfies
the assumptions in \cite{evans2003some} and apply the results
there directly. First, we give uniform bounds for derivatives
of $\widetilde{H}$ with respect to $\Lambda$ and $x$. For simplicity, we use Einstein's notation and remark that all the constants denoted by $C$ are independent on $x$ and $\Lambda$. 

\begin{pro}
        \label{BoundHx}
        Let $(\widetilde{w}, \widetilde{m}, \widetilde{H})$ solve
Problem~\ref{CellProbEulerLagrangeProb}.  Then, there exists
a positive constant $C$ such that 
        \begin{align*}
         |\widetilde{H}_{x_j} |\leq C.  
        \end{align*} 
\end{pro}
\begin{proof}
        Differentiating the equations in \eqref{CellProbEulerLagrangeEq}
with respect to $x_j$, we obtain
        \begin{align}
        \label{CellProbFirstDerivative}
        \begin{cases*}
         (\Lambda_i+\widetilde{w}_{y_i} )\widetilde{w}_{y_ix_j}+V_{x_j}=\frac{1}{\widetilde{m}}\widetilde{m}_{x_j}+\widetilde{H}_{x_j},\\
        -\big (\widetilde{m}_{x_j} (\Lambda_i+\widetilde{w}_{y_i}
)+\widetilde{m}\widetilde{w}_{y_ix_j} \big)_{y_i}=0,\\
        \int_{\mathcal{Y}^d}\widetilde{m}_{x_j}dy=0.
        \end{cases*}
        \end{align}
        Multiplying the first equation in the prior system by
$\widetilde{m}$, integrating the resulting terms, and using 
$\int_{\mathcal{Y}^d}\widetilde{m}dy=1$, we get
        \begin{align}   \label{HxIntegralForm0}
        \begin{split}
        \widetilde{H}_{x_j}&=\int_{\mathcal{Y}^d} (\Lambda_i+\widetilde{w}_{y_i}
)\widetilde{w}_{y_ix_j}\widetilde{m}dy+\int_{\mathcal{Y}^d}V_{x_j}\widetilde{m}dy-\int_{\mathcal{Y}^d}\widetilde{m}_{x_j}dy\\
        &=-\int_{\mathcal{Y}^d} \big( (\Lambda_i+\widetilde{w}_{y_i}
)\widetilde{m} \big)_{y_i}\widetilde{w}_{x_j}dy+\int_{\mathcal{Y}^d}V_{x_j}\widetilde{m}dy-\int_{\mathcal{Y}^d}\widetilde{m}_{x_j}dy.
        \end{split}
        \end{align}
        From the second equation of \eqref{CellProbEulerLagrangeEq},
we know that
        \begin{align}
        \label{HxIntegralFormTerm1}
        \int_{\mathcal{Y}^d} \big( (\Lambda_i+\widetilde{w}_{y_i}
)\widetilde{m} \big)_{y_i}\widetilde{w}_{x_j}dy=0.
        \end{align}
        The third equation of \eqref{CellProbFirstDerivative}
gives that
        \begin{equation}
        \label{HxIntegralFormTerm2}
        \int_{\mathcal{Y}^d}\widetilde{m}_{x_j}dy=0.
        \end{equation} 
        By the smoothness of $V$, the positivity of $\widetilde{m}$,
and $\int_{\mathcal{Y}^d}\widetilde{m}dy=1$, there exists a constant
$C$ such that 
        \begin{equation}
        \label{HxIntegralFormTerm3}
         \bigg|\int_{\mathcal{Y}^d}V_{x_j}\widetilde{m}dy \bigg|\leq
C. 
        \end{equation}
        Therefore, \eqref{HxIntegralForm0}--\eqref{HxIntegralFormTerm3}
yield
        \begin{equation*}
         |\widetilde{H}_{x_j} |\leq C.\qedhere
        \end{equation*}
\end{proof}

\begin{pro}
        Let $(\widetilde{w}, \widetilde{m}, \widetilde{H})$ solve
Problem~\ref{CellProbEulerLagrangeProb}.  Then, there exists
a positive constant $C$ such that 
        \begin{align}
        \label{HMedBoundResult}
        \int_{\mathcal{Y}^d}\frac{1}{\widetilde{m}}\widetilde{m}^2_{x_j}dy+\int_{\mathcal{Y}^d}\widetilde{m}\widetilde{w}_{y_ix_j}^2dy\leq
C. 
        \end{align}
\end{pro}
\begin{proof}
        Multiplying the first equation in \eqref{CellProbFirstDerivative}
by $\widetilde{m}_{x_j}$,  integrating the resulting terms in $\mathcal{Y}^d$, and
using $\int_{\mathcal{Y}^d}\widetilde{m}_{x_j}dy=0$, we get
        \begin{equation}
        \label{CellProbFirstDerivativeFirstEq}
        \int_{\mathcal{Y}^d} (\Lambda_i+\widetilde{w}_{y_i} )\widetilde{w}_{y_ix_j}\widetilde{m}_{x_j}dy+\int_{\mathcal{Y}^d}V_{x_j}\widetilde{m}_{x_j}dy=\int_{\mathcal{Y}^d}\frac{1}{\widetilde{m}}\widetilde{m}^2_{x_j}dy.
        \end{equation}
        Multiplying the second equation of \eqref{CellProbFirstDerivative}
by $\widetilde{w}_{x_j}$ and  integrating by parts, we obtain
        \begin{equation}
        \label{CellProbFirstDerivativeSecondEq}
        \int_{\mathcal{Y}^d}\widetilde{m}_{x_j} (\Lambda_i+\widetilde{w}_{y_i}
)\widetilde{w}_{x_jy_i}dy+\int_{\mathcal{Y}^d}\widetilde{m}\widetilde{w}_{y_ix_j}\widetilde{w}_{x_jy_i}dy=0.
        \end{equation}
        Subtracting \eqref{CellProbFirstDerivativeSecondEq} to
\eqref{CellProbFirstDerivativeFirstEq} and using Cauchy's inequality,
we have
        \begin{align*}
        \int_{\mathcal{Y}^d}\frac{1}{\widetilde{m}}\widetilde{m}^2_{x_j}dy+\int_{\mathcal{Y}^d}\widetilde{m}\widetilde{w}_{y_ix_j}^2dy=\int_{\mathcal{Y}^d}V_{x_j}\widetilde{m}_{x_j}dy\leq
\int_{\mathcal{Y}^d}\frac{1}{2\widetilde{m}}\widetilde{m}_{x_j}^2dy+\int_{\mathcal{Y}^d}\frac{V_{x_j}^2}{2}\widetilde{m}dy.
        \end{align*}
        Thus, by the smoothness of $V$ and $\int_{\mathcal{Y}^d}\widetilde{m}dy=1$,
there exists a constant $C$ such that 
        \begin{equation*}
        \int_{\mathcal{Y}^d}\frac{1}{\widetilde{m}}\widetilde{m}^2_{x_j}dy+\int_{\mathcal{Y}^d}\widetilde{m}\widetilde{w}_{y_ix_j}^2dy\leq
C. \qedhere
        \end{equation*}      
\end{proof}

\begin{pro}
        \label{BoundHxx}
        Let $(\widetilde{w}, \widetilde{m}, \widetilde{H})$ solve
Problem~\ref{CellProbEulerLagrangeProb}.  Then, there exists
a positive constant $C$ such that 
        \begin{equation*}
         |\widetilde{H}_{x_jx_l} |\leq C.
        \end{equation*}
\end{pro}
\begin{proof}
        Differentiating \eqref{CellProbFirstDerivative} with
respect to $x_l$, we get
        \begin{align}
        \label{CellProbSecondDerivative}
        \begin{cases}
        \widetilde{w}_{y_ix_l}\widetilde{w}_{y_ix_j}+ (\Lambda_i+\widetilde{w}_{y_i}
)\widetilde{w}_{y_ix_jx_l}+V_{x_jx_l}=-\frac{1}{\widetilde{m}^2}\widetilde{m}_{x_l}\widetilde{m}_{x_j}+\frac{1}{\widetilde{m}}\widetilde{m}_{x_jx_l}+\widetilde{H}_{x_jx_l},\\
        -\big (\widetilde{m}_{x_jx_l} (\Lambda_i+\widetilde{w}_{y_i}
)+\widetilde{m}_{x_j}\widetilde{w}_{y_ix_l}+\widetilde{m}_{x_l}\widetilde{w}_{y_ix_j}+\widetilde{m}\widetilde{w}_{y_ix_jx_l}
\big)_{y_i}=0,\\
        \int_{\mathcal{Y}^d}\widetilde{m}_{x_jx_l}dy=0.
        \end{cases}
        \end{align}
        Multiplying the first equation in \eqref{CellProbSecondDerivative}
by $\widetilde{m}$, integrating the resulting terms, and using
$\int_{\mathcal{Y}^d}\widetilde{m}dy=1$, we have
        \begin{align}
        \label{BoundHxxEq0}
        \begin{split}
        \widetilde{H}_{x_jx_l}=&\int_{\mathcal{Y}^d} \widetilde{m}
\widetilde{w}_{y_ix_l}\widetilde{w}_{y_ix_j}dy+\int_{\mathcal{Y}^d}
  (\Lambda_i+\widetilde{w}_{y_i} )\widetilde{w}_{y_ix_jx_l}\widetilde{m}
dy +\int_{\mathcal{Y}^d} V_{x_jx_l}\widetilde{m}dy\\
        &+\int_{\mathcal{Y}^d}\frac{\widetilde{m}_{x_l}\widetilde{m}_{x_j}}{\widetilde{m}}dy-\int_{\mathcal{Y}^d}\widetilde{m}_{x_jx_l}dy.\\
        \end{split}
        \end{align}
        From \eqref{CellProbSecondDerivative}, we know that 
        \begin{equation}
        \label{BoundHxxEq1}
        \int_{\mathcal{Y}^d}\widetilde{m}_{x_jx_l}dy=0.
        \end{equation}
        By the smoothness of $V$, the positivity of $\widetilde{m}$,
and $\int_{\mathcal{Y}^d}\widetilde{m}dy=1$, there exists a constant
$C$ such that 
        \begin{equation}
        \label{BoundHxxEq2}
         \bigg|\int_{\mathcal{Y}^d}V_{x_jx_l}\widetilde{m}dy\bigg
|\leq C.
        \end{equation}
        Using H\"{o}lder's inequality and \eqref{HMedBoundResult},
we get
        \begin{equation}
        \label{BoundHxxEq3}
         \bigg|\int_{\mathcal{Y}^d}\widetilde{m}\widetilde{w}_{y_ix_l}\widetilde{w}_{y_ix_j}dy\bigg
|\leq  \bigg(\int_{\mathcal{Y}^d}\widetilde{m}\widetilde{w}_{y_ix_l}^2dy
\bigg)^{\frac{1}{2}} \bigg(\int_{\mathcal{Y}^d}\widetilde{m}\widetilde{w}_{y_ix_j}^2dy\bigg
)^{\frac{1}{2}}\leq C
        \end{equation}
        and 
        \begin{equation}
        \label{BoundHxxEq4}
         \bigg|\int_{\mathcal{Y}^d}\frac{\widetilde{m}_{x_l}\widetilde{m}_{x_j}}{\widetilde{m}}dy\bigg
|\leq  \bigg(\int_{\mathcal{Y}^d}\frac{\widetilde{m}_{x_l}^2}{\widetilde{m}}dy\bigg
)^{\frac{1}{2}} \bigg(\int_{\mathcal{Y}^d}\frac{\widetilde{m}_{x_j}^2}{\widetilde{m}}dy\bigg
)^{\frac{1}{2}}\leq C. 
        \end{equation}
        From the second equation of \eqref{CellProbEulerLagrangeEq},
we obtain
        \begin{equation}
        \label{BoundHxxEq5}
        \int_{\mathcal{Y}^d} (\Lambda_i+\widetilde{w}_{y_i} )\widetilde{w}_{y_ix_jx_l}\widetilde{m}dy=-\int_{\mathcal{Y}^d}
\big(\widetilde{m} (\Lambda_i+\widetilde{w}_{y_i} )\big )_{y_i}\widetilde{w}_{x_jx_l}dy=0.
        \end{equation}
        Therefore, \eqref{BoundHxxEq0}--\eqref{BoundHxxEq5} give
that 
        \begin{equation*}
         |\widetilde{H}_{x_jx_l} |\leq C.\qedhere
        \end{equation*}
\end{proof}

\begin{pro}
        Let $(\widetilde{w}, \widetilde{m}, \widetilde{H})$ solve
Problem~\ref{CellProbEulerLagrangeProb}.  Then, there exists
a positive constant $C$ such that 
        \begin{equation}
        \label{HLambdaMedEstimate}
        \int_{\mathcal{Y}^d}\widetilde{m}\widetilde{w}_{y_i\Lambda_j}^2dy+\int_{\mathcal{Y}^d}\frac{\widetilde{m}_{\Lambda_j}^2}{\widetilde{m}}dy\leq
 C.
        \end{equation}. 
\end{pro}
\begin{proof}
        Let $\widetilde{\delta}_{ij}=1$ if $i=j$ and $\widetilde{\delta}_{ij}=0$
if $i\not=j$. 
        Differentiating  \eqref{CellProbEulerLagrangeEq} with
respect to $\Lambda_j$, we obtain 
        \begin{equation}
        \label{CellProbDLambda}
        \begin{cases}
         (\Lambda_i+\widetilde{w}_{y_i} ) (\widetilde{\delta}_{ij}+\widetilde{w}_{y_i\Lambda_j}
)=\frac{1}{\widetilde{m}}\widetilde{m}_{\Lambda_j}+\widetilde{H}_{\Lambda_j},
\\
        -\big (\widetilde{m}_{\Lambda_j} (\Lambda_i+\widetilde{w}_{y_i}
)+\widetilde{m} (\widetilde{\delta}_{ij}+\widetilde{w}_{y_i\Lambda_j}
)\big )_{y_i}=0,\\
        \int_{\mathcal{Y}^d}\widetilde{m}_{\Lambda_j}dy=0.
        \end{cases}
        \end{equation}
        Multiplying the first equation of the prior system by
$\widetilde{m}_{\Lambda_j}$, integrating the resulting terms,
and using $\int_{\mathcal{Y}^d}\widetilde{m}_{\Lambda_j}dy=0$,
we get
        \begin{equation}
        \label{CellProbDLambdaMulFirstMLambda}
        \int_{\mathcal{Y}^d} (\Lambda_i+\widetilde{w}_{y_i} )\widetilde{\delta}_{ij}\widetilde{m}_{\Lambda_j}dy+\int_{\mathcal{Y}^d}
(\Lambda_i+\widetilde{w}_{y_i} )\widetilde{w}_{y_i\Lambda_j}\widetilde{m}_{\Lambda_j}dy=\int_{\mathcal{Y}^d}\frac{\widetilde{m}_{\Lambda_j}^2}{\widetilde{m}}dy.
        \end{equation}
        Multiplying the second equation in \eqref{CellProbDLambda}
by $\widetilde{w}_{\Lambda_j}$ and integrating by parts, we have

        \begin{equation}
        \label{CellProbDLambdaMulSecondMLambda}
        \int_{\mathcal{Y}^d} \widetilde{m}_{\Lambda_j} (\Lambda_i+\widetilde{w}_{y_i}
)\widetilde{w}_{\Lambda_jy_i}dy+\int_{\mathcal{Y}^d}\widetilde{m}
(\widetilde{\delta}_{ij}+\widetilde{w}_{y_i\Lambda_j} )\widetilde{w}_{\Lambda_jy_i}dy=0.
        \end{equation}
        Subtracting \eqref{CellProbDLambdaMulSecondMLambda} to
\eqref{CellProbDLambdaMulFirstMLambda}  and rearranging, we obtain
        \begin{align}
        \label{HLambdaMedEstimateEq0}
        \begin{split}
        \int_{\mathcal{Y}^d}\widetilde{m}\widetilde{w}_{y_i\Lambda_j}^2dy+\int_{\mathcal{Y}^d}\frac{\widetilde{m}_{\Lambda_j}^2}{\widetilde{m}}dy=&\int_{\mathcal{Y}^d}
(\Lambda_i+\widetilde{w}_{y_i} )\widetilde{\delta}_{ij}\widetilde{m}_{\Lambda_j}dy-\int_{\mathcal{Y}^d}\widetilde{m}\widetilde{\delta}_{ij}\widetilde{w}_{\Lambda_jy_i}dy\\
        =&\int_{\mathcal{Y}^d}\Lambda_i\widetilde{\delta}_{ij}\widetilde{m}_{\Lambda_j}dy+\int_{\mathcal{Y}^d}\widetilde{w}_{y_i}\widetilde{\delta}_{ij}\widetilde{m}_{\Lambda_j}dy-\int_{\mathcal{Y}^d}\widetilde{m}\widetilde{\delta}_{ij}\widetilde{w}_{\Lambda_jy_i}dy.
        \end{split}     
        \end{align}
        Since $\int_{\mathcal{Y}^d}\widetilde{m}_{\Lambda_j}dy=0$,
we have
        \begin{align}
        \label{HLambdaMedEstimateEq1}
        \int_{\mathcal{Y}^d}\Lambda_i\widetilde{\delta}_{ij}\widetilde{m}_{\Lambda_j}dy=0.
        \end{align}
        By Young's inequality,  \eqref{HMedBoundResult},  and
Propositions  \ref{CoercivityWidetideH} and \ref{UpperBoundw},
there exists a constant $C$ such that 
        \begin{equation}
        \label{HLambdaMedEstimateEq2}
         \bigg|\int_{\mathcal{Y}^d}\widetilde{w}_{y_i}\widetilde{\delta}_{ij}\widetilde{m}_{\Lambda_j}dy\bigg
|\leq \frac{1}{2}\int_{\mathcal{Y}^d} \widetilde{m}\widetilde{w}_{y_i}^2dy+\frac{1}{2}\int_{\mathcal{Y}^d}\frac{\widetilde{m}_{\Lambda_j}^2}{\widetilde{m}}dy\leq
C+\frac{1}{2}\int_{\mathcal{Y}^d}\frac{\widetilde{m}_{\Lambda_j}^2}{\widetilde{m}}dy.
        \end{equation}
        Using Young's inequality again, we get
        \begin{align}
        \label{HLambdaMedEstimateEq3}
        \int_{\mathcal{Y}^d}\widetilde{m} |\widetilde{w}_{\Lambda_jy_i}
|dy
        \leq \int_{\mathcal{Y}^d}\frac{\widetilde{m}}{2}dy+\int_{\mathcal{Y}^d}\frac{\widetilde{w}^2_{\Lambda_jy_i}\widetilde{m}}{2}dy\leq
C+\int_{\mathcal{Y}^d}\frac{\widetilde{w}^2_{\Lambda_jy_i}\widetilde{m}}{2}dy.
        \end{align}
        Thus, \eqref{HLambdaMedEstimateEq0}--\eqref{HLambdaMedEstimateEq3}
yield
        \begin{equation*}
        \int_{\mathcal{Y}^d}\widetilde{m}\widetilde{w}_{y_i\Lambda_j}^2dy+\int_{\mathcal{Y}^d}\frac{\widetilde{m}_{\Lambda_j}^2}{\widetilde{m}}dy\leq
 C.\qedhere
        \end{equation*} 
\end{proof}

\begin{pro}
        Let $(\widetilde{w}, \widetilde{m}, \widetilde{H})$ solve
Problem~\ref{CellProbEulerLagrangeProb}.  Then, there exists
a positive constant $C$ such that 
        \begin{equation*}
         |\widetilde{H}_{\Lambda_j} |\leq C (1+ |\Lambda | ).
        \end{equation*}
\end{pro}
\begin{proof}
        Multiply the first equation in \eqref{CellProbDLambda}
by $\widetilde{m}$, we get
        \begin{equation*}
         (\Lambda_i+\widetilde{w}_{y_i} ) (\widetilde{\delta}_{ij}+\widetilde{w}_{y_i\Lambda_j}
)\widetilde{m}=\widetilde{m}_{\Lambda_j}+\widetilde{H}_{\Lambda_j}\widetilde{m}.
        \end{equation*}
        Integrating the preceding identity over $\mathcal{Y}^d$
and taking into account that $\int_{\mathcal{Y}^d}\widetilde{m}dy=1$
and  $\int_{\mathcal{Y}^d}\widetilde{m}_{\Lambda_j}dy=0$, we
obtain
        \begin{equation}
        \label{HLambdaEq1}
        \widetilde{H}_{\Lambda_j}=\int_{\mathcal{Y}^d} (\Lambda_i+\widetilde{w}_{y_i}
)\widetilde{\delta}_{ij}\widetilde{m}dy+\int_{\mathcal{Y}^d}
(\Lambda_i+\widetilde{w}_{y_i} )\widetilde{w}_{y_i\Lambda_j}\widetilde{m}dy.
        \end{equation}
        Multiplying the second equation of \eqref{CellProbEulerLagrangeEq}
by $\widetilde{w}_{\Lambda_j}$ and integrating by parts, we have

        \begin{equation}
        \label{HLambdaEq2}
        \int_{\mathcal{Y}^d} (\Lambda_i+\widetilde{w}_{y_i} )\widetilde{w}_{y_i\Lambda_j}\widetilde{m}dy=0.
        \end{equation}
        By Proposition~\ref{UpperBoundw}, there exists a constant
$C$ such that 
        \begin{equation}
        \label{HLambdaEq3}
         \bigg|\int_{\mathcal{Y}^d} (\Lambda_i+\widetilde{w}_{y_i}
)\widetilde{\delta}_{ij}\widetilde{m}dy\bigg |\leq C (1+ |\Lambda
| ).
        \end{equation}
        Therefore, \eqref{HLambdaEq1}--\eqref{HLambdaEq3} give

        \begin{equation*}
         |\widetilde{H}_{\Lambda_j} |\leq C (1+ |\Lambda | ).\qedhere
        \end{equation*}
\end{proof}

\begin{pro}
        Let $(\widetilde{w}, \widetilde{m}, \widetilde{H})$ solve
Problem~\ref{CellProbEulerLagrangeProb}.  Then, there exists
a positive constant $C$ such that 
        \begin{equation*}
         |\widetilde{H}_{\Lambda_j\Lambda_l} |\leq C.
        \end{equation*}
\end{pro}
\begin{proof}
        Differentiating the first equation in  \eqref{CellProbDLambda}
with respect to $\Lambda_l$, we get
        \begin{equation*}
         (\widetilde{\delta}_{il}+\widetilde{w}_{y_i\Lambda_l}
) (\widetilde{\delta}_{ij}+\widetilde{w}_{y_i\Lambda_j} )+ (\Lambda_i+\widetilde{w}_{y_i}
)\widetilde{w}_{y_i\Lambda_j\Lambda_l}=-\frac{\widetilde{m}_{\Lambda_l}\widetilde{m}_{\Lambda_j}}{\widetilde{m}^2}+\frac{\widetilde{m}_{\Lambda_j\Lambda_l}}{\widetilde{m}}+\widetilde{H}_{\Lambda_j\Lambda_l}.
        \end{equation*}
        Multiplying both sides of the prior equation by $\widetilde{m}$,
integrating, and using $\int_{\mathcal{Y}^d}\widetilde{m}_{\Lambda_j\Lambda_l}dy=0$
and $\int_{\mathcal{Y}^d}\widetilde{m}dy=1$, we obtain
        \begin{align}
        \label{HLambdaLambdaForm}
        \begin{split}
        &\int_{\mathcal{Y}^d} (\widetilde{\delta}_{il}+\widetilde{w}_{y_i\Lambda_l}
) (\widetilde{\delta}_{ij}+\widetilde{w}_{y_i\Lambda_j} )\widetilde{m}dy+\int_{\mathcal{Y}^d}
(\Lambda_i+\widetilde{w}_{y_i} )\widetilde{w}_{y_i\Lambda_j\Lambda_l}\widetilde{m}dy\\
        &\quad=-\int_{\mathcal{Y}^d}\frac{\widetilde{m}_{\Lambda_l}\widetilde{m}_{\Lambda_j}}{\widetilde{m}}dy+\widetilde{H}_{\Lambda_j\Lambda_l}.
        \end{split}
        \end{align}
        We multiply the second equation in \eqref{CellProbEulerLagrangeEq}
by $\widetilde{w}_{\Lambda_j\Lambda_l}$, integrate by parts,
and get
        \begin{equation}
        \label{HLambdaLambdaMed1}
        \int_{\mathcal{Y}^d} (\Lambda_i+\widetilde{w}_{y_i} )\widetilde{w}_{y_i\Lambda_j\Lambda_l}\widetilde{m}dy=0.
        \end{equation}
        By Young's inequality and \eqref{HLambdaMedEstimate},
there exists a constant $C$ such that 
        \begin{equation}
        \label{HLambdaLambdaMed2}
        \int_{\mathcal{Y}^d}\frac{ |\widetilde{m}_{\Lambda_l}\widetilde{m}_{\Lambda_j}
|}{\widetilde{m}}dy\leq \int_{\mathcal{Y}^d}\frac{\widetilde{m}^2_{\Lambda_{l}}}{2\widetilde{m}}dy+\int_{\mathcal{Y}^d}\frac{\widetilde{m}^2_{\Lambda_{j}}}{2\widetilde{m}}dy\leq
C
        \end{equation}
        and 
        \begin{equation}
        \label{HLambdaLambdaMed3}
        \int_{\mathcal{Y}^d}  |\widetilde{w}_{y_i\Lambda_l}\widetilde{w}_{y_i\Lambda_j}
|\widetilde{m}dy\leq \frac{1}{2}\int_{\mathcal{Y}^d}\widetilde{w}_{y_i\Lambda_l}^2\widetilde{m}dy+\frac{1}{2}\int_{\mathcal{Y}^d}\widetilde{w}_{y_i\Lambda_j}^2\widetilde{m}dy\leq
C.
        \end{equation}
        Therefore, combining \eqref{HLambdaLambdaForm}--\eqref{HLambdaLambdaMed3},
we conclude that
        \begin{equation*}
         |\widetilde{H}_{\Lambda_j\Lambda_l} |\leq C.\qedhere
        \end{equation*}
\end{proof}

\begin{pro}
        \label{BoundHxLambda}
        Let $(\widetilde{w}, \widetilde{m}, \widetilde{H})$ solve
Problem~\ref{CellProbEulerLagrangeProb}.  Then, there exists
a positive constant $C$ such that 
        \begin{equation*}
         |\widetilde{H}_{\Lambda_lx_j} |\leq C.
        \end{equation*}
\end{pro}
\begin{proof}
        Differentiating the first equation in \eqref{CellProbFirstDerivative}
 with respect to $\Lambda_l$, we get
        \begin{equation*}
         (\widetilde{\delta}_{il}+\widetilde{w}_{y_i\Lambda_l}
)\widetilde{w}_{y_ix_j}+ (\Lambda_i+\widetilde{w}_{y_i} )\widetilde{w}_{y_ix_j\Lambda_l}=-\frac{\widetilde{m}_{x_j}\widetilde{m}_{\Lambda_l}}{\widetilde{m}^2}+\frac{\widetilde{m}_{\Lambda_lx_j}}{\widetilde{m}}+\widetilde{H}_{\Lambda_lx_j}.
        \end{equation*}
        Multiplying both sides by $\widetilde{m}$, integrating,
and taking into account that $\int_{\mathcal{Y}^d}\widetilde{m}dy=1$
and $\int_{\mathcal{Y}^d}\widetilde{m}_{\Lambda_lx_j}dy=0$, we
obtain
        \begin{equation}
        \label{BoundHxLambdaEq0}
        \widetilde{H}_{\Lambda_lx_j}=\int_{\mathcal{Y}^d}\frac{\widetilde{m}_{x_j}\widetilde{m}_{\Lambda_l}}{\widetilde{m}}dy+\int_{\mathcal{Y}^d}
(\widetilde{\delta}_{il}+\widetilde{w}_{y_i\Lambda_l} )\widetilde{w}_{y_ix_j}\widetilde{m}dy+\int_{\mathcal{Y}^d}
(\Lambda_i+\widetilde{w}_{y_i} )\widetilde{w}_{y_ix_j\Lambda_l}\widetilde{m}dy.
        \end{equation}
        Using the second equation of \eqref{CellProbEulerLagrangeEq},
we have
        \begin{equation}
        \label{BoundHxLambdaEq1}
        \int_{\mathcal{Y}^d} (\Lambda_i+\widetilde{w}_{y_i} )\widetilde{w}_{y_ix_j\Lambda_l}\widetilde{m}dy=0.
        \end{equation}
        By Young's inequality \eqref{HMedBoundResult} and  \eqref{HLambdaMedEstimate},
there exists a constant $C$ such that 
        \begin{equation}
        \label{BoundHxLambdaEq2}
        \int_{\mathcal{Y}^d}\frac{ |\widetilde{m}_{x_j}\widetilde{m}_{\Lambda_l}
|}{\widetilde{m}}dy\leq \frac{1}{2}\int_{\mathcal{Y}^d}\frac{\widetilde{m}_{x_j}^2}{\widetilde{m}}dy+\frac{1}{2}\int_{\mathcal{Y}^d}\frac{\widetilde{m}_{\Lambda_l}^2}{\widetilde{m}}dy\leq
C
        \end{equation}
        and 
        \begin{equation}
        \label{BoundHxLambdaEq3}
        \int_{\mathcal{Y}^d} |\widetilde{w}_{y_i\Lambda_l}\widetilde{w}_{y_ix_j}
|\widetilde{m}dy\leq \frac{1}{2}\int_{\mathcal{Y}^d}\widetilde{w}_{y_i\Lambda_l}^2\widetilde{m}dy+\frac{1}{2}\int_{\mathcal{Y}^d}\widetilde{w}_{y_ix_j}^2\widetilde{m}dy\leq
C.
        \end{equation}
        Therefore, \eqref{BoundHxLambdaEq0}--\eqref{BoundHxLambdaEq3}
yield
        \begin{equation*}
         |\widetilde{H}_{\Lambda_lx_j} |\leq C.\qedhere
        \end{equation*}
\end{proof}

\begin{pro}
\label{UniformConvixity}
        Let $(\widetilde{w}, \widetilde{m}, \widetilde{H})$ solve
Problem~\ref{CellProbEulerLagrangeProb}.  Then, under the assumptions
of Proposition~\ref{UniformlyLowerboundnessofwidetildem}, $\widetilde{H}$
is uniformly convex; that is, for any $\xi\in\mathbb{R}^d$, $\xi=
\{\xi_1,\dots,\xi_d \}$, there exists a positive  constant $C$
such that 
        \begin{equation}
        \label{HStronnglyConvIneq}
        \xi_j\widetilde{H}_{\Lambda_j\Lambda_l}\xi_l\geq  C |\xi
|^2. 
        \end{equation}
\end{pro}
\begin{proof}
        Using \eqref{HLambdaLambdaForm} and \eqref{HLambdaLambdaMed1},
we have
        \begin{align*}
        \begin{split}
        \widetilde{H}_{\Lambda_j\Lambda_l}=\int_{\mathcal{Y}^d}
(\widetilde{\delta}_{il}+\widetilde{w}_{y_i\Lambda_l} ) (\widetilde{\delta}_{ij}+\widetilde{w}_{y_i\Lambda_j}
)\widetilde{m}dy+\int_{\mathcal{Y}^d}\frac{\widetilde{m}_{\Lambda_l}\widetilde{m}_{\Lambda_j}}{\widetilde{m}}dy.
        \end{split}
        \end{align*}
        Then, 
        \begin{align*}
        \xi_j\widetilde{H}_{\Lambda_j\Lambda_l}\xi_l=&\int_{\mathcal{Y}^d}\frac{
|\xi^T\nabla_\Lambda \widetilde{m} |^2}{\widetilde{m}}dy+\int_{\mathcal{Y}^d}
\xi^T (I+\nabla^2_{y\Lambda}\widetilde{w} )^T (I+\nabla^2_{y\Lambda}\widetilde{w}
)\xi\widetilde{m}dy\\
        \geq& \int_{\mathcal{Y}^d}  \big| (I+\nabla^2_{y\Lambda}\widetilde{w}
)^T\xi\big |^2\widetilde{m}dy,
        \end{align*}
        where $I$ is the identity matrix.  By Proposition~\ref{UniformlyLowerboundnessofwidetildem}
and Jensen's inequality, there exists a constant $C$ such that
\begin{equation*}
\begin{aligned}      
       \int_{\mathcal{Y}^d}  | (I+\nabla_{y\Lambda}\widetilde{w}
)\xi |^2\widetilde{m}dy \geq  C\int_{\mathcal{Y}^d}  | (I+\nabla_{y\Lambda}\widetilde{w}
)\xi |^2dy
        \geq  C \bigg|\int_{\mathcal{Y}^d} (I+\nabla_{y\Lambda}\widetilde{w}
)\xi dy\bigg |^2 =C |\xi |^2.
\end{aligned}
\end{equation*}
        Therefore, we conclude that \eqref{HStronnglyConvIneq}
holds.
\end{proof}
\begin{remark}
In the proof of Proposition~\ref{UniformConvixity}, the uniform lower bound of $\widetilde{m}$ is given by Proposition~\ref{UniformlyLowerboundnessofwidetildem}, where we assume that the potential $V$ is separable when $d\geq 2$. Proposition~\ref{UniformlyLowerboundnessofwidetildem} is the only point where we use the structure hypothesis given by \eqref{SeparableV} to get the uniform convexity of \(\widetilde{H}\). 
\end{remark}

The next proposition gives a proof for existence and  uniqueness
of the solution to the homogenized problem.
\begin{pro}
\label{ExistenceLimitProblem}
        Suppose that $V$ is smooth. Assume further that when
$d>1$, $V$ satisfies \eqref{SeparableV}. Then, Problem~\ref{TheLimitProblem}
admits a unique smooth minimizer and Problem \ref{Homogenized} has a unique solution.
\end{pro}
\begin{proof}
By Propositions~\ref{BoundHx}--\ref{UniformConvixity},  $\widetilde{H}$
satisfies the assumptions required in \cite{evans2003some}. Therefore,
Problem~\ref{TheLimitProblem} has a unique smooth minimizer. Accordingly, Problem \ref{Homogenized} admits a unique solution.

\end{proof}

Next, we prove that Problem~\ref{TwoScaleMinimization} has unique smooth minimizer.
\begin{pro}
\label{ExistenceofTwoScaleHomogenization}
Problem~\ref{TwoScaleMinimization} admits a unique minimizer,
$ (\widehat{u}_0,\widehat{u}_1 )$, where $\widehat{u}_0\in C^\infty
(\mathbb{T}^d )$ is the solution to Problem
\ref{TheLimitProblem} and $\widehat{u}_1\in C^{\infty} (\mathbb{T}^d;C^{2,\alpha}_\#
(\mathcal{Y}^d )/\mathbb{R} )$. Moreover, let $\widehat{I}$ be given in Problem \ref{TheLimitProblem} and $\overline{I}$ be as in Problem \ref{TwoScaleMinimization}. Then, \(\widehat I[\widehat u_0]
= \overline{I}
[\widehat{u}_0,\widehat{u}_1 ]\). 
\end{pro}
\begin{proof}
As pointed out in Remark \ref{rmk:p7p4qui}, Problem \ref{TwoScaleHomogenized} and Problem \ref{TwoScaleMinimization} are equivalent. Thus, if we prove the uniqueness of the solution to Problem \ref{TwoScaleHomogenized}, Problem \ref{TwoScaleMinimization} has a unique minimizer. To do that, we use a similar argument as what we did in the proof of Proposition \ref{CellProblemNonUniqueness}. Assume that $(u_0, u_1, m, \overline{H}_1)$ and $(\overline{u}_0, \overline{u}_1, \overline{m}, \overline{H}_2)$ are two solutions to Problem \ref{TwoScaleHomogenized} such that $u_0, \overline{u}_0\in C^\infty(\mathbb{T}^d)$, $u_1, \overline{u}_1\in C^\infty(\mathbb{T}^d;C^{2,\alpha}_{\#}(\mathcal{Y}^d)/\mathbb{R})$, $m, \overline{m}\in C^\infty(\mathbb{T}^d;C^{1,\alpha}(\mathcal{Y}^d))$, $\overline{H}_1, \overline{H}_2\in \mathbb{R}$,  $\int_{\mathbb{T}^d}u_0dx=\int_{\mathbb{T}^d}\overline{u}_0dx=0$, and $\int_{\mathbb{T}^d}\int_{\mathcal{Y}^d}mdydx=\int_{\mathbb{T}^d}\int_{\mathcal{Y}^d}\overline{m}dydx=1$. Then, we have
\begin{equation}
\label{uniqdiffeq}
\begin{cases}
\frac{\left|P+\nabla u_0+\nabla_y u_1\right|^2}{2}-\frac{\left|P+\nabla \overline{u}_0+\nabla_y \overline{u}_1\right|^2}{2}=\ln m - \ln \overline{m}+\overline{H}_1-\overline{H}_2,\\
-\div_x(\int_{\mathcal{Y}^d}m(P+\nabla u_0+\nabla_yu_1))+\div_x(\int_{\mathcal{Y}^d}\overline{m}(P+\nabla \overline{u}_0+\nabla_y\overline{u}_1))=0,\\
-\div_y(m(P+\nabla u_0+\nabla u_1))+\div_y(\overline{m}(P+\nabla \overline{u}_0+\nabla \overline{u}_1))=0.
\end{cases}
\end{equation}
Multiplying the first equation by $m-\overline{m}$, subtracting it from the sum of the second equation multiplied by $u_0-\overline{u}_1$ and the third equation multiplied by $u_1-\overline{u}_1$, integrating by parts, and using $\int_{\mathbb{T}^d}\int_{\mathcal{Y}^d}mdydx=\int_{\mathbb{T}^d}\int_{\mathcal{Y}^d}\overline{m}dydx=1$, we get
\begin{equation*}
\frac{1}{2}\int_{\mathbb{T}^d}\int_{\mathcal{Y}^d}(m+\overline{m})\left|\nabla u_0+\nabla_y u_1-\nabla \overline{u}_0-\nabla \overline{u}_1\right|^2dydx+\int_{\mathbb{T}^d}\int_{\mathcal{Y}^d}(\ln m-\ln \overline{m})(m-\overline{m})dydx=0.
\end{equation*} 
Thus, $$m=\overline{m}$$ and 
\begin{align}
\label{u0u1equbar0ubar1}
\nabla u_0 + \nabla_y u_1=\nabla \overline{u}_0+\nabla_y \overline{u}_1. 
\end{align} 
We integrate \eqref{u0u1equbar0ubar1} over $\mathcal{Y}^d$ and get $\nabla u_0=\nabla \overline{u}_0$. Since $\int_{\mathbb{T}^d}u_0dx=\int_{\mathbb{T}^d}\overline{u}_0dx=0$, we have $u_0=\overline{u}_0$. Thus, according to \eqref{u0u1equbar0ubar1}, $\nabla_y u_1=\nabla_y \overline{u}_1$. Because  $u_1, \overline{u}_1\in C^\infty(\mathbb{T}^d;C^{2,\alpha}_{\#}(\mathcal{Y}^d)/\mathbb{R})$, we have $u_1=\overline{u}_1$. Then, by \eqref{uniqdiffeq}, $\overline{H}_1=\overline{H}_2$. Therefore, Problem \ref{TwoScaleHomogenized} admits at most one solution. Accordingly, Problem \ref{TwoScaleMinimization} has at most one minimizer. 
        
Next, we prove the existence of the minimizer to Problem \ref{TwoScaleMinimization}.         
According to Propositions~\ref{ExistenceUniqueness} and \ref{ExistenceLimitProblem},
we let $\widetilde{w}\in C^\infty (\mathbb{R}^d\times\mathbb{T}^d;C^{2,\alpha}_\#
(\mathcal{Y}^d )/\mathbb{R} )$ be as in Proposition~\ref{ExistenceUniqueness} and $\widehat{u}_0\in C^\infty (\mathbb{T}^d )$ minimize Problem
\ref{TheLimitProblem}.
For $x\in\mathbb{T}^d$, $y\in\mathcal{Y}^d$, we define $\Lambda=P+\nabla
\widehat{u}_0 (x )$ and $\widehat{u}_1=\widetilde{w} (P+\nabla
\widehat{u}_0 (x ), x, y )$. Then, $\widehat{u}_1\in C^\infty
(\mathbb{T}^d;C^{2,\alpha}_\# (\mathcal{Y}^d )/\mathbb{R} )$.
Recalling the definition of $\overline{I}$ in Problem~\ref{TwoScaleMinimization},
we see that
\begin{equation}
\label{MinimizingOneSide}
\inf_{\overset{u\in W^{1,p} (\mathbb{T}^d )}{w\in L^{p} (\mathbb{T}^d;W^{1,p}_\#
(\mathcal{Y}^d )/\mathbb{R} )}}\overline{I} [u,w ]\leq \overline{I}
[\widehat{u}_0,\widehat{u}_1 ].
\end{equation}
Let $\widehat{I}$ be given in Problem \ref{TheLimitProblem}. Then, by the definition of $\widehat{u}_0$ and $\widehat{u}_1$, we
obtain
\[\widehat I[\widehat u_0]
= \overline{I}
[\widehat{u}_0,\widehat{u}_1 ],\]
and
\begin{align}
\label{MinimizingOtherSideEq1}
\begin{split}
\overline{I}
[\widehat{u}_0,\widehat{u}_1 ]&=\int_{\mathbb{T}^d}\int_{\mathcal{Y}^d}e^{ \frac{ |P+\nabla
\widehat{u}_0 (x )+\nabla_y \widehat{u}_1 (x,y ) |^2}{2}+V (x,y
)}dydx\\
&= \inf_{{u}\in {W}^{1,p}  (\mathbb{T}^d )} \int_{\mathbb{T}^d}
\inf_{{w}\in {L}^{p} (\mathbb{T}^d;{W}^{1,p}_\# (\mathcal{Y}^d
)/\mathbb{R} ) } \int_{\mathcal{Y}^d}e^{\frac{ |P+\nabla u (x
) +\nabla_y{w} (x,y ) |^2}{2}+V (x,y )}dy dx.
\end{split}
\end{align}
For any $u\in W^{1,p} (\mathbb{T}^d )$ and $\widetilde{w}\in
{L}^{p} (\mathbb{T}^d;{W}^{1,p}_\# (\mathcal{Y}^d )/\mathbb{R}
)$, we have 
\begin{align*}
\inf_{{w}\in {L}^{p} (\mathbb{T}^d;{W}^{1,p}_\# (\mathcal{Y}^d
)/\mathbb{R} ) } \int_{\mathcal{Y}^d}e^{\frac{ |P+\nabla u (x
) +\nabla_y{w} (x,y ) |^2}{2}+V (x,y )}dy \leq \int_{\mathcal{Y}^d}e^{\frac{
|P+\nabla u (x ) +\nabla_y \widetilde{w} (x,y ) |^2}{2}+V (x,y
)}dy. 
\end{align*}
Thus, 
\begin{align}
\label{MinimizingOtherSideEq2}
\begin{split}
&\inf_{{u}\in {W}^{1,p}  (\mathbb{T}^d )} \int_{\mathbb{T}^d}
\inf_{{w}\in {L}^{p} (\mathbb{T}^d;{W}^{1,p}_\# (\mathcal{Y}^d
)/\mathbb{R} ) } \int_{\mathcal{Y}^d}e^{\frac{ |P+\nabla u (x
) +\nabla_y{w} (x,y ) |^2}{2}+V (x,y )}dy dx\\
&\quad
\leq  \inf_{\overset{u\in W^{1,p} (\mathbb{T}^d )}{\widetilde{w}\in
L^{p} (\mathbb{T}^d;W^{1,p}_\# (\mathcal{Y}^d )/\mathbb{R} )}}
\int_{\mathbb{T}^d}\int_{\mathcal{Y}^d}e^{ (\frac{ |P+\nabla
u (x )+\nabla_y \widetilde{w} (x,y ) |^2}{2}+V (x,y ) )}dydx.
\end{split}
\end{align}
By \eqref{MinimizingOtherSideEq1}, \eqref{MinimizingOtherSideEq2},
and the definition of $\overline{I}$, we get
\begin{equation*}
\overline{I} [\widehat{u}_0,\widehat{u}_1 ]=\int_{\mathbb{T}^d}\int_{\mathcal{Y}^d}e^{
\frac{ |P+\nabla \widehat{u}_0 (x )+\nabla_y \widehat{u}_1 (x,y
) |^2}{2}+V (x,y  )}dydx\leq \inf_{\overset{u\in W^{1,p} (\mathbb{T}^d
)}{w\in L^{p} (\mathbb{T}^d;W^{1,p}_\# (\mathcal{Y}^d )/\mathbb{R}
)}}\overline{I} [u,w ].
\end{equation*}
Combining the preceding equation with \eqref{MinimizingOneSide},
we conclude that 
\begin{equation*}
\overline{I} [\widehat{u}_0,\widehat{u}_1 ]=\inf_{\overset{u\in
W^{1,p} (\mathbb{T}^d )}{w\in L^{p} (\mathbb{T}^d;W^{1,p}_\#
(\mathcal{Y}^d )/\mathbb{R} )}}\overline{I} [u,w ].
\end{equation*}
Therefore, $(\widehat{u}_0, \widehat{u}_1)$ solves Problem~\ref{TwoScaleMinimization}.
\end{proof}

\section{Two-scale homogenization in Higher dimensions}
\label{TwoScaleHomogeInHigherDim}
Here, we establish the  asymptotic behavior of \eqref{POCMFGEq} in any dimension by proving Theorem~\ref{MainTheorem}.

\begin{proof}[Proof of Theorem~\ref{MainTheorem}] 
Let \((u_\epsilon, m_\epsilon, \overline{H}_\epsilon)\) solve Problem \ref{POCMFG} as in the statement. Meanwhile, let  \(\alpha\in(0,1)\) and $q\in [1,\infty)$, let  $u_0\in C^{0,\alpha}\cap{W}^{1,q} (\mathbb{T}^d
)$ with \(\int_{\Tt^d} u_0dx
=0\), $m\in {L}^1
(\mathbb{T}^d\times\mathcal{Y}^d )$
with \(\int_{\mathbb{T}^d}\int_{\mathcal{Y}^d}
m(x,y) dxdy=1\), $u_1\in {L}^q (\mathbb{T}^d;{W}^{1,q}_\# (\mathcal{Y}^d
)/\mathbb{R } )$, 
and \(\overline H(P)\in\Rr\) be given by Proposition~\ref{UniformlyConvergence}.
Let  
$(\widehat{u}_0, \widehat{u}_1 ) \in  C^\infty (\mathbb{T}^d
)\times C^\infty (\mathbb{T}^d;C^{2,\alpha}_\#
(\mathcal{Y}^d )/\mathbb{R} ) $, with \(\int_{\Tt^d}
\hat u_0dx=0\), be the unique solution to Problem~\ref{TwoScaleMinimization}
provided by Proposition~\ref{ExistenceofTwoScaleHomogenization}.

Using Proposition~\ref{TwoScaleConvLipshitz}, Remark~\ref{rmk:onlscwtsc},
the two-scale convergence \eqref{eq:conv2uepsi},  and the smoothness
of \(V\), we obtain
\begin{equation}
\label{eq:mainthmlb1}
\begin{aligned}
\liminf_{\epsilon\rightarrow 0}I_\epsilon [u_\epsilon ]&=\liminf_{\epsilon\rightarrow
0} \int_{\mathbb{T}^d}\int_{\mathcal{Y}^d}e^{\frac{ |P+\nabla
u_\epsilon (x)|^2}{2}} e^{V (x,\frac{x}{\epsilon} )}dydx\\
& \geq \int_{\mathbb{T}^d}\int_{\mathcal{Y}^d}e^{\frac{
|P+\nabla u_0 (x )+\nabla_yu_1 (x,y ) |^2}{2}+V (x,y )}dydx=\overline{I}
[u_0,u_1 ]\geq\overline{I} [\widehat{u}_0,\widehat{u}_1
], 
\end{aligned}
\end{equation}
where in the last inequality, we used that $(\widehat{u}_0, \widehat{u}_1)$ is the minimizer of $\overline{I}$.
 
Let $\psi_0 \in {C}^\infty (\mathbb{T}^d )$, with \(\int_{\Tt^d}
\psi_0 dx =0\), and $\psi_1\in
{C}^\infty (\mathbb{T}^d;C^{2,\alpha}_\# (\mathcal{Y}^d )/\mathbb{R}
)$. For \(x\in\Tt^d\), set \(\psi_\epsilon (x) = \psi_0(x) +
\epsilon \psi_1(x,\frac x\epsilon)\), \(c_\epsilon= \int_{\Tt^d}
\psi_ \epsilon dx\), and \(\tilde \psi_\epsilon = \psi_\epsilon
-c_\epsilon\). Because $u_\epsilon$  minimizes  $I_\epsilon[\cdot]
$ (see \eqref{VariationalFormula}), we have
\begin{equation*}
\begin{aligned}
\limsup_{\epsilon\to0} I_\epsilon [u_\epsilon  ]&\leq \limsup_{\epsilon\to0} I_\epsilon [\tilde \psi_\epsilon ] = \limsup_{\epsilon\to0} 
\int_{\mathbb{T}^d} e^{\frac{
|P+\nabla \psi_0 (x )+\epsilon \nabla_x \psi_1 (x,\frac x\epsilon)
+\nabla_y \psi_1 (x,\frac x\epsilon) |^2}{2}+V (x,\frac x\epsilon)}dx\\
&= \int_{\mathbb{T}^d}\int_{\mathcal{Y}^d}e^{\frac{ |P+\nabla
\psi_0 (x )+\nabla_y\psi_1 (x,y ) |^2}{2}+V (x,y )}dydx=\overline{I}
[\psi_0,\psi_1 ], 
\end{aligned}
\end{equation*}
where we used Propositions~\ref{CaratheodoryFuncConv} and \ref{LipschitzStronglyTwoScaleConv} and Remark~\ref{rmk:onconttwosc}. Hence,
\begin{align}\label{eq:mainthmlb2}
        \limsup_{\epsilon\rightarrow 0}I_\epsilon [u_\epsilon  ] &\leq \inf_{\overset{\psi_0\in C^\infty  (\mathbb{T}^d ),
\int_{\Tt^d}
\psi_0 dx=0}{\psi_1\in
{C}^\infty (\mathbb{T}^d;C^{2,\alpha}_\# (\mathcal{Y}^d )/\mathbb{R}
) }} \overline{I} [\psi_0,\psi_1 ]\leq\overline{I} [\widehat{u}_0,\widehat{u}_1
].
        \end{align}
By \eqref{eq:mainthmlb1} and \eqref{eq:mainthmlb2}, we conclude
that
\begin{equation}\label{eq:limIespitomin}
\begin{aligned}
\lim_{\epsilon\rightarrow
0}I_\epsilon [u_\epsilon  ] = \overline{I} [\widehat{u}_0,\widehat{u}_1
]= \overline{I} [u_0,u_1 ] =\min_{\overset{u\in
W^{1,p} (\mathbb{T}^d ), \int_{\Tt^d}
 udx=0}{w\in L^{p} (\mathbb{T}^d;W^{1,p}_\# (\mathcal{Y}^d )/\mathbb{R}
)}}\overline{I} [u,w ].
\end{aligned}
\end{equation}
Using the strict convexity of the map \(\xi\in\Rr^d\mapsto e^{\frac{|P+\xi|^2}{2}}\),
we conclude that the previous identities yield  $u_0=\widehat{u}_0$ and $u_1=\widehat{u}_1$.
Hence, invoking Proposition~\ref{ExistenceofTwoScaleHomogenization} once more, we conclude that \eqref{OneHomoToTwoHomo} holds. Moreover, recalling  \eqref{DefHEpsilon}
and \eqref{eq:convH}, we have
\begin{equation}
        \label{DefHOverline}
        \overline{H}(P)=\ln \overline{I} [u_0,u_1 ].
        \end{equation}
For  $ (x,y )\in\mathbb{T}^d\times\mathcal{Y}^d$,
 define 
        \begin{align}
        \label{widetildemformula}
        \widetilde{m} (x,y )=e^{\frac{ |P+\nabla u_0 (x )+\nabla_y
u_1 (x,y ) |}{2}+V (x,y )-\overline{H}(P)}.
        \end{align}        
Let  $\phi\in {C}^\infty (\mathbb{T}^d; C^\infty_\#( \mathcal{Y}^d)
)$ be such that $\phi\geq0$. By \eqref{eq:defmepsi}, \eqref{eq:conv2mepsi},
\eqref{eq:conv2uepsi}, \eqref{eq:convH}, 
Proposition~\ref{TwoScaleConvLipshitz}, Remark~\ref{rmk:onlscwtsc},  
the
smoothness of $\phi$ and $V$, and \eqref{widetildemformula}, we obtain
\begin{equation*}
\begin{aligned}
\int_{\mathbb{T}^d}\int_{\mathcal{Y}^d}m
(x,y )\phi (x,y )dydx &= \lim_{\epsilon\rightarrow 0}\int_{\mathbb{T}^d}m_\epsilon
(x ) \phi\Big (x,\frac{x}{\epsilon} \Big)dx \\ &=\lim_{\epsilon\rightarrow
0}\int_{\mathbb{T}^d}e^{\frac{ |P+\nabla u_\epsilon (x ) |^2}{2}} e^{V
(x,\frac{x}{\epsilon} )-\overline{H}_\epsilon(P)}\phi\Big (x,\frac{x}{\epsilon}
\Big)dx\\
 &\geq \int_{\mathbb{T}^d}\int_{\mathcal{Y}^d}e^{\frac{ |P+\nabla
u_0 (x )+\nabla_y u_1 (x,y ) |^2}{2}} e^{V (x,y )-\overline{H}(P)}\phi
(x,y )dydx\\
&=  \int_{\mathbb{T}^d}\int_{\mathcal{Y}^d}\widetilde{m}
(x,y )\phi (x,y )dydx. 
\end{aligned}
\end{equation*}
Thus,       
         \begin{align*}
        \int_{\mathbb{T}^d}\int_{\mathcal{Y}^d}m (x,y )\phi (x,y )dydx\geq \int_{\mathbb{T}^d}\int_{\mathcal{Y}^d}\widetilde{m} (x,y )\phi (x,y )dydx.
        \end{align*}
        Because $\phi$ is an arbitrary nonnegative,  smooth function, we have $m\geq \widetilde{m}$ almost everywhere. By Proposition~\ref{UniformlyConvergence},
  $\int_{\mathbb{T}^d}\int_{\mathcal{Y}^d}m (x,y )dydx=1$. Meanwhile, by \eqref{DefHOverline} and \eqref{widetildemformula}, also $\int_{\mathbb{T}^d}\int_{\mathcal{Y}^d}\widetilde{m} (x,y )dydx=1$.   Hence,  %
\begin{equation*}
\begin{aligned}
\int_{\mathbb{T}^d}\int_{\mathcal{Y}^d} (m (x,y )-\widetilde{m} (x,y ) )dydx=0,
\end{aligned}
\end{equation*}
        which, together with $m-\widetilde{m}\geq 0$, yields $\widetilde{m}=m$ almost everywhere. 
Finally, in view of \eqref{eq:limIespitomin},\eqref{DefHOverline},
and \eqref{widetildemformula},  we conclude that $ (u_0, u_1, m, \overline{H} )$ solves \eqref{TheoremTwoScaleSystemEq}. 

We conclude by observing that the uniqueness of solution to Problems~\ref{TwoScaleHomogenized},
\ref{TwoScaleMinimization}, \ref{Homogenized}, and \ref{TheLimitProblem}  guarantees that
the convergences in Proposition~\ref{UniformlyConvergence} hold
for the whole sequence (and not just for a subsequence). 
\end{proof}

%

\end{document}